\documentclass[oneside,english,11pt]{amsart}
\usepackage[T1]{fontenc}
\usepackage[latin9]{inputenc}
\usepackage{amstext}
\usepackage{amsthm}
\usepackage{amssymb}
\usepackage{esint}
\usepackage[citestyle=alphabetic,bibstyle=alphabetic,doi=false,isbn=false,url=false]{biblatex}
\usepackage[margin=1in]{geometry}
\allowdisplaybreaks

\makeatletter
\numberwithin{equation}{section}
\numberwithin{figure}{section}
\theoremstyle{plain}
\newtheorem{thm}{Theorem}[section]
\theoremstyle{definition}
\newtheorem{defn}[thm]{Definition}
\theoremstyle{plain}
\newtheorem{lem}[thm]{Lemma}
\theoremstyle{plain}

\theoremstyle{plain}
\newtheorem{prop}[thm]{Proposition}
\theoremstyle{remark}
\newtheorem{rem}[thm]{Remark}

\makeatother

\usepackage{babel}

\begin{document}
\title{Tangent Flows of K\"ahler Metric Flows}
\author{Max Hallgren and Wangjian Jian}
\maketitle

\begin{abstract} We improve the description of $\mathbb{F}$-limits of noncollapsed Ricci flows in the K\"ahler setting. In particular, the singular strata $\mathcal{S}^k$ of such metric flows satisfy $\mathcal{S}^{2j}=\mathcal{S}^{2j+1}$. We also prove an analogous result for quantitative strata, and show that any tangent flow admits a nontrivial one-parameter action by isometries, which is locally free on the cone link in the static case. The main results are established using parabolic regularizations of conjugate heat kernel potential functions based at almost-selfsimilar points, which may be of independent interest.
\end{abstract}

\section{Introduction}

Suppose $(M_{i}^{2n},g_{i},p_{i})$ is a sequence of pointed, complete
Riemannian manifolds satisfying 
\begin{align} \label{riclowerbound} 
Rc(g_{i})&\geq-(n-1)g_{i},\\
\text{Vol}(B(p_{i},1))&\geq\nu , \label{vollowerbound} \end{align}
where $\nu >0$. Assume moreover that the sequence $(M_{i},d_{g_{i}},p_{i})$ converges
in the pointed Gromov-Hausdorff sense to a metric space $(X,d,p_{\infty})$. It was shown in \cite{cheegercolding1} that any tangent cone based at a point in $X$ is a metric cone. The singular strata $\mathcal{S}^k$ were defined to be set of $x\in X$ such that no tangent cone based at $x$ is isometric to $C(Z)\times \mathbb{R}^{k+1}$, where $Z$ is any compact metric space with diameter at most $\pi$. Moreover, the Hausdorff dimension estimates $\text{dim}_{\mathcal{H}}(\mathcal{S}^k)\leq k$ were established. 

It is natural to ask what additional properties $X$ satisfies if $(M_i,g_i)$ are assumed to be K\"ahler. Theorem 9.1 of \cite{cheegercoldingtian} showed that in this case $\mathcal{S}^{2j}=\mathcal{S}^{2j+1}$
for $j=0,...,n-1$; roughly speaking, if a tangent cone $C(Y)$ of some $x\in X$ splits a factor of $\mathbb{R}^{2j+1}$, then it actually splits a factor of $\mathbb{R}^{2j+2}$. It was also shown in \cite{gangliucone} that any tangent cone $C(Y)$ of $X$ admits a 1-parameter action by isometries, which extends to an effective isometric action by a torus.  This action is used in \cite{szek2} to obtain an embedding $C(Y)\hookrightarrow \mathbb{C}^N$ whose image is a normal affine algebraic variety, such that the action on $C(Y)$ is the restriction of a linear torus action on $\mathbb{C}^N$. 

The goal of this paper is to prove analogous results in the setting of Ricci flow. A Ricci flow analogue of the singular stratification was first introduced for Ricci flows satisfying a Type-I curvature assumption in \cite{giansingset}. A version defined for general Ricci flows was later defined and studied in \cite{bamlergen3}. To state our results more precisely, we let $(M_{i}^{2n},(g_{i,t})_{t\in(-T_{i},0]})$
be any sequence of Ricci flows equipped with conjugate heat kernel measures $(\nu_{x_i,0;t})_{t\in (-T_i,0]}$ based at $(x_i,0)$, where $T_{\infty}:=\lim_{i\to\infty}T_{i}\in(0,\infty]$
and $\mathcal{N}_{x_{i},0}(1)\geq-Y$ for some $Y<\infty$. In \cite{bamlergen1,bamlergen2,bamlergen3}, Bamler establishes a Ricci flow version of Gromov's compactness theorem and Cheeger-Colding theory, where $\mathbb{F}$-convergence takes the role of pointed Gromov-Hausdorff, and the volume noncollapsing assumption is replaced by the assumed lower bound for Nash entropy. We will review related definitions in Section 2. After passing to a subsequence, we can suppose that 
\[
(M_{i}^{2n},(g_{i,t})_{t\in(-T_{i},0]},(\nu_{x_{i},0;t})_{t\in(-T_{i},0]})\xrightarrow[i\to\infty]{\mathbb{F}}(\mathcal{X},(\nu_{x_{\infty};t})_{t\in(-T_{\infty},0]})
\]
uniformly on compact time intervals, for some future-continuous metric flow
$\mathcal{X}$ of full support over $(T_{\infty},0]$. There is  stratification $\mathcal{S}^0 \subseteq \cdots \subseteq \mathcal{S}^{2n-2} = \mathcal{S}$ of the singular set $\mathcal{S}$ of $\mathcal{X}$ analogous to that of Ricci limit spaces (see Section 2 for details). Our first main theorem can then be stated as follows.

\begin{thm} \label{mainthm1} If $(M_i,(g_{i,t})_{t\in (-T_i ,0]})$ are K\"ahler-Ricci flows, then $\mathcal{S}^{2j+1}=\mathcal{S}^{2j}$ for $j=0,...,n-1$. \end{thm}

For applications to smooth Ricci flows, it is often useful to study the quantitative stratification. A similar stratification was 
studied for Riemannian manifolds satisfying (\ref{riclowerbound}),(\ref{vollowerbound}) in \cite{cheegernaberquant}, where it was used to prove $L^p$ estimates for the Riemannian curvature tensor of Einstein manifolds. The Ricci flow version was again first studied for Type-I Ricci flows \cite{gianquant}, while two different but related definitions for general Ricci flows were used in \cite{bamlergen3}. These are the quantitative strata $\mathcal{S}_{r_1,r_2}^{\epsilon,k}$ and the weak quantitative strata $\widehat{\mathcal{S}}_{r_1,r_2}^{\epsilon,k}$.  Our next result is a quantitative form of Theorem \ref{mainthm1}, from which Theorem \ref{mainthm1} is easily derived. We note that our definition of quantitative strata $\mathcal{S}_{r_1,r_2}^{\epsilon,k}$ is slightly more restrictive than the definition given in \cite{bamlergen3} (see Section 2).

\begin{thm} \label{mainthm2} For any $\epsilon >0$ and $Y,A<\infty$, there exists $\delta =\delta(\epsilon, Y,A)>0$ such that for all $r_2 > r_1 \geq 0$ and $j\in \{0,...,n-1\}$, we have
$$\mathcal{S}_{r_1,r_2}^{\epsilon,2j+1} \cap P^{\ast}(x_{\infty};A,-A^2) \subseteq \mathcal{S}_{r_1,r_2}^{\delta,2j}\cap P^{\ast}(x_{\infty};A,-A^2),$$
$$\widehat{\mathcal{S}}_{r_1,r_2}^{\epsilon,2j+1} \cap P^{\ast}(x_{\infty};A,-A^2) \subseteq \widehat{\mathcal{S}}_{r_1,r_2}^{\delta,2j} \cap P^{\ast}(x_{\infty};A,-A^2).$$
\end{thm}

We observe that unlike Theorem \ref{mainthm1}, Theorem \ref{mainthm2} applies directly to smooth K\"ahler-Ricci flows, since the quantitative strata are generally nonempty even for smooth flows. 

Next, we consider a fixed point $x_0 \in \mathcal{X}_{t_0}$ with $t_0 <0$, and consider a sequence of parabolic rescalings $$(\mathcal{X}^{-t_0,\lambda_k},(\nu_{x_0;t}^{-t_0,\lambda_k})_{t\in [-\lambda_k^2 (t_0-T_{\infty}),0])}),$$ where $\lambda_k \nearrow \infty$. After passing to a subsequence, we can assume $\mathbb{F}$-convergence
$$(\mathcal{X}^{-t_0,\lambda_k},(\nu_{x_0;t}^{-t_0,\lambda_k})_{t\in [-\lambda_k^2 (t_0-T_{\infty}),0])}) \xrightarrow[i\to \infty]{\mathbb{F}} (\mathcal{Y},(\nu_{y_{\infty};t})_{t\in (-\infty,0]}),$$
where $\mathcal{X}^{-t_0,\lambda_k}$ is the metric flow $\mathcal{X}$ with a time translation by $-t_0$ and a parabolic rescaling by $\lambda_k$, and $\mathcal{Y}$ is a metric soliton modeled on a singular shrinking K\"ahler-Ricci soliton $(Y,d_Y,\mathcal{R}_Y,g_Y,f)$ with singularities of codimension four (see Section 2). We now give an analogue of Liu's construction \cite{gangliucone} of a 1-parameter isometric action in the setting of these singular solitons.

\begin{thm} \label{mainthm3} $(Y,d_Y)$ admits a nontrivial 1-parameter action by isometries $(\sigma_s)_{s\in \mathbb{R}}$ which preserves $\mathcal{R}_Y$. Moreover, the infinitesimal generator of the restriction to $\mathcal{R}_Y$ is $J\nabla f$. If in addition $Rc(g_Y)=0$, then $Y$ is a metric cone (by \cite{bamlergen3}), and the action restricted to the cone link is locally free.
\end{thm}

Taking the closure of this 1-parameter subgroup would induce a faithful action of a torus on $Y$, if we knew that the isometry group of $(Y,d_Y)$ is a Lie group -- this holds for Ricci limit spaces by \cite{cheegercolding2,coldingnaber}, and it is likely the arguments can be extended to certain $\mathbb{F}$-limits of Ricci flows (c.f. Remark 2.7 of \cite{bamlergen3}). However, it is currently uncertain whether the arguments of \cite{szek2} can be adapted to show that every tangent cone of a K\"ahler-Ricci flow is an affine variety. This is because the proof of Lemma 2.2 in \cite{szek2} relied on sharp estimates (see \cite{jiangnaberl2}) for the size of singular sets of Gromov-Hausdorff limits of manifolds satisfying (\ref{vollowerbound}) and a two-sided Ricci curvature bound. The analogous estimates for $\mathbb{F}$-limits of noncollapsed Ricci flows are so far unavailable.

The basic idea for proving Theorems \ref{mainthm1} and \ref{mainthm2} is similar to that in the case of Ricci-limit spaces \cite{cheegercoldingtian}. We first consider the model case where $(M^{2n},g,J,f)$ is a smooth, complete gradient K\"ahler-Ricci soliton which isometrically splits a factor of $\mathbb{R}$: $(M,g,f)=(M'\times \mathbb{R},g'+dy^2,f'+\frac{y^2}{4})$, where $(M',g',f')$ is a gradient Ricci soliton of dimension $2n-1$. Then $z:=2 \langle \nabla f , J\nabla y\rangle$ satisfies 
$$\nabla z =J\nabla y -2 Rc(J\nabla y) = J\nabla y,$$
since $Rc(J\nabla y)=JRc(\nabla y)=0$. In particular, $\nabla z$ is parallel and pointwise orthogonal to $\nabla y$. It follows that $z$ is a coordinate for another factor of $\mathbb{R}$ split by $(M,g)$.

We now outline the steps we will take to implement this idea in the singular setting:
\begin{itemize}
    \item We show that any point $y\in \mathcal{X}$ close in $\mathbb{F}$-distance to a metric soliton which either splits $\mathbb{R}^{2k+1}$ or is static and splits a factor of $\mathbb{R}^{2k-1}$ is well-approximated by sequences of points $(y_i,t_i)\in M_i \times (-T_i,0]$ which are $(\epsilon'(\epsilon),r_i)$-selfsimilar for some $r_i \in [r_1,r_2]$, and either $(2k+1,\epsilon'(\epsilon),r_i)$-split, or else $(\epsilon'(\epsilon),r_i)$-static and $(2k-1,\epsilon'(\epsilon),r_i)$-split, where $\lim_{\epsilon\to 0}\epsilon'(\epsilon)=0$. The converse was shown in \cite{bamlergen3}, so it suffices to consider only the weak quantitative strata.
    \item We construct a parabolic regularization $q$ associated to the conjugate heat kernel based at any almost-selfsimilar point $(x_0,t_0)$, which satisfies estimates similar to $4\tau(f-W)$, where $f$ is the potential function for a shrinking GRS.  
    \item Given a strong $(2k+1,\epsilon,r)$-splitting map $y=(y_1,...,y_{2k+1})$ based at an almost-selfsimilar point $(x_0,t_0)$, we show that the functions $$z_i := \frac{1}{2}\langle \nabla q, J\nabla y_i\rangle$$ are almost-splitting functions whose gradients are almost-orthogonal to those of $y_i$, and use appropriate linear combinations of these functions and $y_i$ to conclude that $(x_0,t_0)$ is $(2k+2,\epsilon'(\epsilon),r)$-split. 
    \item The previous step gives $\widehat{\mathcal{S}}_{r_1,r_2}^{\epsilon'(\epsilon),2k+1}\subseteq \widehat{\mathcal{S}}_{r_1,r_2}^{\epsilon,2k}$, hence
    $$\mathcal{S}^{2k+1}=\cup_{\epsilon \in (0,1)} \widehat{\mathcal{S}}_{0,\epsilon}^{\epsilon'(\epsilon),2k+1}\subseteq \cup_{\epsilon \in (0,1)} \widehat{\mathcal{S}}_{0,\epsilon}^{\epsilon,2k}=\mathcal{S}^{2k}.$$
\end{itemize}

In Section 2, we review definitions from \cite{bamlergen2,bamlergen3} relevant to our methods and results. In Section 3, we show that if a limiting metric soliton isometrically splits a factor of $\mathbb{R}^k$, this can be used to find almost-split points in the approximating Ricci flows. In Section 4, we construct a parabolic regularization of approximate Ricci soliton potentials. In Section 5, we use these regularizations to construct almost-splitting maps on K\"ahler-Ricci flows, and finish the proof of Theorems \ref{mainthm1} and \ref{mainthm2}. In Section 6, we construct isometric actions on tangent flows, and prove Theorem \ref{mainthm3}.

\subsection{Acknowledgements}

The authors would like to thank Xiaodong Cao and Jian Song for helpful discussions.

\section{Preliminaries and Notation}

Throughout this paper, we use the following notation convention of Cheeger-Colding theory (c.f. \cite{cheegercoldingwarped}): we let $\Psi(a_1 ,...,a_k|b_1,...,b_{\ell})$ denote a quantity depending on parameters $a_1,...,a_k,b_1,...,b_{\ell}$, which satisfies
$$\lim_{(a_1,...,a_k)\to (0,...,0)} \Psi(a_1,...,a_k|b_1,...,b_{\ell})=0$$
for any fixed $b_1,...,b_{\ell}$. We also adhere to the following convention in \cite{bamlerkleiner}: if we say that a proposition $P(\epsilon)$ depending on a parameter $\epsilon$ holds if $\epsilon \leq \overline{\epsilon}(b_1,...,b_{\ell})$, this means there exists a constant $\overline{\epsilon}$ depending on parameters $b_1,...,b_{\ell}$ such that $P(\epsilon)$ holds whenever $\epsilon \in (0,\overline{\epsilon}]$. The notation $E \geq \underline{E}(b_1,...,b_{\ell})$ is defined analogously. We also let $\mathcal{P}(X)$ denote the space of Borel probability measures on a metric space $X$.

The parabolic analogue of a metric space, defined in \cite{bamlergen2}, is that of a metric flow; metric flow pairs play the role of pointed metric spaces. The definition of a metric flow relies on an auxiliary function $\Phi(x):=\int_{-\infty}^x \frac{1}{\sqrt{4\pi}}e^{-\frac{y^2}{4}}dy$. 

\begin{defn}[Metric Flow Pairs, Definitions 3.2, 5.1 in \cite{bamlergen2}] A metric flow over $I\subseteq \mathbb{R}$ is a tuple $(\mathcal{X},\mathfrak{t},(d_t)_{t\in I},(\nu_{x;s})_{x\in \mathcal{X},s \in I\cap (-\infty,\mathfrak{t}(x)]})$, where $\mathcal{X}$ is a set, $\mathfrak{t}:\mathcal{X}\to I$ is a function, $d_t$ are metrics on the level sets $\mathcal{X}_t:=\mathfrak{t}^{-1}(t)$, and $\nu_{x;s}\in \mathcal{P}(\mathcal{X}_s)$, $s\leq \mathfrak{t}(x)$ are such that $\nu_{x;\mathfrak{t}(x)}=\delta_x$ and the following hold:\\
$(i)$ (Gradient estimate for heat flows) For $s,t\in I$, $s<t$, $T\geq 0$, if $u_s:\mathcal{X}_s\to [0,1]$ is such that $\Phi^{-1}\circ u_s$ is $T^{-\frac{1}{2}}$-Lipschitz (or just measurable if $T=0$), then either $u_t:\mathcal{X}_t\to [0,1]$, $x\mapsto \int_{\mathcal{X}_s}u_s d\nu_{x;s}$, is constant or $\Phi^{-1}\circ u_t$ is $(T+t-s)^{-\frac{1}{2}}$-Lipschitz,\\
$(ii)$ (Reproduction formula) For $t_1 \leq t_2 \leq t_3$ in $I$, $\nu_{x;t_1}(E)=\int_{\mathcal{X}_{t_2}}\nu_{y;t_1}(E)d\nu_{x;t_2}(y)$ for $x\in \mathcal{X}_{t_3}$ and all Borel sets $E\subseteq \mathcal{X}_{t_1}$.\\
A conjugate heat flow on $\mathcal{X}$ is a family $\mu_t \in \mathcal{P}(\mathcal{X}_t)$, $t\in I'$, such that for $s\leq t$ in $I'$, we have $\mu_s(E)=\int_{\mathcal{X}_t} \nu_{x;s}(E) d\mu_t(x)$ for any Borel subset $E\subseteq \mathcal{X}_s$. A metric flow pair $(\mathcal{X},(\mu_t)_{t\in I'})$ consists of a metric flow $\mathcal{X}$, along with a conjugate heat flow $(\mu_t)_{t\in I'}$ such that $\text{supp}(\mu_t)=\mathcal{X}_t$ and $|I\setminus I'|=0$.
\end{defn}

The parabolic analogue of pointed Gromov-Hausdorff convergence is replaced with $\mathbb{F}$-convergence, also introduced in \cite{bamlergen2}.

\begin{defn}[Correspondences and $\mathbb{F}$-Distance, Definitions 5.4, 5.6 in \cite{bamlergen2}]
Given metric flows $(\mathcal{X}^i)_{i\in \mathcal{I}}$ defined over $I'^{,i}$, a correspondence over $I''\subseteq \mathbb{R}$ is a pair 
$$\mathfrak{C}=\left( (Z_t,d_t)_{t\in I''},(\varphi_t^i)_{t\in I''^{,i},i\in \mathcal{I}}\right)$$
where $(Z_t,d_t^Z)$ are metric spaces, $I''^{,i}\subseteq I'^{,i}\cap I''$, and $\varphi_t^i:(\mathcal{X}_t^i,d_t^i)\to (Z_t,d_t^Z)$ are isometric embeddings. The $\mathbb{F}$-distance between metric flow pairs $(\mathcal{X}^j,(\mu_t^j)_{t\in I'^{,j}})$, $j=1,2$, within $\mathfrak{C}$ is the infimum of $r>0$ such that there exists a measurable set $E\subseteq I''$ such that $I''\setminus E\subseteq I''^{,1}\cap I''^{,2}$, $|E|\leq r^2$, and there exist couplings $q_t$ of $(\mu_t^1,\mu_t^2)$, $t\in I''\setminus E$, such that for all $s,t\in I''\setminus E$ with $s\leq t$, we have
$$\int_{\mathcal{X}_t^1 \times \mathcal{X}_t^2}d_{W_1}^{Z_s}\left( (\varphi_s^1)_{\ast}\nu_{x^1;s}^1 , (\varphi_s^2)_{\ast}\nu_{x^2;s}^2 \right) dq_t(x^1,x^2) \leq r.$$
The $\mathbb{F}$-distance between metric flow pairs is the infimum of $\mathbb{F}$-distances within a correspondence $\mathfrak{C}$, where $\mathfrak{C}$ is varied among all correspondences. 
\end{defn}

For the next definition, we suppose $(\mathcal{X}^i,(\mu_t^i)_{t\in I'^{,i}})$ $\mathbb{F}$-converge to $(\mathcal{X}^{\infty},(\mu_t^{\infty})_{t\in I'^{\infty}})$ within the correspondence $\mathfrak{C}$. 
\begin{defn}[Convergence within a correspondence, Definition 6.18 in \cite{bamlergen2}] Given $\mu^i \in \mathcal{P}(\mathcal{X}_{t_i}^i)$ and $\mu^{\infty}\in \mathcal{P}(\mathcal{X}_{t_{\infty}}^{\infty})$, we write $\mu^i \xrightarrow[i\to \infty]{\mathfrak{C}} \mu^{\infty}$ if $t_i \to t_{\infty}$ and there exist $E_i \subseteq I''$ such that $|I''\setminus E_i|\to 0$, $E_i \subseteq I''$ and 
$$\lim_{i\to \infty}\sup_{t\in I''\setminus E} d_W^{Z_t}\left( (\varphi_t^i)_{\ast} \mu_t^i , (\varphi_t^{\infty})_{\ast} \mu_t^{\infty} \right)=0,$$
where $\mu_t^i$ is the conjugate heat flow on $\mathcal{X}^i$ with $\mu_{t_i}^i=\mu^i$, for $i\in \mathbb{N}\cup \{ \infty\}$. We write $x_i \xrightarrow[i \to \infty]{\mathfrak{C}} x_{\infty}$ if $\delta_{x_i} \xrightarrow[i\to \infty]{\mathfrak{C}} \delta_{x_{\infty}}$.
\end{defn}

Next, we recall Kleiner-Lott's notion of a Ricci flow spacetime; it was shown in \cite{bamlergen2,bamlergen3} that the regular part of a metric flow obtained as a limit of Ricci flows possesses this structure.

\begin{defn}[Ricci Flow Spacetime, Definition 1.2 in \cite{klott1}] A Ricci flow spacetime is a tuple $(\mathcal{M},\mathfrak{t},\partial_{\mathfrak{t}},g)$ consisting of a manifold $\mathcal{M}$, a time function $\mathfrak{t}:\mathcal{M}\to \mathbb{R}$, a "time-like" vector field $\partial_{\mathfrak{t}}\in \mathfrak{X}(\mathcal{M})$ with $\partial_{\mathfrak{t}}\mathfrak{t}=1$, and a bundle metric $g$ on the subbundle $\ker(d\mathfrak{t})\subseteq T\mathcal{M}$ satisfying $\mathcal{L}_{\partial_{\mathfrak{t}}}g=-2Rc(g)$, where $Rc(g)|\mathfrak{t}^{-1}(t)$ is defined to be the Ricci curvature of $g|\mathfrak{t}^{-1}(t)$. We write $\mathcal{M}_t := \mathfrak{t}^{-1}(t)$ and $g_t :=g|\mathcal{M}_t$.
\end{defn}

Given a Ricci flow $(M,(g_t)_{t\in I})$ and some $(x,t)\in M\times I$, we let $K(x,t;\cdot,\cdot):M\times (I\cap (-\infty,t))\to(0,\infty)$ denote the conjugate heat kernel based at $(x,t)$, and define $d\nu_{x,t;s}:=K(x,t;\cdot,s)dg_s \in \mathcal{P}(M)$. We now summarize some of the main points of Bamler's weak compactness and partial regularity theory. The notation in this statement will be used throughout the remainder of the paper.

\begin{thm}[c.f. Theorems 7.6, 9.12, 9.31 in \cite{bamlergen2}] \label{bamconvergence} Suppose $(M_i^n,(g_{i,t})_{t\in (-T_i,0]},(x_i,0))$ is a sequence of pointed Ricci flows satisfying $\mathcal{N}_{x_i,0}(1)\geq -Y$ for some $Y<\infty$. Then we can pass to a subsequence to obtain a future-continuous metric flow pair $(\mathcal{X},(\nu_{x_{\infty},t})_{t\in (-T,0]})$ along with a correspondence $\mathfrak{C}$ such that we have the following $\mathbb{F}$-convergence within the correspondence on compact time intervals:
$$(M_i^n,(g_{i,t})_{t\in (-T_i,0]},(\nu_{x_i,0;t})_{t\in (-T_i,0]})\xrightarrow[i\to \infty]{\mathbb{F},\mathfrak{C}} (\mathcal{X},(\nu_{x_{\infty},t})_{t\in (-T,0]}).$$
Moreover, there is an open, dense subset $\mathcal{R}\subseteq \mathcal{X}$ (with respect to the natural topology defined in Section 3 of \cite{bamlergen2}) which admits the structure of a Ricci flow spacetime $(\mathcal{R},\mathfrak{t},\partial_{\mathfrak{t}},g)$, where $\mathfrak{t}$ is the restricted function from the metric flow structure, and each $(\mathcal{X}_t,d_t)$ is the completion of the Riemannian length metric on $(\mathcal{R}_t,d_{g_t})$. In addition, the subbundle $\ker(d\mathfrak{t})\subseteq T\mathcal{R}$ admits an endomorphism $J$ satisfying $\mathcal{L}_{\partial_{\mathfrak{t}}}J=0$ and restricting to an almost-complex structure $J_t$ on each $\mathcal{R}_t$ such that each $(\mathcal{R}_t,g_t,J_t)$ a K\"ahler manifold, and there is an increasing exhaustion $(U_i)$ of $\mathcal{R}$ by precompact open sets along with time-preserving diffeomorphisms $\psi_i :U_i \to M_i$ such that the following hold:\\
$(i)$ $\psi_i^{\ast}g_i \to g$ in $C_{loc}^{\infty}(\mathcal{R})$,\\
$(ii)$ $(\psi_i^{-1})_{\ast}\partial_t \to \partial_{\mathfrak{t}}$ in $C_{loc}^{\infty}(\mathcal{R})$,\\
$(iii)$ If $J_i \in \text{End}(TM_i)$ denote the given complex structures, then $\psi_i^{\ast}J_i \to J$ in $C_{loc}^{\infty}(\mathcal{R})$,\\
$(iv)$ If we write $d\nu_{x_{\infty},t} = v_t dg_t$ on $\mathcal{R}$, then $\psi_i^{\ast}K(x_i,0;\cdot,\cdot)\to v$ in $C_{loc}^{\infty}(\mathcal{R})$.
\end{thm}
\begin{proof} By the mentioned theorems in \cite{bamlergen2,bamlergen3}, it suffices to verify the claims concerning the complex structures. Because $|J_i|_{g_{i,t}}=\sqrt{2n}$ and $\nabla J_i =0$, the Arzela-Ascoli theorem lets us pass to a subsequence so that $\psi_i^{\ast}J_i \to J$, where $J$ restricts to an almost-complex structure on $T\mathcal{R}_t$ for each $t\in (-T,0]$. Moreover, if $\omega_{i,t} \in \Omega^2(M_i)$ denote the K\"ahler forms of $(M_i,g_{i,t})$, then $\psi_i^{\ast}\omega_{i,t} \to \omega_t$, where $\omega_t(\cdot,\cdot):=g_t(J\cdot,\cdot)$. Then $d\omega_{i,t}=0$ and $\partial_t J_i=0$ pass to the limit to give $d\omega_t=0$ and $\mathcal{L}_{\partial_{\mathfrak{t}}}J=0$, so $(\mathcal{R}_t,g_t,J_t)$ is K\"ahler, where $J_t := J|T\mathcal{R}_t$.
\end{proof}

Next, we will recall Bamler's description of the infinitesimal structure of the metric flows $\mathcal{X}$ obtained as $\mathbb{F}$-limits of closed Ricci flows as in Theorem \ref{bamconvergence}.

\begin{defn}[Singular Solitons, Definition 2.15 in \cite{bamlergen3}] A singular space is a tuple $(X,d,\mathcal{R}_X,g_X)$, where $(X,d)$ is a complete, locally compact metric length space, $\mathcal{R}_X\subseteq X$ is a dense open subset admiting a smooth structure so that $(\mathcal{R}_X,g_X)$ is a smooth Riemannian manifold whose length metric is $d|(\mathcal{R}_X\times \mathcal{R}_X)$, and such that for any compact subset $K\subseteq X$ and $D<\infty$, there exist $0<\kappa_1(K,D)<\kappa_2(K,D)<\infty$ such that for all $x\in K$ and $r\in (0,D)$, we have
$$\kappa_1 r^n <|B(x,r)\cap \mathcal{R}_X|<\kappa_2 r^n.$$
A singular shrinking gradient Ricci soliton (GRS) consists of a singular space along with a function $f\in C^{\infty}(\mathcal{R})$ satisfying the Ricci soliton equation on $\mathcal{R}_X$:
$$Rc(g_X)+\nabla^2 f=\frac{1}{2}g_X.$$
\end{defn}

In the following, we will not directly use the precise definition of a metric soliton or static flow modeled on a metric space, but the interested reader may find these definitions in Section 3.8 of \cite{bamlergen2}.

\begin{thm}[Theorems 2.6, 2.16, 2.18 in \cite{bamlergen3}] \label{tangentflows} If $\mathcal{X}$ is a metric flow obtained as in Theorem \ref{bamconvergence}, and $y\in \mathcal{X}$, $t_0 :=\mathfrak{t}(x)$, then for any sequence $\lambda_k \nearrow \infty$, we can pass to a further subsequence so that the time shifted and parabolically rescaled metric flows $(\mathcal{X}^{-t_0,\lambda_k},(\nu_{y;t}^{-t_0,\lambda_k})_{t\in (-\lambda_k^2 (t_0 -T),0]})$ $\mathbb{F}$-converge to a metric flow pair $(\mathcal{Y},(\nu_{y_{\infty};t})_{t\in (-\infty,0]})$, where $(\mathcal{Y}_{<0},(\nu_{y_{\infty};t_{\infty}})_{t\in (-\infty,0)})$ is a metric soliton modeled on a singular shrinking K\"ahler GRS $(Y,d,\mathcal{R}_Y,g_Y,J_Y,f_Y)$. Also, there are diffeomorphisms as in Theorem \ref{bamconvergence} realizing smooth convergence on the regular part of $\mathcal{Y}$. There is an identification $\mathcal{Y}_{<0}\cong Y\times (-\infty,0)$ restricting to isometries $(\mathcal{Y}_t,d_t)\cong (Y,\sqrt{|t|}d)$, and also identifying the spacetime $\mathcal{R}\subseteq \mathcal{Y}$ with $(\mathcal{R}_Y\times (-\infty,0),t,\partial_t -\nabla f_Y,|t|g_Y)$. Writing $d\nu_{y_{\infty};t}=(4\pi \tau)^{-\frac{n}{2}}e^{-f}dg_t$, we have that $f(\cdot,t)$ corresponds to $f_Y$ for all $t<0$ with respect to this identification.

If $Rc(g_Y)=0$, then $\mathcal{Y}_{<0}$ is a static metric flow modeled on the metric cone $(Y,d)$ with vertex $o$. Moreover, in this case, there is an identification $\mathcal{Y}_{<0}\cong Y\times (-\infty,0)$ restricting to isometries $(\mathcal{Y}_t,d_t)\cong (Y,d)$, and identifying the spacetime $\mathcal{R}$ with $(\mathcal{R}_Y \times (-\infty,0),t,\partial_t, g_Y)$; $f(\cdot,t)$ then corresponds to $\frac{1}{4|t|}d^2(o,\cdot)+W_{\infty}$ for each $t<0$. Moreover, $\mathcal{R}_Y \cap \partial B(o,1)$ equipped with the restricted Riemannian metric is a Sasaki-Einstein manifold.
\end{thm}
\begin{proof} By the mentioned theorems in \cite{bamlergen3}, and by Theorem \ref{bamconvergence}, it suffices to recall that a gradient Ricci soliton structure on a K\"ahler manifold is automatically a K\"ahler-Ricci soliton (see section 2.2 of \cite{fik}).
\end{proof}

In order to define Bamler's stratification of the singular set, we must review the notions of almost-split, almost-static, and almost-selfsimilar.

\begin{defn}[Definitions 5.1, 5.5, 5.6, 5.7 in \cite{bamlergen3}] \label{almost} Suppose $(M^n,(g_t)_{t\in I})$ is a closed Ricci flow, $(x_0,t_0)\in M\times I$, $r>0$, $\epsilon \in (0,1)$, and write $d\nu_{x_0,t_0}=(4\pi \tau)^{-\frac{n}{2}}e^{-f}dg$.\\
$(i)$ $(x_0,t_0)$ is $(\epsilon,r)$-selfsimilar if $[t_0-\epsilon^{-1}r^2,t_0]\subseteq I$ and the following hold for $W:=\mathcal{N}_{x,t}(r^2)$:\\
$$\int_{t_0-\epsilon^{-1}r^2}^{t_0-\epsilon r^2} \int_M \tau \left| Rc+\nabla^2 f -\frac{1}{2\tau}g \right|^2 d\nu_{x_0,t_0;t}dt \leq \epsilon ,$$
$$\sup_{t\in [t_0-\epsilon^{-1}r^2,t_0-\epsilon r^2]} \int_M \left| \tau (R+2\Delta f-|\nabla f|^2 )+f-n-W \right|d\nu_{x_0,t_0;t} \leq \epsilon,$$
$$\inf_{M\times [t_0-\epsilon^{-1}r^2,t_0-\epsilon r^2]} r^2 R\geq -\epsilon.$$
$(ii)$ $(x_0,t_0)$ is $(\epsilon,r)$-static if $[t_0-\epsilon^{-1}r^2,t_0]\subseteq I$ and the following hold:\\
$$r^2 \int_{t_0-\epsilon^{-1}r^2}^{t_0-\epsilon r^2} \int_M |Rc|^2 d\nu_{x_0,t_0;t}dt \leq \epsilon,$$
$$\sup_{t\in [t_0-\epsilon^{-1}r^2,t_0 -\epsilon r^2]}\int_M R d\nu_{x_0,t_0;t} \leq \epsilon,$$
$$ \inf_{M\times [t_0-\epsilon^{-1}r^2,t_0-\epsilon r^2]} r^2 R\geq -\epsilon.$$
$(iii)$ $(x_0,t_0)$ is weakly $(k,\epsilon,r)$-split if $[t_0-\epsilon^{-1}r^2,t_0]\subseteq I$ and there is a map $y=(y_1,...,y_k):M\times [t_0-\epsilon^{-1}r^2,t_0-\epsilon r^2] \to \mathbb{R}^k$, called a weak $(k,\epsilon,r)$-splitting map, satisfying the following:
$$r^{-1}\int_{t_0-\epsilon^{-1}r^2}^{t_0 -\epsilon r^2}\int_M |\square y_i| d\nu_{x_0,t_0;t}dt \leq \epsilon,$$
$$r^{-2}\int_{t_0-\epsilon^{-1}r^2}^{t_0 -\epsilon r^2} \int_M |\langle \nabla y_i ,\nabla y_j \rangle -\delta_{ij} |d\nu_{x_0,t_0;t}dt\leq \epsilon .$$
$(iv)$ $(x_0,t_0)$ is strongly $(k,\epsilon,r)$-split if $[t_0-\epsilon^{-1}r^2,t_0]\subseteq I$ and there is a map $y:M\times [t_0-\epsilon^{-1}r^2,t_0-\epsilon r^2] \to \mathbb{R}^k$, called a strong $(k,\epsilon,r)$-splitting map, satisfying the following:\\
$$\square y_i=0 \text{ on } M\times [t_0-\epsilon^{-1}r^2,t_0 -\epsilon r^2],$$
$$r^{-2}\int_{t_0-\epsilon^{-1}r^2}^{t_0 -\epsilon r^2}\int_M |\langle \nabla y_i ,\nabla y_j \rangle -\delta_{ij} |d\nu_{x_0,t_0;t}dt\leq \epsilon ,$$
$$\int_M y_i d\nu_{x_0,t_0;t}=0 \text{ for all } t\in [t_0 -\epsilon^{-1}r^2,t_0-\epsilon r^2].$$
\end{defn}

The following estimate gives rough $L^p$ bounds on various geometric quantities in almost-selfsimilar regions, which we will use frequently.

\begin{prop}[c.f. Proposition 6.2 in \cite{bamlergen3}] Given $\epsilon>0$, if $\alpha \leq \overline{\alpha}$ and $\delta \leq \overline{\delta}(\epsilon)$, then the following holds. Suppose $(M^n,(g_t)_{t\in I})$ is a Ricci flow, $r>0$, $(x_0,t_0)\in M\times I$ is $(\delta,r)$-selfsimilar, $W:=\mathcal{N}_{x_0,t_0}(1)\geq -Y$, and write $(4\pi \tau)^{-\frac{n}{2}}e^{-f}dg:=d\nu :=d\nu_{x_0,t_0}$. Then
\begin{equation} \label{spacetime} \int_{t_0-\epsilon^{-1}r^2}^{t_0-\epsilon r^2} \int_M \left( \tau |Rc|^2 +\tau |\nabla^2 f|^2 +|\nabla f|^2 +\tau |\nabla f|^4 +\tau^{-1}e^{\alpha f} \right) e^{2\alpha f}d\nu_t dt \leq C(Y,\epsilon), 
\end{equation}
\begin{equation} \label{timeslice} \sup_{t\in [t_0-\epsilon^{-1} r^2,t_0-\epsilon r^2]}
\int_M \left( \tau |R| + \tau|\Delta f| +\tau |\nabla f|^2 +e^{\alpha f} \right)e^{2\alpha f}d\nu_t \leq C(Y,\epsilon).
\end{equation}
\end{prop}
\begin{proof} By Proposition 7.1 of \cite{bamlergen3}, we have $\mathcal{N}_{x_0,t_0}(\tau)\geq W-\Psi(\delta|Y,\epsilon)$. We can therefore apply Proposition 6.2 of \cite{bamlergen3} for any $r\in [\epsilon^{\frac{1}{2}},\epsilon^{-\frac{1}{2}}]$, taking $\theta :=\epsilon^2$, to obtain (\ref{spacetime}), (\ref{timeslice}).
\end{proof}

We now review the quantitative stratification of metric flows introduced by Bamler. We will use a slightly stricter definition of $(k,\epsilon,r)$-symmetric points than that in \cite{bamlergen3}. Let $(\mu_t^{\mathbb{R}^k})_{t<0}$ be the Euclidean backwards heat kernel based at $0^k \in \mathbb{R}^k$. 

\begin{defn}[$(k,\epsilon,r)$-symmetric points, c.f. Definition 2.21 in \cite{bamlergen2}] \label{symdef} Given a metric flow $\mathcal{X}$ over $I$, a point $x_0 \in \mathcal{X}_{t_0}$ is $(k,\epsilon,r)$-symmetric if $[t_0 -\epsilon^{-1}r^2,t_0]\subseteq I$ and there is a metric flow pair $(\mathcal{X}',(\mu_t')_{t\leq 0})$ over $(-\infty,0]$ which is an $\mathbb{F}$-limit of noncollapsed Ricci flows as in Theorem \ref{bamconvergence}, and which satisfies one of the following:
\begin{itemize}
    \item[$(b1)$] $(\mathcal{X}'_{<0},(\mu_t')_{t\in (-\infty,0)}) \cong (\mathcal{X}''\times \mathbb{R}^k,(\mu_t ''\otimes \mu_{\mathbb{R}^k})_{t\in (-\infty,0)})$ as metric flow pairs for some metric soliton $(\mathcal{X}'',(\mu_t '')_{t\in (-\infty,0)})$, and this identification restricts to an isometry of Ricci flow spacetimes $\mathcal{R}' \cong \mathcal{R}''\times \mathbb{R}^k$,
    \item[$(b2)$] $(\mathcal{X}'_{<0},(\mu_t')_{t\in (-\infty,0)}) \cong (\mathcal{X}''\times \mathbb{R}^{k-2},(\nu_{x';t} '\otimes \mu_{\mathbb{R}^{k-2}})_{t\in (-\infty,0)})$ as metric flow pairs for some static cone $\mathcal{X}''$ with vertex $x'$, and this identification restricts to an isometry of Ricci flow spacetimes $\mathcal{R}' \cong \mathcal{R}''\times \mathbb{R}^{k-2}$.
\end{itemize}
In addition, $\mathcal{X}',\mathcal{X}''$ must satisfy the following:
\begin{itemize}
    \item[$(c)$] Writing $d\mu_t'= (4\pi \tau)^{-\frac{n}{2}}e^{-f'}dg$ on $\mathcal{R}'$ and $d\mu_t'' = (4\pi \tau)^{-\frac{n}{2}}e^{-f''}dg''$ on $\mathcal{R}'$, we have $Rc(g')+\nabla^2 f' =\frac{1}{2\tau}g'$ on $\mathcal{R}'$, $Rc(g'')+\nabla^2 f'' =\frac{1}{2\tau}g''$ on $\mathcal{R}''$, and $f' =f''+\frac{1}{4\tau}|x|^2$ on $\mathcal{R}'$. In case $(b2)$, we also have $Rc(g')=0$ and $Rc(g'')=0$,
    \item[$(d)$] $\mathcal{N}_{(\mu_{t}')}(\tau)=W$ for all $\tau >0$, where $W \in [-Y,0]$.
\end{itemize}
Finally, we require that $|\mathcal{N}_{x_0}(r^2)-W|<\epsilon$ and
$$d_{\mathbb{F}}\left( (\mathcal{X}_{[-\epsilon^{-1},0]}^{-t_0,r^{-1}},(\nu_{x_0;t}^{-t_0,r^{-1}})_{t\in [-\epsilon^{-1},0]}),(\mathcal{X}_{[-\epsilon^{-1},0]}',(\mu_t')_{t\in [-\epsilon^{-1},0]}) \right)<\epsilon .$$
\end{defn}

The main difference between Definition 2.21 in \cite{bamlergen3} and Definition \ref{symdef} is the added assumptions on the Nash entropy.

There is another notion of almost-symmetric points of a metric flow, defined in terms of smooth Ricci flow approximants.

\begin{defn}[Weakly $(k,\epsilon,r)$-symmetric points, Definition 20.1 in \cite{bamlergen3}] Given a metric flow $\mathcal{X}$ over $I$, a point $x\in \mathcal{X}_{t_0}$ is weakly $(k,\epsilon,r)$-symmetric if $[t_0-\epsilon^{-1}r^2,t_0]\subseteq I$ and there is a closed, pointed Ricci flow $(M',(g_t')_{t\in [-\epsilon^{-1},0]},x')$ such that $(x',0)$ is $(\epsilon,1)$-selfsimilar and either strongly $(k,\epsilon,1)$-split or both $(\epsilon,1)$-static and  strongly $(k-2,\epsilon,1)$-split, which satisfies $|\mathcal{N}_{x',0}(1)-\mathcal{N}_{x_0}(r^2)|\leq \epsilon$ and
$$d_{\mathbb{F}}\left( (\mathcal{X}^{-t_0,r},(\nu_{x_0;t}^{-t_0,r})_{t\in [-\epsilon^{-1},0]}),(M',(g_t')_{t\in [-\epsilon^{-1},0]},(\nu_{x',0;t})_{t\in [-\epsilon^{-1},0]}) \right) <\epsilon.$$
\end{defn}

\begin{defn}[Strata and Quantitative strata of a metric flow, Definitions 2.21, 20.2 in \cite{bamlergen3}] If $\mathcal{X}$ is a metric flow over $I$ and $0\leq r_1 < r_2 <\infty$, $\epsilon>0$, then $\mathcal{S}_{r_1,r_2}^{\epsilon,k}$ consists of the points $x\in \mathcal{X}_t$ such that $[t-\epsilon^{-1}r_2^2,t]\subseteq I$ and $x$ is not $(k+1,\epsilon,r')$-symmetric for any $r'\in (r_1 ,r_2)$. Analogously, $x\in \widehat{\mathcal{S}}_{r_1,r_2}^{\epsilon,k}$ if $[t-\epsilon^{-1}r_2^2,t]\subseteq I$ and $x$ is not weakly $(k+1,\epsilon,r')$-symmetric for any $r'\in (r_1,r_2)$. Define
$$\mathcal{S}^k:=\cup_{\epsilon \in (0,1)}\widehat{\mathcal{S}}_{0,\epsilon}^{\epsilon,k}.$$
\end{defn}
Note that, because $\epsilon \mapsto \widehat{\mathcal{S}}_{0,r}^{\epsilon,k}$ and $r\mapsto \widehat{\mathcal{S}}_{0,r}^{\epsilon,k}$ are both decreasing, we can write
$$\mathcal{S}^k = \cup_{\epsilon \in (0,1)}\cup_{r\in (0,1)}\widehat{\mathcal{S}}_{0,r}^{\epsilon,k}.$$

Bamler showed that, roughly speaking, the weak quantitative strata are qualitatively at least as large as the quantitative strata.

\begin{lem}[c.f. Lemma 20.3 in \cite{bamlergen3}] Suppose $\mathcal{X}$ is an $\mathbb{F}$-limit of closed noncollapsed Ricci flows as in Theorem \ref{bamconvergence}. Given $Y<\infty$, $\epsilon>0$ there exists $\epsilon'(Y,\epsilon)>0$ such that for all $0\leq r_1 <r_2 <\infty$, we have
$$\{x \in \mathcal{X};\mathcal{N}_x(r_2^2)\geq -Y\} \cap \mathcal{S}_{r_1,r_2}^{\epsilon,k} \subseteq \widehat{\mathcal{S}}_{r_1,r_2}^{\epsilon'(Y,\epsilon),k}.$$ 
As a consequence, we have 
$$\cup_{\epsilon \in (0,1)} \mathcal{S}_{0,\epsilon}^{\epsilon,k} \subseteq \mathcal{S}^k.$$
\end{lem}
\begin{proof} Suppose $x\in \mathcal{X}\setminus \widehat{\mathcal{S}}_{r_1,r_2}^{\epsilon'(Y,\epsilon),k}$ and $\mathcal{N}_x(r_2^2)\geq -Y$. By definition, there exists $r\in (r_1,r_2)$ such that $x$ is weakly $(k,\epsilon'(Y,\epsilon),r)$-symmetric. Lemma 20.3 of \cite{bamlergen3} then states that $x$ is $(k,\epsilon,r)$-symmetric, hence $x\notin \mathcal{S}_{r_1,r_2}^{\epsilon,k}$. We note that this Lemma holds even with our stricter definition of $(k,\epsilon ',r)$-symmetric points, by Nash entropy convergence (Theorem 15.45 of \cite{bamlergen3}) and Proposition 7.1 of \cite{bamlergen3}.

Taking $r_1 =0$, $r_2 =\epsilon$, and taking the union over $\epsilon \in (0,1)$ gives
$$\{ x\in \mathcal{X}; \mathcal{N}_x(1)\geq -Y \} \cap \left( \cup _{\epsilon\in (0,1)} \mathcal{S}_{0,\epsilon}^{\epsilon,k} \right) \subseteq \mathcal{S}^k,$$
so the remaining claim follows by taking $Y\nearrow \infty$.
\end{proof}

\begin{rem} We will later show that $\cup_{\epsilon \in (0,1)} \mathcal{S}_{0,\epsilon}^{\epsilon,k} = \mathcal{S}^k$.
\end{rem}

We now recall a result from \cite{bamlergen3} asserting the existence of good cutoff functions vanishing near the singular set of an $\mathbb{F}$-limit of Ricci flows, which will be useful in Section 6.  Assume $\mathcal{X}$ is a metric flow as in Theorem \ref{bamconvergence}.

\begin{lem}[Lemma 15.27 in \cite{bamlergen3}] \label{cutoff} There is a family of smooth functions $\eta_r \in C^{\infty}(\mathcal{R},[0,1])$ satisfying the following:\\
$(i)$ $r_{Rm}\geq r$ on $\{\eta_r>0\}$,\\
$(ii)$ $\eta_r |\{r_{Rm}\geq 2r\}\equiv 1$,\\
$(iii)$ $|\nabla \eta_r |\leq C_0 r^{-1}$, $|\partial_{\mathfrak{t}}\eta_r|\leq C_0 r^{-2}$, where $C_0 <\infty$ is universal\\
$(iv)$ for any $x\in \mathcal{X}_t$, $A<\infty$, and $r>0$, the set $\{ \eta_r >0\}\cap P^{\ast}(x_{\infty};A,-A^2)\cap \mathcal{R}_t$ is relatively compact in $\mathcal{R}_t$,\\
$(v)$ for any maximal integral curve $\gamma:I\to \mathcal{R}$ of $\mathfrak{t}$ (assume $\mathfrak{t}(\gamma(t))=t$ by a constant reparametrization), we have either $\eta_r(\gamma(t))=0$ for $t$ near $t_{\min}:= \inf(I)$, or else $t_{\min}=-T$.
\end{lem}

\section{Sequences of Noncollapsed Ricci Flows whose $\mathbb{F}$-Limits are Split Solitons}

In this section, we show the qualitative equivalence of Bamler's notions of quantitative strata and weak quantitative strata, with one direction already established in \cite{bamlergen3}. We begin with an elementary lemma concerning the 1-Wasserstein distance between product measures.

\begin{lem}
\label{lem:prodwass} Suppose $(X,d^{X}),(Y,d^{Y})$ are metric spaces
and $\mu_{1},\mu_{2}\in\mathcal{P}(X)$, $\nu_{1},\nu_{2}\in\mathcal{P}(Y)$.
Then 
\[ d_{W_{1}}^{X\times Y}\left(\mu_{1}\otimes\nu_{1},\mu_{2}\otimes\nu_{2}\right)\leq d_{W_{1}}^{X}(\mu_{1},\mu_{2})+d_{W_{1}}^{Y}(\nu_{1},\nu_{2}).
\]
\end{lem}

\begin{proof}
Suppose $q_{X}$ is a coupling of $(\mu_{1},\mu_{2})$
and $q_{Y}$ is a coupling of $(\nu_{1},\nu_{2})$. Define $\sigma:X\times X\times Y\times Y\to X\times Y\times X\times Y$,
$(x_{1},x_{2},y_{1},y_{2})\mapsto(x_{1},y_{1},x_{2},y_{2})$, and
$q:=\sigma_{\ast}(q_{X}\otimes q_{Y})$ . Then $q$ is a coupling
of $(\mu_{1}\otimes\nu_{1},\mu_{2}\otimes\nu_{2})$, so we can estimate
\begin{align*}
d_{W_{1}}^{X\times Y}(\mu_{1}\otimes\nu_{1},\mu_{2}\otimes\nu_{2})\leq & \int_{X\times Y\times X\times Y}d^{X\times Y}\left((x_{1},y_{1}),(x_{2},y_{2})\right)dq(x_{1},y_{1},x_{2},y_{2})\\
\leq & \int_{X\times X\times Y\times Y}\left(d^{X}(x_{1},x_{2})+d^{Y}(y_{1},y_{2})\right)dq_{X}(x_{1},x_{2})dq_{Y}(y_{1},y_{2})\\
= & \int_{X\times X}d^{X}(x_{1},x_{2})dq_{X}(x_{1},x_{2})+\int_{Y\times Y}d^{Y}(y_{1},y_{2})dq_{Y}(y_{1},y_{2}).
\end{align*}
Taking the infimum over all such couplings $q_{X},q_{Y}$ gives the
remaining claim. 
\end{proof}

We now show that if a metric soliton splits a factor of $\mathbb{R}^k$, this can be used to extract a sequence of approximating points in smooth Ricci flows which are almost-selfsimilar and almost-split.  This is analogous to the existence of almost-splitting maps in section 2.6 of \cite{cheegercoldingwarped} given Gromov-Hausdorff closeness to a metric product.

\begin{prop} \label{getsplit}
Suppose $(\mathcal{X},(\mu_t)_{t\in(-\infty,0]})$ is
a future continuous metric soliton satisfying the following:
\begin{itemize} 
\item[$(a)$] $(\mathcal{X},(\mu_t)_{t\in(-\infty,0)})$ is an $\mathbb{F}$-limit of $n$-dimensional closed Ricci flows $(M_i,(g_{i,t})_{t\in (-\delta_i^{-1},0]},(\nu_{x_i,0;t})_{t\in (-\delta^{-1},0]})$ as in Theorem \ref{bamconvergence}, with $\mathcal{N}_{x_i,0}(1)\geq -Y$,
\item[$(b)$] $(\mathcal{X},(\mu_t)_{t\in (-\infty,0)}) \cong (\mathcal{X}'\times \mathbb{R}^k,(\mu_t '\otimes \mu_{\mathbb{R}^k})_{t\in (-\infty,0)})$ as metric flow pairs for some metric soliton $(\mathcal{X}',(\mu_t ')_{t\in (-\infty,0)})$, and this identification restricts to an isometry of Ricci flow spacetimes $\mathcal{R} \cong \mathcal{R}'\times \mathbb{R}^k$,
\item[$(c)$] Writing $d\mu_t = (4\pi \tau)^{-\frac{n}{2}}e^{-f}dg$ on $\mathcal{R}$ and $d\mu_t' = (4\pi \tau)^{-\frac{n}{2}}e^{-f'}dg'$ on $\mathcal{R}'$, we have $Rc(g)+\nabla^2 f =\frac{1}{2\tau}g$ on $\mathcal{R}$, $Rc(g')+\nabla^2 f' =\frac{1}{2\tau}g'$ on $\mathcal{R}'$, and $f =f'+\frac{1}{4\tau}|x|^2$ on $\mathcal{R}$.
\item[$(d)$] $\mathcal{N}_{(\mu_{t})}(\tau)=W$ for all $\tau >0$, where $W \in [-Y,0]$.
\end{itemize}
Then for any $\epsilon>0$, $(x_i,0)$ is $(\epsilon,1)$-selfsimilar and $(k,\epsilon,1)$-split for sufficiently large $i=i(\epsilon)\in \mathbb{N}$. If in addition, $Rc(g)=0$, then $(x_i,0)$ is also $(\epsilon,1)$-static for sufficiently large $i\in \mathbb{N}$.
\end{prop}

\begin{proof}
By Nash entropy convergence (Theorem 15.45 of \cite{bamlergen3}), assumption $(d)$ and Proposition 7.1 of \cite{bamlergen3} imply that $(x_i,0)$ is $(\epsilon,1)$-selfsimilar for sufficiently large $i\in \mathbb{N}$. If $Rc(g)=0$, then $(x_i,0)$ is $(\epsilon,1)$-static for large $i \in \mathbb{N}$ by Claim 22.7 of \cite{bamlergen3}.  Suppose the remaining claim fails, so that there exists $\epsilon>0$ such that, after passing to a subsequence, $(x_i,0)$ is not $(k,\epsilon,1)$-split. Fix a correspondence $\mathfrak{C}$ such that 
\begin{equation}
(M_{i}^{n},(g_{i,t})_{t\in(-\delta_{i}^{-1},0]},(\nu_{x_{i},0;t})_{t\in(\delta_{i}^{-1},0]})\xrightarrow[i\to\infty]{\mathbb{F},\mathfrak{C}}(\mathcal{X},(\nu_{x_{\infty};t})_{t\in(-\infty,0]})\label{eq:splitFcon}
\end{equation}
on compact time intervals. By passing to a subsequence, we
can assume the convergence is time-wise for all times in some subset
$I'\subseteq(-\infty,0)$, where $|\mathbb{R}\setminus I'|=0$. 

Choose a sequence $t_{j}\nearrow0$, and recall that 
\[
H_{n}|t|\geq\text{Var}(\mu_{t})=\text{Var}(\mu_{t}')+\text{Var}(\mu_{t}^{\mathbb{R}^{k}})\geq\text{Var}(\mu_{t}'),
\]
so we can find $z_{j}\in\mathcal{X}_{t_{j}}'$ such that $\text{Var}(\delta_{z_{j}},\mu_{t_{j}}')<H_{n}|t_{j}|$.
Then 
\begin{align*}
d_{W_{1}}^{\mathcal{X}_{t_{j}}}(\delta_{(z_{j},0^{k})},\mu_{t_{j}})= & d_{W_{1}}^{\mathcal{X}_{t_{j}}'\times\mathbb{R}^{k}}(\delta_{z_{j}}\otimes\delta_{0^{k}},\mu_{t_{j}}'\otimes\mu_{t_{j}}^{\mathbb{R}^{k}})\leq d_{W_{1}}^{\mathcal{X}_{t_{j}}'}(\delta_{z_{j}},\mu_{t_{j}}')+d_{W_{1}}^{\mathbb{R}^{k}}(\delta_{0^{k}},\mu_{t_{j}}^{\mathbb{R}^{k}})\leq2H_{n}|t_{j}|.
\end{align*}
Fix $\tau>0$. Letting $e^{\alpha}\in\mathbb{R}^{k}$ be the standard
basis vectors, we can find $z_{j,i}^{\alpha},z_{j,i}\in M_{i}$ such that
$(z_{j,i}^{\alpha},t_{0})\xrightarrow[i\to\infty]{\mathfrak{C}}(z_{j},e^{\alpha})$ and $(z_{j,i},t_{j})\xrightarrow[i\to\infty]{\mathfrak{C}}(z_{j},0^{k})$,
so that there are subsets $E_{j,i}\subseteq[-\tau^{-1},0]$ such that
$\lim_{i\to\infty}|E_{j,i}|=0$, $[-\tau^{-1},t_j]\setminus E_{j,i} \subseteq I''^{,j}$ and
\[
\lim_{i\to\infty}\sup_{t\in[-\tau^{-1},t_{j}]\setminus E_{j,i}}d_{W_{1}}^{Z_{t}}\left((\varphi_{t}^{i})_{\ast}\nu_{(z_{j,i}^{\alpha},t_{j});t}^{i},(\varphi_{t}^{\infty})_{\ast}\nu_{(z_{j},e^{\alpha});t}^{i}\right)=0
\]
\[
\lim_{i\to\infty}\sup_{t\in[-\tau^{-1},t_{j}]\setminus E_{j,i}}d_{W_{1}}^{Z_{t}}\left((\varphi_{t}^{i})_{\ast}\nu_{(z_{j,i},t_{j});t}^{i},(\varphi_{t}^{\infty})_{\ast}\nu_{(z_{j},0^{k});t}^{i}\right)=0.
\]
We can pass to further subsequences and use a diagonal argument to
obtain subsets $E_{j}\subseteq[-j,0]$ with $[t_{j},0]\subseteq E_{j}$
such that $|E_{j}|\leq2^{-j}$, $[-\tau^{-1},0]\setminus E_j \subseteq I''^{,j}$, and $z_{j}',z_j'^{,\alpha}\in M_{j}$ such that
\[
\sup_{t\in[-j,0]\setminus E_{j}}d_{W_{1}}^{Z_{t}}\left((\varphi_{t}^{j})_{\ast}\nu_{(z_{j}',t_{j});t}^{j},(\varphi_{t}^{\infty})_{\ast}\nu_{(z_{j},0^{k});t}\right)\leq2^{-j},
\]
\[
\sup_{t\in[-j,0]\setminus E_{j}}d_{W_{1}}^{Z_{t}}\left((\varphi_{t}^{j})_{\ast}\nu_{(z_{j}'^{,\alpha},t_{j});t}^{j},(\varphi_{t}^{\infty})_{\ast}\nu_{(z_{j},e^{\alpha});t}\right)\leq2^{-j}
\]
We now show 
$(\nu_{(z_{j}',t_{j});t}^{j})_{t\in[-j,t_{j}]}\xrightarrow[i\to\infty]{\mathfrak{C}}(\mu_{t})_{t\in(-\infty,0)}$ uniformly
on compact time intervals. In fact 
\begin{align*}
d_{W_{1}}^{Z_{t}}\left((\varphi_{t}^{j})_{\ast}\nu_{(z_{j}',t_{j});t}^{j},(\varphi_{t}^{\infty})_{\ast}\mu_{t}\right)\leq & d_{W_{1}}^{Z_{t}}\left((\varphi_{t}^{j})_{\ast}\nu_{(z_{j}',t_{j});t}^{j},(\varphi_{t}^{\infty})_{\ast}\nu_{(z_{j},0^{k});t}\right)+d_{W_{1}}^{\mathcal{X}_{t}}(\nu_{(z_{j},0^{k});t},\mu_{t})\\
\leq & 2^{-j}+2H_{n}|t_{j}|
\end{align*}
for all $t\in[-j,0]\setminus E_{j}$. Similarly, for any $\tau>0$, for sufficiently large $j\in\mathbb{N}$,
we have
\[
d_{W_{1}}^{\mathcal{X}_{t_{j}}}(\delta_{(z_{j},e^{\alpha})},\mu_{t_{j}}'\otimes\nu_{e^{\alpha};t_{j}}^{\mathbb{R}^{k}})\leq d_{W_{1}}^{\mathcal{X}_{t_{j}}}(\delta_{z_{j}},\mu_{t_{j}}')+d_{W_{1}}^{\mathcal{X}_{t_{j}}}(\delta_{e^{\alpha}},\nu_{e^{\alpha};t_{j}}^{\mathbb{R}^{k}})\leq2H_{n}|t_{j}|,
\]
so by repeating the above reasoning with $(z_{j}',t_{j})$ replaced
by $(z_{j}'^{,\alpha},t_{j})$, we can also assume that 
$$(\nu_{(z_{j}'^{,\alpha},t_{j});t})_{t\in[-T_{j},0]}\xrightarrow[j\to\infty]{\mathfrak{C}}(\mu_{t}'\otimes\nu_{e^{\alpha};t}^{\mathbb{R}^{k}})_{t\in(-\infty,0)}$$
on compact time intervals. Because $(\nu_{x_{j},0;t})_{t\in[-T_{j},0]}\xrightarrow[j\to\infty]{\mathfrak{C}}(\mu_{t})_{t\in(-\infty,0)}$,
we can pass to a further subsequence to find subsets $E_{j}'$ with
$|E_{j}'|\leq2^{-j}$ such that 
\[ 
d_{W_{1}}^{Z_{t}}\left((\varphi_{t}^{\infty})_{\ast}\mu_{t},(\varphi_{t}^{j})_{\ast}\nu_{(x_{j},0);t}^{j}\right)\leq2^{-j}
\]
for all $t\in[-j,0]\setminus E_{j}'$. It follows that 
\begin{equation} \label{harmonicconvergence} d_{W_{1}}^{g_{j,t}}(\nu_{(z_{j}',t_{j});t}^{j},\nu_{(x_{j},0);t}^{j})\leq d_{W_{1}}^{Z_{t}}\left((\varphi_{t}^{j})_{\ast}\nu_{(z_{j}',t_{j});t}^{j},(\varphi_{t}^{\infty})_{\ast}\mu_{t}\right)+d_{W_{1}}^{Z_{t}}\left((\varphi_{t}^{\infty})_{\ast}\mu_{t},(\varphi_{t}^{j})_{\ast}\nu_{(x_{j},0);t}^{j}\right)\leq2^{-j+1} \end{equation}
for all $t\in[-j,0]\setminus(E_{j}\cup E_{j}')$, where $|E_{j}\cup E_{j}'|\leq2^{-j+1}$.

Let $\psi_{j}:U_{j}\to V_{j}\subseteq M_{j}\times [-j,0]$ be time-preserving diffeomorphisms from
a precompact exhaustion $(U_{j})$ of $\mathcal{R}$ realizing smooth convergence as in Theorem \ref{bamconvergence}. Then $\psi_{j}^{\ast}K(z_{j}',t_{j};\cdot,\cdot)\to(4\pi\tau)^{-\frac{n}{2}}e^{-f}$ in
$C_{loc}^{\infty}(\mathcal{R})$ and $f=f'+\frac{1}{4\tau}|x|^{2}$,  so we have the following convergence in $C_{loc}^{\infty}(\mathcal{R})$:
\[
\psi_{j}^{\ast}K(z_{j}',t_{j};\cdot,\cdot)\to(4\pi\tau)^{-\frac{n}{2}}e^{-f'-\frac{|x|^{2}}{4\tau}},
\]
\[
\psi_{j}^{\ast}K(z_{j}'^{,\alpha},t_{j};\cdot,\cdot)\to(4\pi\tau)^{-\frac{n}{2}}e^{-f'-\frac{|x-e^{\alpha}|^{2}}{4\tau}}.
\]
Writing $K(z_{j}',t_{j};\cdot,\cdot)=(4\pi\tau_{j})^{-\frac{n}{2}}e^{-f_{j}}$
and $K(z_{j}'^{,\alpha},t_{j};\cdot,\cdot)=(4\pi\tau_{j})^{-\frac{n}{2}}e^{-f_{j}^{\alpha}}$, this means
\[
\psi_{j}^{\ast}f_{j}\to f'+\frac{|x|^{2}}{4\tau},\qquad\psi_{j}^{\ast}f_{j}^{\alpha}\to f'+\frac{|x-e^{\alpha}|^{2}}{4\tau}
\]
in $C_{loc}^{\infty}(\mathcal{R})$ . 

\noindent \textbf{Claim 1: }For any $\delta'>0$ when $j=j(\delta')$
is sufficiently large, $(z_{j}',t_{j})$ and $(z_{j}'^{,\alpha},t_{j})$
are $(\delta',1)$-selfsimilar. 

Because $\int_{\mathcal{R}_{-\tau}'}f'd\mu_{t}'-\frac{(n-k)}{2}=W$ for all $t\in (-\infty,0)$, the proof of Corollary 15.47
in \cite{bamlergen3} gives the following for any $\tau\in(0,\infty)$:
\begin{align*}
\lim_{j\to\infty}\mathcal{N}_{z_{j}';t_{j}}(\tau)= & \int_{\mathcal{R}_{-\tau}}fd\mu_{t}-\frac{n}{2}=\int_{\mathcal{R}_{-\tau}'\times\mathbb{R}^{k}}\left(f'+\frac{|x|^{2}}{4\tau}\right)d(\mu_{t}'\otimes\mu_{t}^{\mathbb{R}^{k}})-\frac{n}{2}\\
= & \int_{\mathcal{R}_{-\tau}'}f'd\mu_{t}'-\frac{(n-k)}{2}=:W
\end{align*}
and similarly, 
\[
\lim_{j\to\infty}\mathcal{N}_{z_{j}'^{,\alpha};t_{j}}(\tau)=\int_{\mathcal{R}_{-\tau}'\times\mathbb{R}^{k}}\left(f'+\frac{|x-e^{\alpha}|^{2}}{4\tau}\right)d(\mu_{t}'\otimes\mu_{t}^{\mathbb{R}^{k}})-\frac{n}{2}=W.
\]
The claim follows by choosing $\tau$ sufficiently large and appealing
to Proposition 7.1 of \cite{bamlergen3}. $\square$

Choose $t\in[-1,-\frac{1}{2}]\cap I'$. Then 
\begin{align*}
d_{W_{1}}^{g_{j,-1}}\left(\nu_{z_{j}'^{,\alpha},t_{j};-1}^{j},\nu_{z_{j}',t_{j};-1}^{j}\right)\leq & d_{W_{1}}^{Z_{t}}\left((\varphi_{t}^{j})_{\ast}\nu_{z_{j}'^{,\alpha},t_{j};t}^{j},(\varphi_{t}^{\infty})_{\ast}(\mu_{t}'\otimes\nu_{(e^{\alpha},0^{k});t}^{\mathbb{R}^{k}})\right)\\&+d_{W_{1}}^{Z_{t}}\left((\varphi_{t}^{j})_{\ast}\nu_{z_{j}',t_{j};t}^{j},(\varphi_{t}^{\infty})_{\ast}(\mu_{t}'\otimes\mu_{t}^{\mathbb{R}^{k}})\right) +d_{W_{1}}^{\mathbb{R}^{k}}\left(\nu_{e^{\alpha},0;t}^{\mathbb{R}^{k}},\mu_{t}^{\mathbb{R}^{k}}\right).
\end{align*}
Because $d_{W_{1}}^{\mathbb{R}^{k}}\left(\nu_{e^{\alpha},0;t}^{\mathbb{R}^{k}},\mu_{t}^{\mathbb{R}^{k}}\right)\leq 1$ and the $\mathbb{F}$-convergence (\ref{eq:splitFcon}) is timewise at time $t$, we have $$d_{W_{1}}^{g_{j,-1}}\left(\nu_{z_{j}'^{,\alpha},t_{j};-1}^{j},\nu_{z_{j}',t_{j};-1}^{j}\right)\leq 2$$
for sufficiently large $j\in\mathbb{N}$. Now fix $\delta>0$ and
$\beta>0$ small. Taking $\delta'\leq\overline{\delta}'(\delta,n,Y)$
in Claim 1, we may therefore appeal
to the proof of Proposition 10.8 in \cite{bamlergen3} (up to Claim
10.31) to conclude that, setting $W_{j}':=\mathcal{N}_{z_{j}',t_{j}}(1)$,
$W_{j}^{\alpha}:=\mathcal{N}_{z_{j}'^{,\alpha},t_{j}}(1)$ , the functions
\[
u_{j}^{\alpha}:=2\tau_{j}(f_{j}^{\alpha}-f_{j}')-2\tau_{j}(W_{j}^{\alpha}-W_{j}')
\]
satisfy the following properties for all $j=j(\delta)\in \mathbb{N}$ sufficiently
large, assuming $\gamma \leq \overline{\gamma}$:

$(i)$ $\int_{-\delta^{-1}}^{-\delta}\int_{M_{j}}\left(\tau_{j}^{-\frac{1}{2}}|\partial_{t}u_{j}^{\alpha}|+|\nabla^{2}u_{j}^{\alpha}|^{2}\right)e^{\gamma f_{j}'}d\nu_{z_{j}',t_{j};t}^{j}dt\leq\delta$,

$(ii)$ $\int_{-\delta^{-1}}^{-\delta}\int_{M_{j}}\tau_{j}^{-1}|\nabla u_{j}^{\alpha}|^{4}e^{2\gamma f_{j}'}d\nu_{z_{j}',t_{j};t}^{j}dt\leq C(Y,\gamma),$

$(iii)$ $\int_{-\delta^{-1}}^{-\delta}\int_{M_{j}}|\langle\nabla u_{j}^{\alpha},\nabla u_{j}^{\beta}\rangle-q_{\alpha\beta}^j|d\nu_{z_j',t_j;t}^jdt\leq\delta$ for some $q_{\alpha \beta}^j \in \mathbb{R}$.

\noindent Moreover, for any $b>0$, we can combine $(ii),(iii)$ to
estimate 
\begin{align*}
\int_{-2\epsilon^{-1}}^{-\frac{1}{2}\epsilon}\int_{M_{j}}|\langle\nabla u_{j}^{\alpha},\nabla u_{j}^{\beta}\rangle-q_{\alpha\beta}^j|e^{\gamma f_{j}'}d\nu_{t}dt\leq & \frac{1}{b}\int_{-2\epsilon^{-1}}^{-\frac{1}{2}\epsilon}\int_{M_{j}}|\langle\nabla u_{j}^{\alpha},\nabla u_{j}^{\beta}\rangle-q_{\alpha\beta}^j|d\nu_{t}dt\\&+b\int_{-2\epsilon^{-1}}^{-\frac{1}{2}\epsilon}\int_{M_{j}}|\langle\nabla u_{j}^{\alpha},\nabla u_{j}^{\beta}\rangle-q_{\alpha\beta}^j|e^{2\gamma f_{j}'}d\nu_{z_{j}',t_{j};t}^{j}\\
\leq & b^{-1}\delta+C(\epsilon,Y,\gamma)b.
\end{align*}
For any $\delta'>0$, we can therefore choose $b\leq\overline{b}(\delta',Y,\gamma)$
and then $\delta\leq\overline{\delta}(\delta',Y,\gamma)$ so that

$(iii')$ $\int_{-2\epsilon^{-1}}^{-\frac{1}{2}\epsilon}\int_{M_{j}}|\langle\nabla u_{j}^{\alpha},\nabla u_{j}^{\beta}\rangle-q_{\alpha\beta}^j|e^{\gamma f_{j}'}d\nu_{t}dt\leq\delta'$ for some $q_{\alpha \beta}^j \in \mathbb{R}$\\
whenever $j=j(\epsilon,\delta')\in \mathbb{N}$ is sufficiently large.\\
\noindent \textbf{Claim 2: } For any $\delta''>0$, twe have $|q_{\alpha\beta}^j-\delta_{\alpha\beta}|\leq\delta''$
for sufficiently large $j=j(\delta'')\in\mathbb{N}$. 

Fix any compact set $K\subseteq\mathcal{R}_{[-2,-1]}$, and set $K_t := K\cap \mathcal{R}_t$ or $t\in [-2,-1]$. For sufficiently
large $j$, we can estimate
\begin{align*}
\int_{-2}^{-1}\int_{M_{j}}|\langle\nabla u_{j}^{\alpha},\nabla u_{j}^{\beta}\rangle-\delta_{\alpha\beta}|d\nu_{z_{j}',t_{j};t}^{j}dt \hspace{-40 mm} & \\ =& \int_{-2}^{-1}\int_{\psi_{j,t}(K_t)}|\langle\nabla u_{j}^{\alpha},\nabla u_{j}^{\beta}\rangle-\delta_{\alpha\beta}|d\nu_{z_{j}',t_{j};t}^{j}dt +\int_{-2}^{-1}\int_{M_{j}\setminus\psi_{j,t}(K_t)}|\langle\nabla u_{j}^{\alpha},\nabla u_{j}^{\beta}\rangle-\delta_{\alpha\beta}|d\nu_{z_{j}',t_{j};t}^{j}dt\\
\leq & \int_{-2}^{-1}\int_{K_t}2\tau_{j}\left|\langle\nabla(f_{j}^{\alpha}-f_{j}),\nabla(f_{j}^{\beta}-f_{j})\rangle\circ\psi_{j}-\delta_{\alpha\beta}\right|\left(K(z_{j}',t_{j};\cdot,\cdot)\circ\psi_{j}\right)d(\psi_{j,t}^{\ast}g_{j,t})dt\\
 & +\left(\int_{-2}^{-1}\int_{M_{j}\setminus\psi_{j}(K)}|\langle\nabla u_{j}^{\alpha},\nabla u_{j}^{\beta}\rangle-\delta_{\alpha\beta}|^{2}d\nu_{z_{j}',t_{j};t}^{j}dt\right)^{\frac{1}{2}}\left(\int_{-2}^{-1}\nu_{z_{j}',t_{j};t}^{j}(M\setminus\psi_{j}(K))dt\right)^{\frac{1}{2}}\\
\leq & \sup_{K}\left|2\tau_{j}\langle\nabla(f_{j}^{\alpha}-f_{j}),\nabla(f_{j}^{\beta}-f_{j})\rangle\circ\psi_{j}-\delta_{\alpha\beta}\right|\\
 & +C(Y,\gamma)\left(1-\int_{-2}^{-1}\nu_{z_{j}',t_{j};t}^{j}(\psi_{j,t}(K_t))dt\right)^{\frac{1}{2}},
\end{align*}
where we used estimate $(ii)$ to obtain the last inequality. Next, we observe that as $j\to\infty$, $2\tau_{j}\langle\nabla(f_{j}^{\alpha}-f_{j}),\nabla(f_{j}^{\beta}-f_{j})\rangle\circ\psi_{j}$ converges in $C_{loc}^{\infty}(\mathcal{R})$ to
\begin{align*}
 & 2\tau\left\langle \nabla\left(f'+\frac{|x-e^{\alpha}|^{2}}{4\tau}\right)-\nabla\left(f'+\frac{|x|^{2}}{4\tau}\right),\nabla\left(f'+\frac{|x-e^{\beta}|^{2}}{4\tau}\right)-\nabla\left(f'+\frac{|x|^{2}}{4\tau}\right)\right\rangle \\
 &\hspace{12 mm}= \left\langle (x-e^{\alpha})-x,(x-e^{\beta})-x\right\rangle =\delta_{\alpha\beta}
\end{align*}
Moreover, we have 
\[
\int_{-2}^{-1}\nu_{z_{j}',t_{j};t}^{j}(\psi_{j,t}(K_t))dt=\int_{-2}^{-1}\int_{K_t}\left(K(z_{j}',t_{j};\cdot,t)\circ\psi_{j,t}\right)d(\psi_{j,t}^{\ast}g_{j,t})dt\to\int_{-2}^{-1}\mu_{t}(K)dt.
\]
Combining expressions, we get
\begin{align*}
\limsup_{j\to \infty} |\delta_{\alpha\beta}-q_{\alpha\beta}^j|\leq & \limsup_{j\to \infty}\int_{-2}^{-1}\int_{M_{j}}|\langle\nabla u_{j}^{\alpha},\nabla u_{j}^{\beta}\rangle-\delta_{\alpha\beta}|d\nu_{z_{j}',t_{j};t}^{j}dt\\ &+\limsup_{j\to \infty}\int_{-2}^{-1}\int_{M_{j}}|\langle\nabla u_{j}^{\alpha},\nabla u_{j}^{\beta}\rangle-q_{\alpha\beta}^j|d\nu_{z_{j}',t_{j};t}^{j}dt\\
\leq & C(\gamma,n) \left(1-\int_{-2}^{-1}\mu_{t}(K)dt\right)^{\frac{1}{2}}+\delta'
\end{align*}
for any compact set $K\subseteq\mathcal{R}_{[-2,-1]}$. Choosing a compact exhaustion of $\mathcal{R}_{[-2,-1]}$ and $\delta'\leq\overline{\delta}'(\delta'')$
then gives the claim. $\square$

By the $W_1$-distance estimate (\ref{harmonicconvergence}), we can apply Proposition 8.1 of \cite{bamlergen3} with the following choice of parameters: $\alpha_{0}=\gamma$, $\alpha_{1}=0$, $t_{0}=t_{j}$,
$t_{1}=0$, $D=1$, $s\in[-\epsilon^{-1},-\epsilon]$, $t^{\ast}=\frac{1}{2}\theta(1)\gamma \epsilon$, where $\theta$ is defined in the aforementioned proposition;
then for $j\in\mathbb{N}$ sufficiently large, we have 
\[
d\nu_{x_{j},0;s}^{j}\leq C(Y,\gamma)e^{\gamma f_{j}'}d\nu_{z_{j}',t_{j};s}^{j}
\]
for all $s\in[-\epsilon^{-1},-\epsilon]$. Combining this with $(i),(iii')$, the $(\delta,1)$-selfsimilarity of $(z_i',t_i)$, and taking $\delta$ sufficiently small gives the desired contradiction.
\end{proof}

Next, we verify that an $\mathbb{F}$-limit of metric solitons splitting $\mathbb{R}^k$ is also a metric soliton splitting $\mathbb{R}^k$.

\begin{lem} \label{diagonal}
Suppose $(\mathcal{X}_{i},(\mu_t^i)_{t\in(-\infty,0)})$ is a sequence of metric solitons satisfying the following:
\begin{itemize}
\item[$(a)_i$] Each pair $(\mathcal{X}_{i},(\mu_t^i)_{t\in(-\infty,0)})$ is an $\mathbb{F}$-limit of $n$-dimensional closed Ricci flows as in Theorem \ref{bamconvergence}, with Nash entropy lower bound $-Y$,
\item[$(b)_i$] $(\mathcal{X}_i,(\mu_t^i)_{t\in (-\infty,0)}) \cong (\mathcal{X}_i'\times \mathbb{R}^k,(\mu_t '^{,i}\otimes \mu_{\mathbb{R}^k})_{t\in (-\infty,0)})$ as metric flow pairs for some metric solitons $(\mathcal{X}_i',(\mu_t '^{,i})_{t\in (-\infty,0)})$, and this identification restricts to an isometry of Ricci flow spacetimes $\mathcal{R}_i \cong \mathcal{R}_i'\times \mathbb{R}^k$,\\
\item[$(c)_i$] Writing $d\mu_t^i = (4\pi \tau)^{-\frac{n}{2}}e^{-f_i}dg_{i}$ on $\mathcal{R}_i$ and $d\mu_t'^{,i} = (4\pi \tau)^{-\frac{n}{2}}e^{-f_i'}dg_{i}'$ on $\mathcal{R}_i'$, we have $Rc(g_{i})+\nabla^2 f_i =\frac{1}{2\tau}g_{i}$ on $\mathcal{R}_i$, $Rc(g_{i}')+\nabla^2 f_i' =\frac{1}{2\tau}g_{i}'$ on $\mathcal{R}_{i}'$, and $f_i =f_i'+\frac{1}{4\tau}|x|^2$ on $\mathcal{R}_i$.
\item[$(d)_i$] $\mathcal{N}_{(\mu_{i,t})}(\tau)=W_i$ for all $\tau >0$, where $W_i \in [-Y,0]$.
\end{itemize}
Assume that $(\mathcal{X}_{i},(\mu_t^i)_{t\in(-\infty,0)})$ $\mathbb{F}$-converge to another metric flow pair $(\mathcal{X}_{\infty},(\mu_t^{\infty})_{t\in(-\infty,0)})$. Then $(\mathcal{X},(\mu_t^{\infty})_{t\in(-\infty,0)})$ is a metric soliton satisfying $(a)-(d)$ of Proposition \ref{getsplit}. 

If in addition $Rc(g_i')=0$, then we have $(\mathcal{X}_{\infty},(\mu_t^{\infty})) \cong (\mathcal{X}''\times \mathbb{R}^k,(\nu_{x'';t})_{t\in (-\infty ,0)})$ as metric flow pairs for some static metric cone $\mathcal{X}''$ with vertex $x''$. Moreover, this identification restricts to an isometry of Ricci flow spacetimes $\mathcal{R}^{\infty}\cong \mathcal{R}''\times \mathbb{R}^k$, and $Rc(g^{\infty})=0$, $Rc(g'')=0$.
\end{lem}

\begin{proof}
By passing to subsequences and using a diagonal argument, we may assume that 
$$(M_{i}^{n},(g_{i,t})_{t\in(-\epsilon_{i}^{-1},0]},(x_i,0))$$
is a sequence of closed, pointed Ricci flows such that $\mathcal{N}_{x_{i},0}(1)\geq-Y$,
\[
d_{\mathbb{F}}\left((M_{i},(g_{i,t})_{t\in(-\epsilon_{i}^{-1},0]},(\nu_{x_{i},0;t})_{t\in(-\epsilon_i^{-1},0]}),(\mathcal{X}_{i},(\mu_t^i)_{t\in(-\epsilon_i^{-1},0]})\right) \leq \epsilon_i,
\]
where $\lim_{i\to \infty}\epsilon_i =0$. By Proposition \ref{getsplit}, we may moreover assume that $(x_i,0)$ are $(\epsilon_i,1)$-selfsimilar and strongly $(k,\epsilon_i,1)$-split. In particular, 
$$(M_{i},(g_{i,t})_{t\in(-T_{i},0]},(\nu_{x_{i}',0;t})_{t\in(-T_{i},0]}) \xrightarrow[i\to \infty]{\mathbb{F}} (\mathcal{X}_{\infty},(\mu_t^{\infty})_{t\in(-\infty,0]}).$$
Then $\mathcal{X}$ satisfies $(a)$ by construction, $(b)$ by Theorem 15.50 in \cite{bamlergen3}, and $(d)$ by Proposition 7.1 in \cite{bamlergen3} and the Nash entropy convergence Theorem 15.45 of \cite{bamlergen3}. Moreover, we have $Rc(g)+\nabla^2 f=\frac{1}{2\tau}g$ on $\mathcal{R}$ by Theorem 15.69 of \cite{bamlergen3}. By $(b)$, we have
$$(4\pi \tau)^{-\frac{n}{2}}e^{-f}dg = \left( (4\pi \tau)^{-\frac{n-k}{2}}e^{-f'}dg'\right) \otimes \left( (4\pi \tau)^{-\frac{k}{2}}e^{-\frac{|x|^2}{4\tau}} \right)=(4\pi \tau)^{-\frac{n}{2}}e^{-f'-\frac{|x|^2}{4\tau}}dg$$
on $\mathcal{R}$, so that $f=f'+\frac{|x|^2}{4\tau}$. Thus $Rc(g),g,\nabla^2 f$ all split with respect to the decomposition $T\mathcal{R}=T\mathcal{R}'\oplus T\mathbb{R}^k$, so restricting $Rc(g)+\nabla^2 f = \frac{1}{2\tau}g$ to $T\mathcal{R}'$ gives $Rc(g')+\nabla^2 f' =\frac{1}{2\tau}g'$ on $\mathcal{R}'$.

Finally, suppose that $Rc(g_i')=0$ for all $i\in \mathbb{N}$. By the proof of Claim 22.7 in \cite{bamlergen3}, we can also assume $(x_i,0)$ are $(\epsilon_i,1)$-static. Then Theorem 15.80 of \cite{bamlergen3} guarantees the remaining claims. 
\end{proof}

We may now apply the previous results to prove the reverse qualitative inclusion of quantitative singular strata as that proved in \cite{bamlergen3}

\begin{prop} \label{equiv}
For any $\epsilon>0$, there exists $\delta=\delta(\epsilon,Y,A)>0$
such that the following holds. Suppose $(\mathcal{X},(\nu_{x_{\infty};t})_{t\in (-T,0]})$ is a metric flow pair obtained as an $\mathbb{F}$-limit of noncollapsed Ricci flows as in Theorem \ref{bamconvergence} (with $\mathcal{N}_{x_i,0}(1)\geq -Y$). Assume $y_{\infty}\in\mathcal{X}_{<0} \cap P^{\ast}(x_{\infty},A,-A^2)$ is $(k,\delta,r)$-symmetric and that $(M,(g_t)_{t\in [-\delta^{-1},0]},(y_0,0))$ is a closed pointed Ricci flow such that $|\mathcal{N}_{y_0,0}(1)-\mathcal{N}_{y_{\infty}}(1)|<\delta$ and
$$d_{\mathbb{F}}\left( (M,(g_t)_{t\in [-\delta^{-1},0]},(\nu_{y_0,0;t})_{t\in [-\delta^{-1},0]}),(\mathcal{X}^{-t_0,r^{-1}},(\nu_{y_\infty;t}^{-t_0,r^{-1}})_{t\in [-\delta^{-1},0]}) \right) <\delta ,$$
where $t_0 := \mathfrak{t}(y_{\infty})$. Then one of the following holds:

$(i)$ $(y_0,t_0)$ is $(k,\epsilon,1)$-split and $(\epsilon,1)$-selfsimilar,

$(ii)$ $(y_0,t_0)$ is $(k-2,\epsilon,1)$-split, $(\epsilon,1)$-static,
and $(\epsilon,1)$-selfsimilar.\\
In particular, $y_{\infty}$ is weakly $(k,\epsilon,1)$-symmetric, and $\widehat{\mathcal{S}}_{r_1,r_2}^{\epsilon,k} \subseteq \mathcal{S}_{r_1,r_2}^{\delta(\epsilon,Y,A),k}$. Moreover, we have
$$\mathcal{S}=\cup_{\epsilon \in (0,1)} \mathcal{S}_{0,\epsilon}^{\epsilon,k}.$$
\end{prop}

\begin{proof}
The hypotheses ensure that $\mathcal{N}_{y_{\infty}}(1)\geq -Y'(Y,A)$. By time translation and parabolic rescaling, we can assume $r=1$ and $\mathfrak{t}(y_{\infty})=0$. Suppose by way of contradiction there is a sequence $\delta_i \searrow 0$ along with metric flows $\mathcal{X}^i$ each obtained as $\mathbb{F}$-limits of closed noncollapsed $n$-dimensional Ricci flows, $(k,\delta_i,1)$-symmetric points $y_{\infty,i}\in \mathcal{X}_0^i$ with $\mathcal{N}_{y_{\infty,i}}(1)\geq -Y'$, closed pointed Ricci flows $(M_i,(g_{i,t})_{t\in [-\delta_i^{-1},0]},(y_i,0))$ satifying $|\mathcal{N}_{y_i,0}(1)-\mathcal{N}_{y_{\infty,i}}(1)|<\delta_i$ and 
$$d_{\mathbb{F}}\left( (M_i,(g_{i,t})_{t\in [-\delta_i^{-1},0]},(\nu_{y_i,0;t})_{t\in [-\delta_i^{-1},0]}),(\mathcal{X}_i,(\nu_{y_{\infty,i};t})_{t\in [-\delta_i^{-1},0]}) \right) <\delta_i ,$$
such that $(i),(ii)$ both fail for $(y_i,0)$. Because $y_{\infty,i}$ are $(k,\delta_i,1)$-symmetric, we can find metric flow pairs $(\mathcal{X}_i',(\mu_t'^{,i})_{t\in [-\delta_i^{-1},0]})$ satisfying properties $(b)-(d)$ in Definition \ref{symdef}, along with $|\mathcal{N}_{\mu_t'^{,i}}(1)-\mathcal{N}_{y_{\infty,i}}(1)|<\delta_i$ and 
$$d_{\mathbb{F}}\left( (\mathcal{X}_i,(\nu_{y_{\infty},i};t)_{t\in [-\delta_i^{-1},0]}),(\mathcal{X}_i',(\mu_t'^{,i})_{t\in [-\delta_i^{-1},0]}) \right) <\delta_i.$$
Because $\mathcal{N}_{y_i,0}(1)\geq -Y'-\delta_i \geq -2Y'$, Theorem \ref{bamconvergence} lets us pass to a subsequence to obtain a future-continuous metric flow pair $(\mathcal{X}',(\mu_t')_{t\in (-\infty,0]})$ such that 
$$(M_i,(g_{i,t})_{t\in [-\delta_i^{-1},0]},(\nu_{y_i,0;t})_{t\in [-\delta_i^{-1},0]})\xrightarrow[i\to \infty]{\mathbb{F}} (\mathcal{X}',(\mu_t')_{t\in (-\infty,0]})$$
uniformly on compact time intervals. By construction we also have
$$(\mathcal{X}_i',(\mu_t'^{,i})_{t\in [-\delta_i^{-1},0]}) \xrightarrow[i\to \infty]{\mathbb{F}} (\mathcal{X}',(\mu_t')_{t\in (-\infty,0]})$$
on compact time intervals, so by Lemma \ref{diagonal}, $\mathcal{X}'$ also satisfies properties $(b)-(d)$ of Definition \ref{symdef}. By Proposition \ref{getsplit}, we conclude that $(y_i,0)$ satisfies one of $(i)$ or $(ii)$ when $i\in \mathbb{N}$ is sufficiently large, a contradiction.
\end{proof}

\section{Strong Almost-GRS Potentials}

In this section, we construct parabolic regularizations of potential functions associated to conjugate heat kernels based at almost-selfsimilar points of a Ricci flow. These functions still satisfy the almost-soliton identities, but also satisfy additional estimates which will be useful in Section 5.

\begin{defn} A strong $(\epsilon,r)$-soliton potential based at $(x_{0},t_{0})\in M\times I$
is a function $h\in C^{\infty}(M\times[t_{0}-\epsilon^{-1}r^{2},t_{0}-\epsilon r^{2}])$
such that if $W:=\mathcal{N}_{x_{0},t_{0}}(r^{2})$, then 
\begin{flalign*} &\hspace{5 mm}(i) \hspace{3 mm} \square\left(4\tau(h-W)\right)=-2n,& \end{flalign*}
\vspace{-6 mm}
\begin{flalign*} &\hspace{5 mm}(ii) \hspace{3 mm} r^{-2}\int_{t_{0}-\epsilon^{-1}r^{2}}^{t_{0}-\epsilon r^{2}}\int_{M}\left|\tau(R+|\nabla h|^{2})-(h-W)\right|d\nu_{t}dt\leq\epsilon,,& \end{flalign*}
\vspace{-4 mm}
\begin{flalign*} &\hspace{5 mm}(iii) \hspace{3 mm} \int_{t_{0}-\epsilon^{-1}r^2}^{t_{0}-\epsilon r^2}\int_{M}\tau \left| Rc+\nabla^{2}h-\frac{1}{2\tau}g\right|^{2}d\nu_{t}dt\leq\epsilon ,& \end{flalign*}
\vspace{-4 mm}
\begin{flalign*} &\hspace{5 mm}(iv) \hspace{3 mm} \sup_{t\in[t_{0}-\epsilon^{-1}r^2,t_{0}-\epsilon r^2]}\int_{M}\left|\tau(R+2\Delta h-|\nabla h|^{2})+h-n-W\right|d\nu_{t}\leq\epsilon ,& \end{flalign*}
\vspace{-4 mm}
\begin{flalign*} & \hspace{5 mm}(v) \hspace{3 mm} \int_{M}\left(h-\frac{n}{2}\right)d\nu_{t}=W \text{ for all } t\in [t_0-\epsilon^{-1}r^2,t_0-\epsilon r^2].& \end{flalign*}
\end{defn}

\vspace{4 mm}
The following proposition is an analogue of Bamler's construction (Theorem 12.1 of \cite{bamlergen3}) of strong almost-splitting maps which approximate weak almost-splitting maps. 

\begin{prop} \label{strongsoliton}
For any $\epsilon>0$, $Y<\infty$, the following holds whenever $\delta\leq\overline{\delta}(\epsilon,Y)$.
Suppose $(M^{n},(g_{t})_{t\in[-T,0]})$ is a closed Ricci flow with
$\mathcal{N}_{x_{0},t_{0}}(r^{2})\geq-Y$. Assume $(x_{0},t_{0})\in M\times I$
is $(\delta,r)$-selfsimilar, and set $q:=4\tau(f-W)$, where $d\nu=d\nu_{x_{0},t_{0}}=(4\pi\tau)^{-\frac{n}{2}}e^{-f}dg$
and $W:=\mathcal{N}_{x_{0},t_{0}}(r^{2})$. Then there exists a function $q'\in C^{\infty}(M\times[t_{0}-\epsilon^{-1}r^{2},t_{0}-\epsilon r^{2}])$ such that $f':=\frac{1}{4\tau}q'+W$ is a strong $(\epsilon,r)$-soliton potential based at $(x_0,t_0)$, and
\begin{equation} \label{close} \int_{t_{0}-\epsilon^{-1}r^{2}}^{t_{0}-\epsilon r^{2}}\int_{M}|\nabla(f-f')|^{2}d\nu_{t}dt + \sup_{t\in [t_0-\epsilon^{-1}r^2,t_0-\epsilon r^2]} \int_M (f-f')^2 d\nu_t \leq\epsilon .\end{equation}
\end{prop}

\begin{proof}
Without loss of generality, we can assume $r=1$ and $t_{0}=0$. Bamler's on-diagonal upper bounds for the heat kernel (Theorem 7.1 in \cite{bamlergen1}) imply $f\geq-\Lambda(Y)$ on
$M\times[-\epsilon^{-1},-\epsilon]$ if $\delta \leq \overline{\delta}$. Fix $Z\in(\Lambda,\infty)$
to be determined, and let $\chi\in C^{\infty}(\mathbb{R})$ be such
that $\chi(s)=s$ for $s\in(-\infty,\frac{1}{2}Z]$, $|\chi'|\leq 2$,  $|\chi''|\leq 10Z^{-1}$, and $\chi(s)=Z$ for all $s\in[Z,\infty]$. Set $\widetilde{q}:=\chi\circ q$, which satisfies $\widetilde{q}\geq -4\tau(\Lambda+Y)$.

\noindent \textbf{Step 1: (Bound the truncation errors) }We first
apply Proposition 6.5 of \cite{bamlergen3} to obtain
\[
\nu_{t}(\{q\geq Z\})\leq e^{-\frac{Z}{8\tau}}\int_{\{q\ge Z\}}e^{\frac{q}{8\tau}}d\nu_{t}\leq C(Y)e^{-\frac{Z}{8\tau}}\int_{\{q\geq Z\}}e^{\frac{1}{2}f}d\nu_{t}\leq C(Y)e^{-\frac{Z}{8\tau}}
\]
for all $t\in [-\frac{1}{2}\delta^{-1},-\delta]$. Thus, for any
$p\in[1,\infty)$, 
\begin{align*}
\int_{\{q\geq Z\}}q^{p}d\nu_{t}= & \int_{\{q\geq Z\}}\left(p\int_{0}^{q(x)}r^{p-1}dr\right)d\nu_{t}(x)=p\int_{0}^{\infty}\int_{\{q\geq Z\}}1_{\{r\leq q(x)\}}r^{p-1}d\nu_{t}(x)dr\\
= & p\int_{Z}^{\infty}r^{p-1}\nu_{t}(\{q\geq r\})dr +pZ^p \nu_t(\{q\geq Z\}) \\
\leq&  C(Y,p)\int_{Z}^{\infty}r^{p-1}e^{-\frac{r}{8\tau}}dr +C(Y,p)Z^p e^{-\frac{Z}{8\tau}}
\end{align*}
In particular, we have 
\[
\int_{\{q\ge Z\}}q^{p}d\nu_{t}\leq C(Y,p,\epsilon)e^{-\frac{\epsilon Z}{10}}.
\]
for all $t\in [-10\epsilon^{-1},-\frac{1}{10}\epsilon]$. Using (\ref{spacetime}), we can estimate
\begin{align*}
\int_{-10\epsilon^{-1}}^{-\frac{1}{10}\epsilon}\int_{M}|\nabla(q-\widetilde{q})|^{2}d\nu_{t}dt\leq & 2\int_{-10\epsilon^{-1}}^{-\frac{1}{10}\epsilon}\int_{\{q\geq Z/2\}}|\nabla q|^{2}d\nu_{t}dt\\
&\leq \left( \int_{-10\epsilon^{-1}}^{-\frac{1}{10}\epsilon} \int_M |\nabla q|^4 d\nu_t dt \right)^{\frac{1}{2}} \left( \int_{-10\epsilon^{-1}}^{-\frac{1}{10}\epsilon} \nu_t (\{ q \geq Z/2 \}) dt \right)^{\frac{1}{2}}\\ &\leq C(Y,\epsilon)e^{-\frac{\epsilon Z}{20}}.
\end{align*}
Next, we estimate
\[
\int_{M}(q-\widetilde{q})^{2}d\nu_{t}\leq\int_{\{q\geq Z/2\}}2q^{2}d\nu_{t}\leq C(Y,\epsilon)e^{-\frac{\epsilon Z}{10}}
\]
for any $t\in [-10\epsilon^{-1},-\frac{1}{10}\epsilon]$,
and apply (\ref{spacetime}) to obtain
\begin{align*}
\int_{-10\epsilon^{-1}}^{-\frac{1}{10}\epsilon}\int_{M}|\Delta(q-\widetilde{q})|d\nu_{t}dt & \leq  \int_{-10\epsilon^{-1}}^{-\frac{1}{10}\epsilon}\int_{M}\left|1-(\chi'\circ q)\right|\cdot|\Delta q|d\nu_{t}dt+\int_{-10\epsilon^{-1}}^{-\frac{1}{10}\epsilon}\int_{M}|\chi''\circ q|\cdot|\nabla q|^{2}d\nu_{t}dt\\
 &\leq \left( \int_{-10\epsilon^{-1}}^{-\frac{1}{10}\epsilon} \int_M \left( 4|\Delta q|^2 +|\nabla q|^4 \right) d\nu_t dt \right)^{\frac{1}{2}} \left( \int_{-10\epsilon^{-1}}^{-\frac{1}{10}\epsilon} \nu_t (\{ q \geq Z/2 \}) dt \right)^{\frac{1}{2}} \\
 &\leq C(Y,\epsilon)e^{-\frac{\epsilon Z}{20}}.
\end{align*}
\textbf{Step 2: (Estimate the error by parabolic regularization) } By Step 1, we can choose
$t^{\ast}\in[-2\epsilon^{-1},-2\epsilon^{-1}+1]$ such that \[
\int_{M}|\nabla(q-\widetilde{q})|^{2}d\nu_{t^{\ast}}+\int_{M}|q-\widetilde{q}|^{2}d\nu_{t^{\ast}}+\int_{M}|\Delta(q-\widetilde{q})|d\nu_{t^{\ast}}\le C(Y,\epsilon)e^{-\frac{\epsilon Z}{20}}.
\]

\noindent Let $q'\in C^{\infty}(M\times[t^{\ast},0])$ solve $\square q'=-2n$,
with $q'(\cdot,t^{\ast}):=\widetilde{q}$. Because $-\Lambda \leq\widetilde{q}\leq Z$,
the maximum principle gives $-8\epsilon^{-1}(\Lambda+Y)-4n\epsilon^{-1}\leq q'\leq Z$ on
$M\times[t^{\ast},0]$, so if $Z\geq 8\epsilon^{-1}(\Lambda+Y)$, then (\ref{timeslice}) and the almost-selfsimilar inequalities (Definition \ref{almost}$(i)$) imply
\begin{align*}
\frac{d}{dt}\int_{M}(\widetilde{q}-q')^{2}d\nu_{t}= & 2\int_{M}(\widetilde{q}-q')(\square\widetilde{q}+2n)d\nu_{t}-2\int_{M}|\nabla(\widetilde{q}-q')|^{2}d\nu_{t}\\
\leq & C(\epsilon,Y)Z\int_{M}\left( |\square q+2n|+|1-\chi'\circ q|+(\chi''\circ q)|\nabla q|^2  \right)d\nu_{t}\\& -2\int_{M}|\nabla(\widetilde{q}-q')|^{2}d\nu_{t} \\
\leq & C(\epsilon,Y)Z\nu_t(\{ q\geq Z/2 \})+\left( \nu_t(\{q \geq Z/2\}) \right)^{\frac{1}{2}} \left( \int_M |\nabla q|^4 d\nu_t \right)^{\frac{1}{2}} \\
& +\Psi(\delta|\epsilon,Y,Z)-2\int_M |\nabla (\widetilde{q}-q')|^2 d\nu_t
\end{align*}
so we may integrate in time, using Step 1, H\"older's inequality, and (\ref{spacetime}) to obtain
\begin{align*}
\sup_{t\in[t^{\ast},-\frac{1}{10}\epsilon]}\int_{M}(\widetilde{q}-q')^{2}d\nu_{t}+\int_{t^{\ast}}^{-\frac{1}{10}\epsilon}\int_{M}|\nabla(\widetilde{q}-q')|^{2}d\nu_{t}dt\leq & \Psi(\delta|Y,\epsilon,Z) + C(Y,\epsilon)Z e^{-\frac{\epsilon Z}{20}}.
\end{align*}
Combining this with the estimates of Step 1 gives
$$\sup_{t\in [t^{\ast},-\frac{1}{10}\epsilon]}\int_M (q-q')^2 d\nu_t +\int_{t^\ast}^{-\frac{1}{10}\epsilon} \int_M |\nabla (q-q')|^2 d\nu_t dt \leq \Psi(\delta|Y,\epsilon,Z)+C(Y,\epsilon)Z e^{-\frac{\epsilon Z}{20}},$$
hence (\ref{close}) holds if we choose $Z\geq \underline{Z}(\epsilon,Y)$ and then $\delta \leq \overline{\delta}(\epsilon,Y,Z)$.\\

\noindent \textbf{Step 3: (Show $f'$ is an almost-soliton potential
function) } We now consider a quantity analogous to Perelman's differential Harnack quantity:
$$w':=\tau \left( R+\frac{1}{2\tau}\Delta q' -\frac{1}{16\tau^2}|\nabla q'|^2 \right) +\frac{1}{4\tau}q'-n+W,$$
so that
\begin{align} \label{harnack} \nonumber \square \left( 16\tau (w'-W) \right) & = \square \left( 16\tau^2 +8\tau \Delta q' -|\nabla q'|^2 +4q' -16\tau n \right) \\ & = 32\tau^2 \left| Rc +\frac{1}{4\tau}\nabla^2 q' -\frac{1}{2\tau}g \right|^2.
\end{align}
We can write
\begin{align*} w'-W &= (w'-w)+(w-W)\\
&= \frac{1}{2}\Delta (q'-q)-\frac{1}{16\tau}\left( |\nabla q'|^2 -|\nabla q|^2 \right) +\frac{1}{4\tau}(q'-q)+(w-W).
\end{align*}
By Step 1, we can estimate
\begin{align*} \int_M \left| |\nabla q'|^2 -|\nabla q|^2 \right| d\nu_{t^{\ast}} &\leq C \int_{\{q\geq Z/2\}} |\nabla q|^2 d\nu_{t^{\ast}} \leq C(Y,\epsilon)e^{-\frac{\epsilon Z}{20}},
\end{align*}
and so (also using the almost-selfsimilar identities)
\begin{equation} \label{initial} \int_M 16\tau|w'-W|d\nu_{t^{\ast}} \leq C(Y,\epsilon)e^{-\frac{\epsilon Z}{20}} +C(\epsilon)\delta.\end{equation}
Using (\ref{spacetime}) and Step 2, we have 
\begin{align} \label{squared}
\int_{t^{\ast}}^{-\frac{\epsilon}{2}}\int_{M}\left||\nabla q'|^{2}-|\nabla q|^{2}\right|d\nu_{t}dt\leq & \int_{t^{\ast}}^{-\frac{\epsilon}{2}}\int_{M}\left|\langle\nabla(q'-q),\nabla(q'-q)+2\nabla q\rangle\right|d\nu_{t}dt\\ \nonumber
\leq & \int_{t^{\ast}}^{-\frac{\epsilon}{2}}\int_{M}\left|\nabla(q'-q)\right|^{2}d\nu_{t}dt+C(Y,\epsilon)\left(\int_{t^{\ast}}^{-\frac{\epsilon}{2}}\int_{M}|\nabla(q'-q)|^{2}d\nu_{t}dt\right)^{\frac{1}{2}}\\ \nonumber
\leq & \Psi(\delta|\epsilon,Y,Z)+C(Y,\epsilon)Z e^{-\frac{\epsilon Z}{40}},
\end{align}
\begin{align} \label{laplace}
\int_{t^{\ast}}^{-\frac{\epsilon}{2}} \left| \int_M \Delta (q'-q)d\nu_t \right| dt &= \int_{t^{\ast}}^{-\frac{\epsilon}{2}} \left| \int_M \langle \nabla f, \nabla (q'-q) \rangle d\nu_t \right| dt\\ \nonumber &\leq \left( \int_{t^{\ast}}^{-\frac{\epsilon}{2}} \int_M |\nabla f|^2 d\nu_t dt \right)^{\frac{1}{2}} \left( \int_{t^{\ast}}^{-\frac{\epsilon}{2}} \int_M |\nabla (q'-q)|^2 d\nu_t dt \right)^{\frac{1}{2}}\\ \nonumber
&\leq \Psi(\delta|\epsilon,Y,Z)+C(Y,\epsilon)Z e^{-\frac{\epsilon Z}{40}}.
\end{align}
Choose a cutoff function $\zeta\in C^{\infty}([t^{\ast},-\frac{1}{2}\epsilon])$
satisfying $\zeta(t)|[t^{\ast},-\epsilon]\equiv 1$, $|\zeta'|\leq 4\epsilon^{-1}$, and $\zeta(-\frac{\epsilon}{2})=0$. Then (\ref{harnack}) implies
\begin{equation}\label{evolution} \int_{t^{\ast}}^{-\frac{\epsilon}{2}} \zeta \frac{d}{dt} \int_M \tau(w'-W)d\nu_t dt = 2\int_{t^{\ast}}^{-\frac{\epsilon}{2}} \int_M \zeta \tau^2 \left| Rc+\frac{1}{4\tau}\nabla^2 q' -\frac{1}{2\tau}g \right|^2 d\nu_t dt. \end{equation}
On the other hand, integration by parts gives
\begin{equation} \label{IBP} \int_{t^{\ast}}^{-\frac{\epsilon}{2}} \zeta \frac{d}{dt} \int_M \tau(w'-W)d\nu_t dt = -\int_{t^{\ast}}^{-\frac{\epsilon}{2}} \zeta' \int_M \tau (w'-W)d\nu_t dt -\int_M \tau(w'-W)d\nu_{t^{\ast}}.\end{equation}
By combining (\ref{evolution}),(\ref{IBP}) with estimates (\ref{initial}),(\ref{squared}),(\ref{laplace}), and Steps 1,2, we obtain
\begin{equation} \label{goodsoliton} \int_{t^{\ast}}^{-\epsilon} \int_M \tau^2 \left| Rc+\frac{1}{4\tau}\nabla^2 q'-\frac{1}{2\tau}g \right|^2 d\nu_t dt \leq \Psi(\delta|\epsilon,Y,Z) +C(Y,\epsilon)Z e^{-\frac{\epsilon Z}{40}}.\end{equation}
Thus, property $(iii)$ of strong almost-soliton potential functions holds if we choose $Z\geq \underline{Z}(\epsilon,Y)$ and $\delta \leq \overline{\delta}(\epsilon,Y,Z)$.
In the sense of distributions,
\[
\square |\tau(w'-W)| \leq 2\tau^2 \left| Rc+\frac{1}{4\tau}\nabla^2 q'-\frac{1}{2\tau}g \right|^2,
\]
so we can use (\ref{goodsoliton}) to estimate
\begin{align*} \int_M |\tau(w'-W)| d\nu_t - \int_M |\tau(w'-W)|d\nu_{t^{\ast}} &= \int_{t^{\ast}}^t \int_M \square |\tau(w'-W)|d\nu_s ds \\ &\leq \Psi(\delta|Y,\epsilon,Z)+C(Y,\epsilon)Z e^{-\frac{\epsilon Z}{40}} \end{align*}
for all $t\in [t^{\ast},-\epsilon]$. Combining this with (\ref{initial}) gives property $(iv)$ of strong almost-soliton functions if we choose $Z\geq \underline{Z}(\epsilon,Y)$ and $\delta \leq \overline{\delta}(\epsilon,Y,Z)$, while $(ii)$
follows by combining $(iii),(iv)$. To verify $(v)$, we note that
$$\frac{d}{dt}\left( \tau \int_M  \left( f'-\frac{n}{2}\right) d\nu_t -\tau W\right)= \frac{1}{4}\int_M (\square q +2n)d\nu_t =0,$$
and moreover 
$$\left| \int_M \left( f'-\frac{n}{2} \right) d\nu_{-1} - W \right| \leq \int_M |q'-q|d\nu_{-1} \leq \Psi(\delta |Y,\epsilon)+C(Y,\epsilon)Z e^{-\frac{\epsilon Z}{40}},$$
so we can add a small constant to $q$ to obtain $f'$ satisfying $(v)$, without affecting properties $(i)-(iv)$ or (\ref{close})
\end{proof}

In the case where $(x_0,t_0)$ is also almost-static, the scalar and Ricci curvature terms are small, and $4\tau (h-W)$ is a regularization of Bamler's almost-radial function (see Proposition 13.1 of \cite{bamlergen3} when $k=0$). 

\begin{defn} A strong $(\epsilon,r)$-radial function based at $(x_{0},t_{0})\in M\times I$ is a function $h\in C^{\infty}(M\times[t_{0}-\epsilon^{-1}r^{2},t_{0}-\epsilon r^{2}])$
such that if $W:=\mathcal{N}_{x_{0},t_{0}}(r^{2})$, then 
\begin{flalign*} & \hspace{5 mm} (i) \hspace{3 mm} \square q=-2n, &
\end{flalign*}
\begin{flalign*} & \hspace{5 mm} (ii) \hspace{3 mm} r^{-4}\int_{t_{0}-\epsilon^{-1}r^{2}}^{t_{0}-\epsilon r^{2}}\int_{M}\left| |\nabla q|^{2}-4q \right|d\nu_{t}dt\leq\epsilon, &
\end{flalign*}
\begin{flalign*} & \hspace{5 mm} (iii) \hspace{3 mm} r^{-2}\int_{t_{0}-\epsilon^{-1}r^2}^{t_{0}-\epsilon r^2}\int_{M}|\nabla^{2}q-2g|^{2}d\nu_{t}dt\leq\epsilon, &
\end{flalign*}
\begin{flalign*} & \hspace{5 mm} (iv) \hspace{3 mm} \int_{M} q d\nu_{t}= 2n\tau \text{ for all } t\in [t_0-\epsilon^{-1}r^2,t_0-\epsilon r^2]. &
\end{flalign*}
\end{defn}

Given these definitions, we can rephrase Proposition \ref{strongsoliton}, and give a criterion for the existence of strong $(\epsilon,r)$-radial functions. Moreover, we will establish slightly improved estimates, which will be useful for the proof of Theorem \ref{mainthm3}.

\begin{prop} \label{stronger} For any $\epsilon>0$, $Y<\infty$ and $p\in [1,\infty)$, the following holds if $\delta \leq \overline{\delta}(\epsilon,Y,p)$ and $\alpha \leq \overline{\alpha}(\epsilon,Y)$. Suppose $(M^n,(g_t)_{t\in I},(x_0,t_0))$ is a closed, pointed Ricci flow satisfying $\mathcal{N}_{x_0,t_0}(r^2)\geq -Y$. Assume $(x_0,t_0)$ is $(\delta,r)$-selfsimilar and $h$ is a strong $(\delta,r)$-soliton potential. Then
\begin{equation} \label{gradient} \sup_{t\in [t_0-\epsilon^{-1}r^2,t_0-\epsilon r^2]} \int_M |\nabla h|^p d\nu_t \leq C(Y,\epsilon,p) ,\end{equation}
\begin{equation} \label{improved} \int_{t_0-\epsilon^{-1}r^2}^{t_0-\epsilon r^2} \int_M \left(\tau \left| Rc+\nabla^2 h-\frac{1}{2\tau}g \right|^2 +r^{-2}\left| \tau (R+|\nabla h|^2) -(h-W) \right| \right)e^{\alpha f} d\nu_t dt \leq \epsilon .\end{equation}
If in addition $(x_0,t_0)$ is $(\delta,r)$-static, then $q:=4\tau (h-W)$ is a strong $(\epsilon,r)$-radial function satisfying
$$\int_{t_0-\epsilon^{-1}r^2}^{t_0-\epsilon r^2} \int_M \left( r^{-2}|\nabla^2 q-2g|^2 +r^{-4}\left| |\nabla q|^2 -4q \right| \right)e^{\alpha f} d\nu_t dt \leq \epsilon.$$
\end{prop}
\begin{proof} 
By time translation and parabolic rescaling, we can assume $r=1$ and $t_0=0$. 
Fixing $T\in [\epsilon^{-1},\delta^{-1}]$, we use properties $(ii),(v)$ of strong almost-soliton potentials to get
\begin{align*} \int_{-T}^{-\epsilon} \int_M |\nabla q|^2 d\nu_t dt &\leq \int_{-T}^{-\epsilon} \int_M 16\tau \left( \tau(R+|\nabla h|^2) -(h-W) \right) d\nu_t dt + 16T^3 \delta^{-1} +\int_{-T}^{-\epsilon}\int_M 4q d\nu_t dt \\
&\leq 16T\delta + 16T^3 \delta + 8nT^2 \leq 10n T^2
\end{align*}
assuming $\epsilon \leq \overline{\epsilon}$. If we choose $T = 2p\epsilon^{-1}$, then we can therefore find $\widehat{t}\in [2p\epsilon^{-1}-1,2p\epsilon^{-1}]$ such that 
$$\int_M |\nabla q|^2 d\nu_{\widehat{t}} \leq 10nT^2,$$
so the hypercontractivity of the heat kernel (Theorem 12.1 in \cite{bamlergen1}) gives
$$\sup_{t\in [-2\epsilon,-\frac{\epsilon}{2}]} \int_M |\nabla q|^p d\nu_t \leq CT^2 = C(p,\epsilon).$$
By Cauchy's inequality, (\ref{improved}) will follow from
$$\int_{-\epsilon^{-1}}^{-\epsilon} \int_M \left( \tau \left| Rc+\nabla^2 h -\frac{1}{2\tau}g\right|^2 +|\tau(R+|\nabla h|^2)-(h-W)| \right) e^{\alpha f}d\nu_t dt\leq C(Y,\epsilon)$$
if $\alpha \leq \overline{\alpha}(\epsilon,Y)$. Fix a cutoff function $\zeta \in C^{\infty}([-2\epsilon^{-1},-\frac{1}{2}\epsilon])$ such that $\zeta(-2\epsilon^{-1})=\zeta(-\frac{1}{2}\epsilon)=0$, $\zeta|[-\epsilon^{-1},-\epsilon]\equiv 1$, and $|\zeta'|\leq 4\epsilon^{-1}$. We compute (recalling the definition of $w'$ from the proof of Proposition \ref{strongsoliton})
\begin{align*} \frac{d}{dt}\int_M \tau (w'-W)e^{\alpha f} & = \int_M \square \left( \tau (w'-W) \right) e^{\alpha f}d\nu_t -\tau \int_M (w'-W) \square^{\ast} \left( (4\pi \tau )^{-\frac{n}{2}} e^{-(1-\alpha)f} \right) dg_t \\ &= 
2\tau^2 \int_M \left| Rc+\nabla^2 h -\frac{1}{2\tau}g \right|^2 e^{\alpha f} d\nu_t \\&\hspace{6 mm}+\alpha \tau^2 \int_M (w'-W)\left( R+(1-\alpha)|\nabla f|^2 -\frac{n}{2\tau} \right)e^{\alpha f}d\nu_t.
\end{align*}
Multiplying both sides by $\zeta$ and integrating, then rearranging gives (assuming $\alpha \leq \overline{\alpha}(Y,\epsilon)$)
\begin{align*} \int_{-\epsilon^{-1}}^{-\epsilon}\int_M \tau^2 \left| Rc+\nabla^2 h-\frac{1}{2\tau}g \right|^2 e^{\alpha f}d\nu_t
\leq & C(Y,\epsilon)\left( \int_{-2\epsilon^{-1}}^{-\frac{1}{2}\epsilon} \int_M |w'-W|^2 d\nu_t dt \right)^\frac{1}{2} \\ &\hspace{-30 mm} +C(Y,\epsilon)\left( \int_{-2\epsilon^{-1}}^{-\frac{1}{2}\epsilon} \int_M |w'-W|^2 d\nu_t dt \right)^{\frac{1}{2}}\left( \int_{-2\epsilon^{-1}}^{-\frac{1}{2}\epsilon} \int_M (R^2 +|\nabla f|^4 +1)e^{2 \alpha f}d\nu_t dt \right)^{\frac{1}{2}}.
\end{align*}
The $L^2$ Poincare inequality and property $(iv)$ of Definition 4.3 give
\begin{align*} \int_{-2\epsilon^{-1}}^{-\frac{1}{2}\epsilon} \int_M |w'-W|^2 d\nu_t dt &\leq C(Y,\epsilon) \int_{-2\epsilon^{-1}}^{-\frac{1}{2}\epsilon} \int_M \left(R^2 +|\nabla h|^4 +2|\nabla ^2 h|^2 +h^2 +1 \right) d\nu_t dt \\ &
\leq C(Y,\epsilon)\int_{-2\epsilon^{-1}}^{-\frac{1}{2}\epsilon} \int_M \left( |Rc|^2 + |\nabla h|^4 +1 \right)  d\nu_t dt.
\end{align*}
The $L^4$ estimate for $|\nabla h|$ is a consequence of (\ref{gradient}), while the Ricci curvature is bounded using (\ref{spacetime}).

Now suppose $(x_0,0)$ is also $(\delta,1)$-static. Then $q$ clearly satisfies properties $(i),(iv)$, while $(iii),(iv)$ follow from combining properties $(iii),(ii)$, respectively, of strong almost-soliton potentials with the almost-static inequalities. The remaining inequality follows from the improved estimate for strong almost-soliton potentials and the estimate
$$\int_{-2\epsilon^{-1}}^{-\frac{1}{2}\epsilon} \int_M |Rc|^2 e^{\alpha f}d\nu_t +\sup_{t\in [-2\epsilon^{-1},-\frac{1}{2}\epsilon]} \int_M R e^{\alpha f} d\nu_t \leq \Psi(\delta|Y,\epsilon),$$
which itself follows from Cauchy's inequality, the almost-static inequalities, and (\ref{spacetime}),(\ref{timeslice}).
\end{proof}

\begin{rem} It is also possible to construct regularized versions of Bamler's almost radial functions (c.f. Section 13 of \cite{bamlergen3}) when $k>0$ near points which are almost-selfsimilar, almost-static, and almost-split, but we will not need this.
\end{rem}

\section{Improved Splitting for Noncollapsed K\"ahler-Ricci Flows}

Near a point which is almost-selfsimilar and almost-split, we obtain estimates on the Ricci curvature in the direction of the almost-splitting. 
\begin{lem} \label{tech}
Suppose $(M^n,(g_t)_{t\in I},(x,0))$ is a closed, pointed Ricci flow, $\mathcal{N}_{x,0}(1)\geq -Y$, $y\in C^{\infty}(M\times[-\delta^{-1},-\delta])$ is a strong
$(1,\delta,1)$-splitting map, and $(x,0)$ is $(\delta,1)$-selfsimilar. Then 
\[
\int_{-\epsilon^{-1}}^{-\epsilon}\int_{M}|Rc(\nabla y)|d\nu_{t}dt\leq\Psi(\delta|\epsilon,Y).
\]
\end{lem}

\begin{proof}
Suppose by way of contradiction there exist $\epsilon>0$, $Y<\infty$,
a sequence $\delta_{i}\searrow0$ and closed Ricci flows $(M_{i},(g_{i,t})_{t\in(-\delta_{i}^{-1},0]})$
along with $(\delta_{i},1)$-selfsimilar points $(x_{i},0)$ and strong
$(1,\delta_{i},1)$-splitting maps $y_{i}\in C^{\infty}(M_{i}\times[-\delta_{i}^{-1},-\delta_{i}])$ based at $(x_i,0)$ such that 
\[
\liminf_{i\to\infty}\int_{-\epsilon^{-1}}^{-\epsilon}\int_{M}|Rc_{g_{i}}(\nabla y_{i})|d\nu_{t}^{i}dt>0,
\]
where $\nu^{i}$ is the conjugate heat kernel of $(M_{i},(g_{i,t})_{t\in[-\delta_{i}^{-1},0]})$
based at $(x_{i},0)$. By passing to a subsequence, we can assume $\mathbb{F}$-convergence
\[
(M_{i},(g_{i,t})_{t\in[-\delta_{i}^{-1},0]},(\nu_{t}^{i})_{t\in[-\delta_{i}^{-1},0]})\xrightarrow[i\to\infty]{\mathbb{F}}(\mathcal{X},(\mu_{t})_{t\in(-\infty,0]})
\]
on compact time intervals, where $\mathcal{X}$ is a future-continuous metric soliton; moreover, $d\mu_{t}=(4\pi\tau)^{-\frac{n}{2}}e^{-f}dg_{t}$, on $\mathcal{R}$, where $f\in C^{\infty}(\mathcal{R})$ satisfies $Rc+\nabla^2 f=\frac{1}{2\tau}g$ on $\mathcal{R}$. Let $(U_{i})$ be a precompact exhaustion of $\mathcal{R}$,
and let $\psi_{i}:U_{i}\to M_{i}$ be time-preserving diffeomorphisms
such that $\psi_{i}^{\ast}g_i\to g$ and $\psi_{i}^{\ast}K^{i}(x_{i},0;\cdot,\cdot)\to (4\pi \tau)^{-\frac{n}{2}}e^{-f}$ in $C_{loc}^{\infty}(\mathcal{R})$, where
$K^{i}$ is the conjugate heat kernel of $(M_{i},(g_{i,t})_{t\in(-\delta_{i}^{-1},0]})$.
By Theorem 15.50 of {[}Bam3{]}, we have a splitting of Ricci flow
spacetimes $\mathcal{R}\cong\mathcal{R}'\times\mathbb{R}$ and of
metric flows $\mathcal{X}\cong\mathcal{X}'\times\mathbb{R}$, and
$\psi_{i}^{\ast}y_{i}\to y_{\infty}$ in $C_{loc}^{\infty}(\mathcal{R})$,
where $y_{\infty}:\mathcal{X}\to\mathbb{R}$ is the projection onto
the $\mathbb{R}$-factor. In particular, we have $Rc_{g_{\infty}}(\nabla y_{\infty})=0$,
hence $Rc_{\psi_{i}^{\ast}g_{i}}(\nabla(\psi_{i}^{\ast}y_{i}))\to0$
in $C_{loc}^{\infty}(\mathcal{R})$ as $i\to\infty$. Let $K\subseteq\mathcal{R}_{[-\epsilon^{-1},-\epsilon]}$
be an arbitrary compact subset, and set $K_{t}:=\mathcal{X}_{t}\cap K$,
so that
\begin{align*}
\limsup_{i\to\infty}\int_{-\epsilon^{-1}}^{-\epsilon}\int_{M}|Rc_{g_{i}}(\nabla y_{i})|d\nu_{t}^{i}dt \hspace{-30 mm}& \\
\leq & \limsup_{i\to\infty}\left(\int_{-\epsilon^{-1}}^{-\epsilon}\int_{\psi_{i,t}(K_t)}|Rc_{g_{i}}(\nabla y_{i})|d\nu_{t}^{i}dt+\int_{-\epsilon^{-1}}^{-\epsilon}\int_{M_{i}\setminus\psi_{i,t}(K_{t})}|Rc_{g_{i}}(\nabla y_{i})|d\nu_{t}^{i}dt\right)\\
\leq & \limsup_{i\to\infty}\left(\int_{-\epsilon^{-1}}^{-\epsilon}\int_{K_{t}}|Rc_{\psi_{i}^{\ast}g_{i}}(\nabla(\psi_{i}^{\ast}y_{i}))|\psi_{i}^{\ast}K^{i}(x_{i},0;\cdot,\cdot)d(\psi_{i}^{\ast}g_{i,t})dt\right)\\
 & +\limsup_{i\to\infty}\left(\int_{-\epsilon^{-1}}^{-\epsilon}\int_{M_{i}\setminus\psi_{i,t}(K_{t})}|Rc_{g_{i}}(\nabla y_{i})|^{\frac{3}{2}}d\nu_{t}^{i}dt\right)^{\frac{2}{3}}\left(\int_{-\epsilon^{-1}}^{-\epsilon}\nu_{t}^{i}\left(M_{i}\setminus\psi_{i,t}(K_{t})\right)dt\right)^{\frac{1}{3}}\\
\leq & \limsup_{i\to\infty}\left(\int_{-\epsilon^{-1}}^{-\epsilon}\int_{M_{i}}|Rc_{g_{i}}(\nabla y_{i})|^{\frac{3}{2}}d\nu_{t}^{i}dt\right)^{\frac{2}{3}}\left(\int_{-\epsilon^{-1}}^{-\epsilon}\left(1-\nu_{t}^{i}(\psi_{i,t}(K_{t}))\right)dt\right)^{\frac{1}{3}}\\
\leq & \left(\int_{-\epsilon^{-1}}^{-\epsilon}\left(1-\mu_{t}(K_{t})\right)dt\right)^{\frac{1}{2}}\limsup_{i\to\infty}\left(\int_{-\epsilon^{-1}}^{-\epsilon}\int_{M_{i}}|Rc_{g_{i}}(\nabla y_{i})|^{\frac{3}{2}}d\nu_{t}^{i}dt\right)^{\frac{2}{3}}.
\end{align*}
We then estimate
\[
\int_{-\epsilon^{-1}}^{-\epsilon}\int_{M_{i}}|Rc_{g_{i}}(\nabla y_{i})|^{\frac{3}{2}}d\nu_{t}^{i}dt\leq\left(\int_{-\epsilon^{-1}}^{-\epsilon}\int_{M}|Rc_{g_{i}}|^{2}d\nu_{t}^{i}dt\right)^{\frac{3}{4}}\left(\int_{-\epsilon^{-1}}^{-\epsilon}\int_{M}|\nabla y_{i}|^{4}d\nu_{t}^{i}dt\right)^{\frac{1}{4}}\leq C(Y,\epsilon)
\]
for sufficiently large $i\in\mathbb{N}$. Since $K\subseteq\mathcal{R}_{[-\epsilon^{-1},-\epsilon]}$
was arbitrary, we obtain
\[
\limsup_{i\to\infty}\int_{-\epsilon^{-1}}^{-\epsilon}\int_{M}|Rc_{g_{i}}(\nabla y_{i})|d\nu_{t}^{i}dt=0,
\]
a contradiction. 
\end{proof}

\begin{rem} This proof is easily adapted to show $L^p$ bounds on $Rc(\nabla y)$ for any $p\in [1,2)$, but fails for $p=2$.  This creates additional technical difficulties in the proof of Proposition \ref{meat}.
\end{rem}

Next, we use the estimates for strong almost-soliton potential functions and strong almost-splitting maps to construct new almost splitting maps on a K\"ahler-Ricci flows. We observe that many of these estimates would fail without the use of the parabolic approximation $h$. 

\begin{prop} \label{meat}
Suppose $h\in C^{\infty}(M\times[-\delta^{-1},-\delta])$ is a strong
$(\delta,1)$-soliton potential function at $(x_{0},0)$ satisfying
$$ \int_{-\delta^{-1}}^{-\delta} \int_M |\nabla (h-f)|^2 d\nu_t dt \leq \delta,$$
and assume $y\in C^{\infty}(M\times[-\delta^{-1},-\delta])$
is a strong $(1,\delta,1)$-splitting map, where $(x_{0},0)$ is $(\delta,1)$-selfsimilar.
If $\delta \leq \overline{\delta}(Y,\epsilon,p)$, then the following hold, where $q:=4\tau (h-W)$ and $z:=\frac{1}{2}\langle\nabla q,J\nabla y\rangle$:

\begin{flalign*} & \hspace{5 mm} (i) \hspace{3 mm} \int_{-\epsilon^{-1}}^{-\epsilon}\int_{M}|\nabla^{2}q|^{2}\left( |\nabla y|^{2p} + |\nabla q|^{2p} \right) d\nu_{t}dt \leq C(Y,\epsilon,p) \text{ for each } p\in \mathbb{N},&
\end{flalign*}
\begin{flalign*} & \hspace{5 mm} (ii) \hspace{3 mm}
\int_{-\epsilon^{-1}}^{-\epsilon}\int_{M}\left||\nabla z|^{2}-1\right|d\nu_{t}dt\leq\Psi(\delta|Y,\epsilon),&
\end{flalign*}
\begin{flalign*} & \hspace{5 mm} (iii) \hspace{3 mm}
\int_{-\epsilon^{-1}}^{-\epsilon}\int_{M}\left|\langle\nabla z,\nabla y\rangle\right|d\nu_{t}dt\leq\Psi(\delta|Y,\epsilon),&
\end{flalign*}
\begin{flalign*} & \hspace{5 mm} (iv) \hspace{3 mm} \int_{-\epsilon^{-1}}^{-\epsilon}\int_{M}|\square z|d\nu_{t}dt\leq\Psi(\delta|Y,\epsilon),&
\end{flalign*}
\begin{flalign*} & \hspace{5 mm} (v) \hspace{3 mm} \int_{-\epsilon^{-1}}^{-\epsilon}\int_M |\nabla z -J\nabla y|^2 d\nu_t dt \leq \Psi (\delta |Y,\epsilon).&
\end{flalign*}

\noindent In particular, for $\delta\leq\overline{\delta}(Y,\epsilon)$,
$(y,z)$ is a weak $(2,\epsilon,1)$-splitting map. 
\end{prop}

\begin{proof}
$(i)$ Observe that 
$$\square |\nabla q|^{2(p+1)} \leq (p+1)|\nabla q|^{2p}\square |\nabla q|^2 =-2(p+1)|\nabla q|^{2p}|\nabla^2 q|^2 .$$
Upon integration, (\ref{gradient}) lets us estimate
$$2(p+1)\int_{-\epsilon^{-1}}^{-\epsilon} \int_M |\nabla^2 q|^2 |\nabla q|^{2p} d\nu_t dt \leq - \left. \int_M |\nabla q|^{2(p+1)}d\nu_t \right|_{t=-2\epsilon^{-1}}^{t=-\frac{1}{2}\epsilon}\leq C(Y,\epsilon,p)$$
assuming $\delta \leq \overline{\delta}(Y,\epsilon,p)$. Next, we estimate
\begin{align*} \square \left( |\nabla y|^{2(p+1)}|\nabla q|^2 \right) &\leq -2(p+1)|\nabla^2 y|^2 |\nabla y|^{2p}|\nabla q|^2 -2|\nabla y|^{2(p+1)}|\nabla^2 q|^2 -2\langle \nabla |\nabla y|^{2(p+1)}, \nabla |\nabla q|^2 \rangle \\& \leq -2|\nabla y|^{2(p+1)}|\nabla^2 q|^2 +8(p+1)|\nabla^2 y|\cdot |\nabla y|^{2p+1}|\nabla^2 q|\cdot |\nabla q|
\end{align*}
Integration on $M\times[-2\epsilon^{-1},-\frac{1}{2}\epsilon]$ then gives
\begin{align*}
\int_{-2\epsilon^{-1}}^{-\frac{1}{2}\epsilon}\int_{M}|\nabla^{2}q|^{2}|\nabla y|^{2p}d\nu_{t}dt\hspace{-30 mm}& \\ 
\leq& \left.\int_{M}|\nabla y|^{2(p+1)}|\nabla q|^{2}d\nu_{t}\right|_{t=-2\epsilon^{-1}}^{t=-\frac{1}{2}\epsilon}\\&+8(p+1)\left(\int_{-2\epsilon^{-1}}^{-\frac{1}{2}\epsilon}\int_{M}|\nabla^{2}y|^{2}|\nabla y|^{2(2p+1)}d\nu_{t}dt\right)^{\frac{1}{2}}\left(\int_{-2\epsilon^{-1}}^{-\frac{1}{2}\epsilon}\int_{M}|\nabla^{2}q|^{2}|\nabla q|^{2}d\nu_{t}dt\right)^{\frac{1}{2}}\\
\leq & C(Y,\epsilon,p)
\end{align*}
assuming $\delta\leq\overline{\delta}(Y,\epsilon,p)$, where we used Bamler's estimates for strong almost-splitting maps (Proposition 12.21 of {[}Bam3{]}).

$(ii)$ For any $t\in [-2\epsilon^{-1},-\frac{1}{2}\epsilon]$, 
\begin{align*}
\int_{M}\left(\langle\nabla q,\nabla y\rangle-2y\right)d\nu_{t}= & -(4\pi\tau)^{-\frac{n}{2}}4\tau\int_{M}\langle\nabla e^{-f},\nabla y\rangle dg_{t}+4\tau\int_{M}\langle\nabla(h-f),\nabla y\rangle d\nu_{t}\\
= & 4\tau\int_{M}(\Delta y)d\nu_{t}+4\tau\int_{M}\langle\nabla(h-f),\nabla y\rangle d\nu_{t},
\end{align*}
so we can estimate
\begin{align*}
\int_{-2\epsilon^{-1}}^{-\frac{1}{2}\epsilon}\left|\int_{M}\left(\langle\nabla q,\nabla y\rangle-2y\right)d\nu_{t}\right| dt \leq & 4\tau n\int_{-2\epsilon^{-1}}^{-\frac{1}{2}\epsilon}\int_{M}|\nabla^{2}y|^{2}d\nu_{t}dt\\&+4\tau\left(\int_{-2\epsilon^{-1}}^{-\frac{1}{2}\epsilon}\int_{M}|\nabla(h-f)|^{2}d\nu_{t}dt\right)^{\frac{1}{2}}\left(\int_{-2\epsilon^{-1}}^{-\frac{1}{2}\epsilon}\int_{M}|\nabla y|^{2}d\nu_{t}dt\right)^{\frac{1}{2}}\\
\leq & \Psi(\delta|Y,\epsilon).
\end{align*}
Using the $L^{1}$-Poincare inequality and Lemma \ref{tech}, \begin{align*}
\int_{-2\epsilon^{-1}}^{-\frac{1}{2}\epsilon}\int_{M}\left|\langle\nabla q,\nabla y\rangle-2y\right|d\nu_{t}dt  \hspace{-35 mm}& \\ \leq & \int_{-2\epsilon^{-1}}^{-\frac{1}{2}\epsilon}\left| \int_{M}\left(\langle\nabla q,\nabla y\rangle-2y\right)d\nu_{t}\right| dt +C\int_{-2\epsilon^{-1}}^{-\frac{1}{2}\epsilon}\tau\int_{M}\left|\nabla\left(\langle\nabla q,\nabla y\rangle-2y\right)\right|d\nu_{t}dt\\
\leq & \Psi(\delta|Y,\epsilon)+C(\epsilon)\int_{-2\epsilon^{-1}}^{-\frac{1}{2}\epsilon}\int_{M}\left|(4\tau Rc+\nabla^{2}q-2g)(\nabla y)\right|d\nu_{t}dt +C(\epsilon)\int_{-2\epsilon^{-1}}^{-\frac{1}{2}\epsilon}\int_{M}|Rc(\nabla y)|d\nu_{t}dt\\
 & +C(\epsilon)\left(\int_{-2\epsilon^{-1}}^{-\frac{1}{2}\epsilon}\int_{M}|\nabla^{2}y|^{2}d\nu_{t}dt\right)^{\frac{1}{2}}\left(\int_{-2\epsilon^{-1}}^{-\frac{1}{2}\epsilon}\int_{M}|\nabla q|^{2}d\nu_{t}dt\right)^{\frac{1}{2}}\\
\leq & \Psi(\delta|Y,\epsilon)+C(\epsilon)\left(\int_{-2\epsilon^{-1}}^{-\frac{1}{2}\epsilon}\int_{M}|4\tau Rc+\nabla^{2}q-2g|^{2}d\nu_{t}dt\right)^{\frac{1}{2}}\left(\int_{-2\epsilon^{-1}}^{-\frac{1}{2}\epsilon}\int_{M}|\nabla y|^{2}d\nu_{t}dt\right)^{\frac{1}{2}}\\
\leq & \Psi(\delta|Y,\epsilon).
\end{align*}
Then H\"older's inequality and (\ref{gradient}) give
\begin{align*}
\int_{-2\epsilon^{-1}}^{-\frac{1}{2}\epsilon}\int_{M}\left(\langle\nabla q,\nabla y\rangle-2y\right)^{2}d\nu_{t}dt \hspace{-20 mm}& \\ \leq& \left(\int_{-2\epsilon^{-1}}^{-\frac{1}{2}\epsilon}\int_{M}\left|\langle\nabla q,\nabla y\rangle-2y\right|d\nu_{t}dt\right)^{\frac{1}{2}}\left(\int_{-2\epsilon^{-1}}^{-\frac{1}{2}\epsilon}\int_{M}\left|\langle\nabla q,\nabla y\rangle-2y\right|^{3}d\nu_{t}dt\right)^{\frac{1}{2}}\\
\leq & \Psi(\delta|Y,\epsilon)\left(\int_{-2\epsilon^{-1}}^{-\frac{1}{2}\epsilon}\int_{M}(|\nabla q|^{3}|\nabla y|^{3}+|y|^{3})d\nu_{t}dt\right)^{\frac{1}{2}}\\
\leq & \Psi(\delta|Y,\epsilon).
\end{align*}
Next, we compute $$
\square\left(\langle\nabla q,\nabla y\rangle-2y\right)^{2}=-2\left|\nabla\left(\langle\nabla q,\nabla y\rangle-2y\right)\right|^{2}-4\left(\langle\nabla q,\nabla y\rangle-2y\right)\langle\nabla^{2}q,\nabla^{2}y\rangle. $$
Fix a cutoff function $\zeta\in C^{\infty}([-2\epsilon^{-1},-\frac{1}{2}\epsilon])$
such that $\zeta(-2\epsilon^{-1})=\zeta(-\frac{1}{2}\epsilon)=0$,
$\zeta|[-\epsilon^{-1},-\epsilon]\equiv1$, and $|\zeta'|\leq4\epsilon^{-1}$.
Then 
\begin{align*}
\int_{-2\epsilon^{-1}}^{-\frac{1}{2}\epsilon}\int_{M}\zeta'\left(\langle\nabla q,\nabla y\rangle-2y\right)^{2}d\nu_{t}dt=&-  \int_{-2\epsilon^{-1}}^{-\frac{1}{2}\epsilon}\zeta\left(\frac{d}{dt}\int_{M}\left(\langle\nabla q,\nabla y\rangle-2y\right)^{2}d\nu_{t}\right)dt\\
=&  2\int_{-2\epsilon^{-1}}^{-\frac{1}{2}\epsilon}\int_{M}\zeta\left|\nabla\left(\langle\nabla q,\nabla y\rangle-2y\right)\right|^{2}d\nu_{t}dt\\
 & +4\int_{-2\epsilon^{-1}}^{-\frac{1}{2}\epsilon}\int_{M}\zeta\left(\langle\nabla q,\nabla y\rangle-2y\right)\langle\nabla^{2}q,\nabla^{2}y\rangle d\nu_{t}dt,
\end{align*}
so that (using part $(i)$ and Proposition 12.21 of \cite{bamlergen3})
\begin{align*}
\int_{-\epsilon^{-1}}^{-\epsilon}\int_{M}\left|\nabla\left(\langle\nabla q,\nabla y\rangle-2y\right)\right|^{2}d\nu_{t}dt \hspace{-30 mm}&\\ \leq & 4\epsilon^{-1}\int_{-2\epsilon^{-1}}^{-\frac{1}{2}\epsilon}\int_{M}\left(\langle\nabla q,\nabla y\rangle-2y\right)^{2}d\nu_{t}dt +4\int_{-2\epsilon^{-1}}^{-\frac{1}{2}\epsilon}\int_{M}|\nabla q||\nabla^{2}q|\cdot|\nabla y||\nabla^{2}y|d\nu_{t}dt\\
 & +4\int_{-2\epsilon^{-1}}^{-\frac{1}{2}\epsilon}\int_{M}|\nabla^{2}q|\cdot|\nabla^{2}y||y|d\nu_{t}dt\\
\leq & \Psi(\delta|Y,\epsilon)+4\left(\int_{-2\epsilon^{-1}}^{-\frac{1}{2}\epsilon}\int_{M}|\nabla^{2}q|^{2}|\nabla q|^{2}d\nu_{t}dt\right)^{\frac{1}{2}}\left(\int_{-2\epsilon^{-1}}^{-\frac{1}{2}\epsilon}\int_{M}|\nabla^{2}y|^{2}|\nabla y|^{2}d\nu_{t}dt\right)^{\frac{1}{2}}\\
 & +4\left(\int_{-2\epsilon^{-1}}^{-\frac{1}{2}\epsilon}\int_{M}|\nabla^{2}q|^{2}d\nu_{t}dt\right)^{\frac{1}{2}}\left(\int_{-2\epsilon^{-1}}^{-\frac{1}{2}\epsilon}\int_{M}|\nabla^{2}y|^{2}|y|^{2}d\nu_{t}dt\right)^{\frac{1}{2}}\\
\le & \Psi(\delta|Y,\epsilon).
\end{align*}
Integrating
$$\square \left( |\nabla y|^2 |\nabla q|^4 \right) \leq -2|\nabla^2 y|^2 |\nabla q|^4 -16\langle \nabla^2 y(\nabla y),\nabla^2 q(\nabla q)\rangle |\nabla q|^2$$
against the conjugate heat kernel, and applying part $(i)$ and (\ref{gradient}) gives
\begin{align*} \int_{-\epsilon^{-1}}^{-\epsilon} \int_M |\nabla^2 y|^2 |\nabla q|^4 d\nu_t dt \leq & \left| \left. \int_M |\nabla y|^2 |\nabla q|^2 d\nu_t \right|_{t=-\epsilon^{-1}}^{t=-\epsilon} \right| \\&+ 8 \left( \int_{-\epsilon^{-1}}^{-\epsilon} \int_M |\nabla^2 y|^2 |\nabla y|^2 d\nu_t dt \right)^{\frac{1}{2}} \left( \int_{-\epsilon^{-1}}^{-\epsilon} \int_M |\nabla^2 q|^2 |\nabla q|^6 d\nu_t dt \right)^{\frac{1}{2}} \\ \leq & C(Y,\epsilon)
\end{align*}
assuming $\delta \leq \overline{\delta}(Y,\epsilon)$. We can use H\"older's inequality to estimate
\begin{align*} \int_{-\epsilon^{-1}}^{-\epsilon} \int_M \left| (\nabla^2 q-2g)(\nabla y)\right|^2 d\nu_t dt \hspace{-20 mm} &\\
&\leq  2\int_{-\epsilon^{-1}}^{-\epsilon}\int_M \left| \nabla (\langle \nabla q,\nabla y\rangle -2y)\right|^2 d\nu_t dt +2\int_{-\epsilon^{-1}}^{-\epsilon}\int_M |\nabla^2 y|^2 |\nabla q|^2 d\nu_t dt\\
&\leq \Psi(\delta|Y,\epsilon)+\left( \int_{-\epsilon^{-1}}^{-\epsilon}|\nabla^2 y|^2 |\nabla q|^4 d\nu_t dt\right)^{\frac{1}{2}}\left( \int_{-\epsilon^{-1}}^{-\epsilon}\int_M |\nabla^2 y|^2 d\nu_t dt\right)^{\frac{1}{2}}\\
&\leq \Psi(\delta|Y,\epsilon)
\end{align*}
and we also obtain the rough estimate
$$\int_{-\epsilon^{-1}}^{-\epsilon} \int_M \left| (\nabla^2 q+2g)(\nabla y)\right|^2 d\nu_t dt \leq 2\int_{-\epsilon^{-1}}^{-\epsilon} \int_M \left( |\nabla^2 q|^2 |\nabla y|^2 +4|\nabla y|^2 \right) d\nu_t dt \leq C(Y,\epsilon).$$
Now use H\"older's inequality, and combine estimates:
\begin{align*} \int_{-\epsilon^{-1}}^{-\epsilon}\int_{M}\left||\nabla^{2}q(\nabla y)|^2-4\right|d\nu_{t}dt \hspace{-30 mm} & \\
\leq &  \int_{-\epsilon^{-1}}^{-\epsilon}\int_M \left\langle (\nabla^2 q+2g)(\nabla y),(\nabla^2 q-2g)(\nabla y)\right\rangle d\nu_t dt+ 4\int_{-\epsilon^{-1}}^{-\epsilon} \int_M \left| 1-|\nabla y|^2 \right| d\nu_t dt \\
\leq & \Psi(\delta|Y,\epsilon).
\end{align*}
On the other hand, we can estimate
\begin{align*}
\int_{-\epsilon^{-1}}^{-\epsilon}\int_{M}\left||\nabla^{2}q(\nabla y)|^{2}-|\nabla^{2}q(J\nabla y)|^{2}\right|d\nu_{t}dt \hspace{-30 mm}&
\\ \leq&\left(\int_{-\epsilon^{-1}}^{-\epsilon}\int_{M}\left|\left|\nabla^{2}q\left(\frac{\nabla y}{\sqrt{1+|\nabla y|^2}}\right)\right|^{2}-\left|\nabla^{2}q\left(J\frac{\nabla y}{\sqrt{1+|\nabla y|^2}}\right)\right|^{2}\right|d\nu_{t}dt\right)^{\frac{1}{2}}\\
&\times\left(\int_{-\epsilon^{-1}}^{-\epsilon}\int_{M}\left||\nabla^{2}q(\nabla y)|^{2}-|\nabla^{2}q(J\nabla y)|^{2}\right|\left(1+|\nabla y|^{2}\right)d\nu_{t}dt\right)^{\frac{1}{2}}.
\end{align*}
To estimate the first integral, we observe that
\begin{align*}
\left|\left|\nabla^{2}q\left(\nabla y\right)\right|^{2}-\left|\nabla^{2}q\left(J\nabla y\right)\right|^{2}\right| \hspace{-20 mm}&\\ =&\left|\left(\left|\nabla^{2}q\left(\nabla y\right)\right|^{2}-|(4\tau Rc-2g)(\nabla y)|^{2}\right)-\left(\left|\nabla^{2}q\left(J\nabla y\right)\right|^{2}-|(4\tau Rc-2g)(J\nabla y)|^{2}\right)\right|\\
\leq & \left|\left\langle \left(-4\tau Rc+\nabla^{2}q+2g\right)(\nabla y),\left(4\tau Rc+\nabla^{2}q-2g\right)(\nabla y)\right\rangle \right|\\
 & +\left|\left\langle \left(-4\tau Rc+\nabla^{2}q+2g\right)(J\nabla y),\left(4\tau Rc+\nabla^{2}q-2g\right)(J\nabla y)\right\rangle \right|,
\end{align*}
so that 
\begin{align*} \int_{-\epsilon^{-1}}^{-\epsilon}\int_{M}\left|\left|\nabla^{2}q\left(\frac{\nabla y}{\sqrt{1+|\nabla y|^2}}\right)\right|^{2}-\left|\nabla^{2}q\left(J\frac{\nabla y}{\sqrt{1+|\nabla y|^2}}\right)\right|^{2}\right|d\nu_{t}dt \hspace{-80 mm} \\
\leq& 2\left(\int_{-\epsilon^{-1}}^{-\epsilon}\int_{M}|4\tau Rc+\nabla^{2}q-2g|^{2}d\nu_{t}dt\right)^{\frac{1}{2}} \left(\int_{-\epsilon^{-1}}^{-\epsilon}\int_{M}|4\tau Rc-\nabla^{2}q-2g|^{2}d\nu_{t}dt\right)^{\frac{1}{2}}\\
\leq & \Psi(\delta|Y,\epsilon). \end{align*}
For the second integral, we only need the course upper bound 
\begin{align*} \int_{-\epsilon^{-1}}^{-\epsilon}\int_{M}\left||\nabla^{2}q(\nabla y)|^{2}-|\nabla^{2}q(J\nabla y)|^{2}\right|\left(1+|\nabla y|^{2}\right)d\nu_{t}dt&\leq\int_{-\epsilon^{-1}}^{-\epsilon}\int_{M}4|\nabla^{2}q|^{2}\left(|\nabla y|^{4}+|\nabla y|^2 \right)d\nu_{t}dt\\
&\leq C(Y,\epsilon)\end{align*}
by part $(i)$. Combining expressions, we finally obtain
\begin{align*} \int_{-\epsilon^{-1}}^{-\epsilon} \int_M \left| |\nabla z|^2-1 \right| d\nu_t dt \leq &  \frac{1}{4}\int_{-\epsilon^{-1}}^{-\epsilon} \int_M \left| |\nabla^2 q(J\nabla y)|^2-4 \right| d\nu_t dt +\frac{1}{4} \int_{-\epsilon^{-1}}^{-\epsilon} \int_M |\nabla^2 y|^2 |\nabla q|^2 d\nu_t dt \\ &+ \frac{1}{2}\left( \int_{-\epsilon^{-1}}^{-\epsilon} \int_M |\nabla^2 q|^2 |\nabla q|^2 d\nu_t dt \right)^{\frac{1}{2}}\left( \int_{-\epsilon^{-1}}^{-\epsilon} \int_M |\nabla^2 y|^2 |\nabla q|^2 d\nu_t dt \right)^{\frac{1}{2}}\\ \leq & \Psi(\delta|Y,\epsilon).
\end{align*}

\noindent $(iii)$ Because $X\mapsto Rc(JX,X)$ and $X\mapsto g(JX,X)$
are skew-symmetric, we have
\begin{align*}
\langle\nabla z,\nabla y\rangle= & \frac{1}{2}\nabla^{2}q(J\nabla y,\nabla y)+\frac{1}{2}\langle\nabla_{\nabla y}J\nabla y,\nabla q\rangle\\
= & \frac{1}{2}\left(4\tau Rc+\nabla^{2}q-2g\right)(J\nabla y,\nabla y)-\frac{1}{2}\nabla^{2}y(\nabla y,J\nabla q),\end{align*}
which allows us to estimate (again using Proposition 12.21 of \cite{bamlergen3})
\begin{align*} \int_{-\epsilon^{-1}}^{-\epsilon}\int_{M}\left|\langle\nabla z,\nabla y\rangle\right|d\nu_{t}dt \leq & \int_{-\epsilon^{-1}}^{-\epsilon}\int_{M}\left|\langle4\tau Rc+\nabla^{2}q-2g,\nabla y\otimes J\nabla y\rangle\right|d\nu_{t}dt\\
 & +\int_{-\epsilon^{-1}}^{-\epsilon}\int_{M}|\nabla^{2}y|\cdot|\nabla y|\cdot|\nabla q|d\nu_{t}dt\\
\leq & \left(\int_{-\epsilon^{-1}}^{-\epsilon}\int_{M}\left|4\tau Rc+\nabla^{2}q-2g\right|^{2}d\nu_{t}dt\right)^{\frac{1}{2}}\left(\int_{-\epsilon^{-1}}^{-\epsilon}\int_{M}|\nabla y|^{4}d\nu_{t}dt\right)^{\frac{1}{2}}\\
 & +\left(\int_{-\epsilon^{-1}}^{-\epsilon}\int_{M}|\nabla^{2}y|^{2}|\nabla y|^{2}d\nu_{t}dt\right)^{\frac{1}{2}}\left(\int_{-\epsilon^{-1}}^{-\epsilon}\int_{M}|\nabla q|^{2}d\nu_{t}dt\right)^{\frac{1}{2}}\\
\leq & \Psi(\delta|Y,\epsilon).
\end{align*}
$(iv)$ We compute
\begin{align*}
\square\langle\nabla q,J\nabla y\rangle= & 2Rc(\nabla q,J\nabla y)+\langle\square\nabla q,J\nabla y\rangle+\langle\nabla q,J\square\nabla y\rangle-\langle\nabla^{2}q,\nabla(J\nabla y)\rangle\\
= & 2Rc(\nabla q,J\nabla y)-\langle Rc(\nabla q),J\nabla y\rangle-\langle\nabla q,JRc(\nabla y)\rangle-\langle\nabla^{2}q,\nabla(J\nabla y)\rangle\\
= & -\langle\nabla^{2}q,\nabla(J\nabla y)\rangle,
\end{align*}
so that 
$$\int_{-\epsilon^{-1}}^{-\epsilon}\int_{M}|\square z|d\nu_{t}dt\leq\left(\int_{-\epsilon^{-1}}^{-\epsilon}\int_{M}|\nabla^{2}q|^{2}d\nu_{t}dt\right)^{\frac{1}{2}}\left(\int_{-\epsilon^{-1}}^{-\epsilon}\int_{M}|\nabla^{2}y|^{2}d\nu_{t}dt\right)^{\frac{1}{2}}\leq\Psi(\delta|Y,\epsilon).$$
$(v)$ Using part $(ii)$ and Lemma \ref{tech}, we estimate
\begin{align*} \int_{-\epsilon^{-1}}^{-\epsilon} \int_M \left| \nabla z-J\nabla y \right|^2 d\nu_t dt \hspace{-30 mm}& \\
&= \int_{-\epsilon^{-1}}^{-\epsilon} \int_M \left( |\nabla z| ^2 + |\nabla y|^2 - 2\langle \nabla z,J\nabla y\rangle \right)d\nu_t dt\\ & \leq \int_{-\epsilon^{-1}}^{-\epsilon} \int_M \left( 2 - \nabla^2 q(J\nabla y,J\nabla y) \right) d\nu_t dt + \int_{-\epsilon^{-1}}^{-\epsilon} \int_M |\nabla^2 y| \cdot |\nabla y|\cdot |\nabla q| d\nu_t dt +\Psi(\delta |Y,\epsilon) \\
&\leq  \int_{-\epsilon^{-1}}^{-\epsilon} \int_M \left| 4\tau Rc+\nabla^2 q-2g \right| \cdot |\nabla y|^2 d\nu_t dt \\ & \hspace{4 mm}+ \int_{-\epsilon^{-1}}^{-\epsilon}\int_M |Rc(\nabla y)|\cdot |\nabla y| d\nu_t dt +\Psi(\delta |Y,\epsilon) \\ &\leq \int_{-\epsilon^{-1}}^{-\epsilon}\int_M |Rc(\nabla y)|\cdot \left| |\nabla y| - 1 \right| d\nu_t dt + \int_{-\epsilon^{-1}}^{-\epsilon} \int_M |Rc(\nabla y)|d\nu_t dt +\Psi(\delta|Y,\epsilon) \\ &\leq \Psi(\delta |Y,\epsilon).
\end{align*}
\end{proof}

Next, we prove an elementary lemma which will allow us to form almost splitting maps using a linear combination of almost splitting maps along with the new almost-splitting maps constructed in Lemma \ref{meat}.

\begin{lem} \label{vectorspace} Given $N,k\in \mathbb{N}$, there exists $C=C(N)<\infty$ such that the following holds. Suppose $(\mathcal{V},\langle \cdot, \cdot \rangle )$ is a real inner product space of dimension at most $N$, and let $J$ be a complex structure on $\mathcal{V}$ compatible with the inner product: $\langle Jv,Jw\rangle =\langle v,w\rangle$ for all $v,w\in \mathcal{V}$. If $v_1,...,v_{2k+1}\in \mathcal{V}$ are orthonormal, then there are $(a_{ij})_{1\leq i\leq 2k+2}^{1\leq j \leq 2k+1}$ and $(b_{ij})_{1\leq i\leq 2k+2}^{1\leq j\leq 2k+1}$ with $|a_{ij}|+|b_{ij}|\leq C$ such that 
$$\widetilde{v}_i := \sum_{j=1}^{2k+1} \left(a_{ij}v_j + b_{ij}Jv_j\right), \hspace{6 mm} 1\leq i \leq 2k+2,$$
are orthonormal.
\end{lem}
\begin{proof} It suffices to show the existence of $c(n)>0$ such that for any orthonormal tuple $(v_1,...,v_{2k+1})$ in $\mathcal{V}$, there exists $i\in \{1,...,2k+1\}$ such that
$$w_i:= Jv_i -\sum_{j=1}^{2k+1}\langle Jv_i,v_j \rangle v_j$$
satisfies $|w_i|\geq c(n)$. Suppose by way of contradiction there exist orthonormal tuples $(v_1^{\alpha} ,...,v_{2k+1}^{\alpha})_{\alpha \in \mathbb{N}}$ such that $\max_{1\leq i\leq 2k+1}|w_i^{\alpha}| \to 0$ as $\alpha\to \infty$. We can pass to subsequences so that $\lim_{\alpha \to \infty} v_i^{\alpha}=v_i^{\infty}$, where $(v_1^{\infty},...,v_{2k+1}^{\infty})$ is an orthonormal tuple. Then, for each $i\in \{1,...,2k+1\}$, we have
$$0=\lim_{\alpha \to \infty} w_i^{\alpha} = Jv_i^{\infty} -\sum_{j=1}^{2k+1}\langle Jv_i^{\infty} ,v_j^{\infty} \rangle v_j^{\infty}.$$
That is, $Jv_1,...,Jv_{2k+1}$ are in the $\mathbb{R}$-linear span of $v_1,...,v_{2k+1}$. This means that the $\mathbb{R}$-linear span of $v_1,...,v_{2k+1}$, equipped with the restriction of the complex structure, is a complex vector space of real dimension $2k+1$, a contradiction. 
\end{proof}

We finally establish the improved splitting for K\"ahler-Ricci flows.

\begin{prop} \label{quantthm1} For any $\epsilon>0$, $k\in \{0,...,n\}$, $Y<\infty$, the following holds whenever $\delta \leq \overline{\delta}(\epsilon,Y)$. Suppose $(M^{2n},(g_t)_{t\in I})$ is a K\"ahler-Ricci flow with $\mathcal{N}_{x_0,t_0}(r^2)\geq -Y$ for some $r>0$. If $(x_0,t_0)\in M\times I$ is $(\delta,r)$-selfsimilar and strongly $(2k+1,\delta,r)$-split, then $(x_0,t_0)$ is weakly $(2k+2,\epsilon,r)$-split. 
\end{prop}
\begin{proof} By parabolic rescaling and time translation, we can assume $t_0=0$ and $r=1$. For ease of notation, write $\nu_t := \nu_{x_0,0;t}$. Let $y:M\times [-\delta^{-1},\delta]\to \mathbb{R}^{2k+1}$ be a strong $(2k+1,\delta,1)$-splitting map, and let $q\in C^{\infty}(M\times [-\delta^{-1},-\delta])$ be a strong $(\delta,1)$-soliton potential, both based at $(x_0,0)$. For each $i \in \{1,...,2k+1\}$, Proposition \ref{meat} states that the functions $z_i := \frac{1}{2}\langle \nabla q , J\nabla y_i \rangle$ are weak $(1,\Psi(\delta|Y),1)$-splitting maps based at $(x_0,0)$ which satisfy
$$\int_{-\epsilon^{-1}}^{-\epsilon}\int_M |\nabla z_i -J\nabla y_i|^2 d\nu_t dt \leq \Psi(\delta|Y,\epsilon).$$ 
By replacing $y$ with $A\circ y+b$ for some $A\in \mathbb{R}^{(2k+1)\times (2k+1)}$ with $|A-I_{2k+1}|\leq \Psi(\delta|Y,\epsilon)$ and $|b|\leq \Psi(\delta|Y,\epsilon)$, we can assume that
$$\int_{-\epsilon^{-1}}^{-\epsilon}\int_M \langle \nabla y_i ,\nabla y_j \rangle d\nu_t dt =\delta_{ij}.$$
Similar to the proof of Proposition 10.8 in \cite{bamlergen3}, we consider the finite-dimensional real vector space $\mathcal{V}$ spanned by $\nabla y_i,J\nabla y_i \in \mathfrak{X}(M\times [-\epsilon^{-1},-\epsilon])$, equipped with the restricted $L^2$ inner product
$$\langle X_1 , X_2 \rangle_{L^2} = \int_{-\epsilon^{-1}}^{-\epsilon}\int_M \langle X_1 ,X_2 \rangle d\nu_t dt$$
and the obvious complex structure. Lemma \ref{vectorspace} then provides $a_{ij},b_{ij}$, where $1\leq i \leq 2k+2$, $1\leq j\leq 2k+1$, such that
$$\widetilde{V}_i := \sum_{j=1}^{2k+1}a_{ij}\nabla y_j +b_{ij}J\nabla y_j$$
are orthonormal in $\mathcal{V}$, and $|a_{ij}|,|b_{ij}|\leq C(k)$. It follows that
$$\widetilde{y}_i :=\sum_{j=1}^{2k+1}a_{ij} y_j +b_{ij}z_j$$
define a weak $(2k+2,\Psi(\delta|Y,\epsilon),1)$-splitting map $\widetilde{y}=(\widetilde{y}_1,...,\widetilde{y}_{2k+2})$ based at $(x_0,0)$.
\end{proof}

\begin{proof}[Proof of Theorem \ref{mainthm2}] We first verify the existence of a lower bound of the Nash entropy on all of $P^{\ast}(x_{\infty},A,-A^2)$. Given $y\in P^{\ast}(x_{\infty},A,-A^2)$, we can find a sequence $(y_i,t_i)\in M_i \times (-T_i,0]$ such that $(y_i,t_i)\xrightarrow[i\to\infty]{\mathfrak{C}} y_{\infty}$. Lemma 15.8 of \cite{bamlergen3} implies that $(y_i,t_i)\in P^{\ast}(x_i,2A,-(2A)^2)$, and Bamler's Nash entropy oscillation estimate (Corollary 5.11 in \cite{bamlergen1}) then gives $\mathcal{N}_{y_i,t_i}(1) \geq -Y'(Y,A)$. Taking the limit as $i\to \infty$, we obtain (via Theorem 15.45 in \cite{bamlergen3}) $\mathcal{N}_y(1)\geq -Y'(Y,A)$. The inclusion 
$$\widehat{\mathcal{S}}_{r_1,r_2}^{\epsilon,2j+1} \cap P^{\ast}(x_{\infty};A,-A^2) \subseteq \widehat{\mathcal{S}}_{r_1,r_2}^{\delta(\epsilon,Y,A),2j} \cap P^{\ast}(x_{\infty};A,-A^2)$$
thus follows from Proposition \ref{quantthm1}. The remaining claim is then a consequence of the inclusions between quantitative strata and weak quantitative strata (Lemma 20.3 of \cite{bamlergen3} and Proposition \ref{equiv}).
\end{proof}

\begin{proof}[Proof of Theorem \ref{mainthm1}] Taking $r_1=0$ and $r_2=\epsilon$ in Theorem \ref{mainthm2} gives
$$\widehat{\mathcal{S}}_{0,\epsilon}^{\epsilon,2j+1} \cap P^{\ast}(x_{\infty},A,-A^2) \subseteq \widehat{\mathcal{S}}_{0,\epsilon}^{\delta(\epsilon,Y,A),2j}\cap P^{\ast}(x_{\infty};A,-A^2).$$
Then taking the union over $\epsilon>0$ gives
$$\mathcal{S}^{2j+1}\cap P^{\ast}(x_{\infty},A,-A^2) \subseteq \mathcal{S}^{2j+2} \cap P^{\ast}(x_{\infty},A,-A^2).$$
Finally, taking $A\nearrow \infty$ gives the claim.
\end{proof}

\section{An Isometric Action on Tangent Flows}

Suppose $(M,g,J,f)$ is a (not necessarily complete) shrinking gradient K\"ahler-Ricci soliton, and let $\omega:=g(J\cdot,\cdot)$ be the corresponding K\"ahler form. Then the Ricci soliton equation gives
\begin{align*} (\mathcal{L}_{J\nabla f}J)(W)&=	\mathcal{L}_{J\nabla f}(JW)-J([J\nabla f,W])
=	\nabla_{J\nabla f}(JW)-\nabla_{JW}(J\nabla f)-J\left(\nabla_{J\nabla f}W-\nabla_{W}(J\nabla f)\right)
\\&=	-J\nabla_{JW}\nabla f-\nabla_{W}\nabla f= JRc(JW)+Rc(W)-\frac{1}{2}(J^2 W+W)=0 \end{align*}
for any vector field $W\in \mathfrak{X}(\mathcal{R})$, and
\[
\mathcal{L}_{J\nabla f}\omega=di_{J\nabla f}\omega=d\left(\omega(J\nabla f,\cdot)\right)=-d\left(g(\nabla f,\cdot)\right)=-d(df)=0,
\]
so that $J\nabla f$ is a real holomorphic Killing vector field on $M$.  We now prove the completeness of the flow of this vector field for tangent flows.

\begin{prop} \label{action}
Suppose $(M_{i}^{2n},(g_{i,t})_{t\in[-\epsilon_{i}^{-1},0]})$ are closed K\"ahler-Ricci
flows, and that $(x_{i},0)$ are $(\epsilon_{i},1)$-selfsimilar,
where $\epsilon_{i}\searrow0$.
Assume 
\[
(M,(g_{i,t})_{t\in(-\epsilon_i^{-1},0)},(\nu_{x_i,0;t})_{t\in[-\epsilon_i^{-1},0)})\xrightarrow[i\to\infty]{\mathbb{F}}(\mathcal{X},(\nu_{x_{\infty};t})_{t\in(-\infty,0)})
\]
on compact time intervals
where $\mathcal{X}$ is a metric soliton modeled on a singular shrinking
K\"ahler-Ricci soliton $(X,d,\mathcal{R}_X,g_X,f_X)$ as in Theorem \ref{tangentflows}. Set $q:=4\tau (f_X-W) \in C^{\infty}(\mathcal{R})$, where $\mathcal{R}\subseteq \mathcal{X}$ is the regular part of the metric flow. Then $J\nabla q$ is
complete, and the heat kernel satisfies the following infinitesimal
symmetry for all $(x_{1},x_{0})\in\mathcal{R}\times\mathcal{R}$ with
$\mathfrak{t}(x_{0})<\mathfrak{t}(x_{1})$:

\[
\langle\nabla_{x_{1}}K(x_{1};x_{0}),J\nabla q(x_1) \rangle+\langle\nabla_{x_{0}}K(x_{1};x_{0}),J\nabla q(x_0)\rangle=0.
\]
Moreover, the flow of $J\nabla q$ extends to a 1-parameter action by isometries on all of $X$. 
\end{prop}

\begin{rem}
This proof is modeled on Theorem 15.50 in \cite{bamlergen3}.
\end{rem}

\begin{proof}
Let $(U_{i})$ be a precompact exhaustion of $\mathcal{R},$with open
embeddings $\psi_{i}:U_{i}\to M_{i}\times(-\epsilon_i^{-1},0)$ realizing the
$\mathbb{F}$-convergence on the regular part as in Theorem \ref{bamconvergence}. By Proposition \ref{strongsoliton} and the proof of Theorem 15.69 in \cite{bamlergen3}, we can find almost-GRS potential functions $h_{i}\in C^{\infty}(M_{i}\times[-\epsilon_{i}^{-1},-\epsilon_{i}])$
such that, $h_{i}\circ\psi_{i}\to f_X$ in $C_{loc}^{\infty}(\mathcal{R})$, where we identify $\mathcal{R}\cong \mathcal{R}_X\times(-\infty,0)$ as in Theorem \ref{tangentflows}. Now
fix $t_{0}\in(-\infty,0)$ and $u^{\infty}\in C_{c}^{\infty}(\mathcal{R}_{t_{0}})$,
and let $u^{i}\in C^{\infty}(M_{i}\times\{t_{0}\})$ be an approximating
sequence: $u^{i}\circ\psi_{i}\to u^{\infty}$ in $C_{loc}^{\infty}(\mathcal{R})$.
Let $(\zeta_{h}^{i})_{h\in(-\alpha,\alpha)}$ be the flow of $J_{i}\nabla q_{i}$,
and define $u'^{,i}\in C^{\infty}(M_{i}\times\{t_{0}\}\times(-\alpha,\alpha))$
by $u'^{,i}(x,t_{0},h):=u^{i}(\zeta_{h}^{i}(x),t_{0})$, so that $u'^{,i}(\cdot,t_{0},0)=u^{i}(\cdot,t_{0})$
and $\partial_{h}u'^{,i}(x,t_{0},h)=\langle\nabla u^{i},J_{i}\nabla q_{i}\rangle(\zeta_{h}^{i}(x),t_{0})$.
Next, let $u''^{,i}\in C^{\infty}(M_{i}\times[t_{0},0]\times(-\alpha,\alpha))$
be given by $u''^{,i}(\cdot,t_{0},\cdot)=u'^{,i}$ and $\square u''^{,i}(\cdot,\cdot,h)=0$
for all $h\in(-\alpha,\alpha)$. Letting $\omega_i, \rho_i \in \Omega^2(M_i)$ denote the K\"ahler and Ricci 2-forms, respectively, we have $$\langle\nabla^{2}u''^{,i},g_{i}(J_{i}\cdot,\cdot)\rangle=\langle\nabla^{2}u''^{,i},\omega_{i}\rangle=0,$$
$$\langle \nabla^2 u''^{,i},Rc_{g_i}(J_i \cdot, \cdot)\rangle =\langle \nabla^2 u''^{,i} ,\rho_i \rangle =0,$$
so we can estimate
\begin{align*}
\square\left|\partial_{h}u''^{,i}-\langle\nabla u''^{,i},J_{i}\nabla q_{i}\rangle\right|\leq & 2\left|\langle\nabla^{2}u''^{,i},\nabla(J_{i}\nabla q_{i})\rangle\right|=2\left|\left\langle \nabla^{2}u''^{,i},\left(4\tau Rc(g_i)+\nabla^{2}q_{i}-2g_{i}\right)(J_{i}\cdot,\cdot)\right\rangle \right|\\
\leq & 2|\nabla^{2}u''^{,i}|\cdot|4\tau Rc(g_i)+\nabla^{2}q_{i}-2g_{i}|.
\end{align*}
Let $\nu^{i}:=\nu_{x_{i},0}$, and integrate $\square|\nabla u''^{,i}|^{2}=-2|\nabla^{2}u''^{,i}|^{2}$
against $\nu^{i}$ to obtain
\[
2\int_{t_{0}}^{t_{1}}\int_{M_i}|\nabla^{2}u''^{,i}|^{2}d\nu_{t}^{i}dt\leq\int_{M_i}|\nabla u'^{,i}|^{2}d\nu_{t_{0}}^i.
\]
However, we know that $u'^{,i}\circ\psi_{i}\to u'^{,\infty}$ in $C_{loc}^{\infty}(\mathcal{R}_{t_{0}})$,
where $u'^{,\infty}(x)=u^{\infty}(\zeta_{h}^{\infty}(x))$, and $(\zeta_{h}^{\infty})$
is the (partially defined) flow of $J\nabla q$. In particular, we
can estimate, for any $t_{1}\in(t_{0},0)$,
\begin{align*}
\sup_{t\in[t_{0},t_{1}]}\int_{M_i}\left|\partial_{h}u''^{,i}-\langle\nabla u''^{,i},J_{i}\nabla q_{i}\rangle\right|d\nu_{t}^i\leq & \left(\int_{t_{0}}^{t_{1}}\int_{M_i}|4\tau Rc(g_i )+\nabla^{2}q_{i}-2g_{i}|^{2}d\nu_{t}^{i}dt\right)^{\frac{1}{2}}\left(\int_{M_i}|\nabla u'^{,i}|^{2}d\nu_{t_{0}}^i\right)^{\frac{1}{2}}\\
\leq & \Psi(\epsilon_{i}|t_{0},t_{1}).
\end{align*}
Because $u''^{,i}\circ\psi_{i}\to u''^{,\infty}$ in $C_{loc}^{\infty}(\mathcal{R}_{[t_{0},0)})$,
where $u''^{,\infty}(x,h)=\int_{\mathcal{R}_{t_{0}}}K(x;y)u'(y,h)dg_{t_{0}}(y)$,
we obtain $\partial_{h}u''^{,\infty}=\langle\nabla u''^{,\infty},J\nabla q\rangle$.
We can therefore compute, for all $x\in\mathcal{R}_{t_{1}}$,
\begin{align*}
\partial_{h}|_{h=0}u''^{,\infty}(x,0)= & \left. \frac{\partial}{\partial h} \right|_{h=0}\int_{\mathcal{R}_{t_{0}}}K(x;y)u'(y,h)dg_{t_{0}}(y)=\int_{\mathcal{R}_{t_{0}}}K(x;y)\langle\nabla u,J\nabla q\rangle(y)dg_{t_{0}}(y)\\
= & -\int_{\mathcal{R}_{t_{0}}}u(y)\left(\langle\nabla_{y}K(x;y),J\nabla q(y)\rangle+ K(x;y)\text{div}(J\nabla q)(y)\right)dg_{t_{0}}(y)\\
= & -\int_{\mathcal{R}_{t_{0}}}u(y)\langle\nabla_{y}K(x;y),J\nabla q(y)\rangle dg_{t_{0}}(y)
\end{align*}
since 
$$\text{div}(J\nabla q)=\langle \nabla (J\nabla q),g\rangle =- \langle 2\omega-4\tau \rho,g \rangle =0.$$
On the other hand, we have
\[
\partial_{h}|_{h=0}u''^{,\infty}(x,0)=\langle\nabla u''^{,\infty}(x,0),J\nabla q(x)\rangle=\int_{\mathcal{R}_{t_{0}}}\langle\nabla_{x}K(x;y),J\nabla q(x)\rangle u(y)dg_{t_{0}}(y),
\]
and the infinitesimal symmetry follows.

By Theorem 15.45$(c)$ in \cite{bamlergen3}, any $x\in\mathcal{R}_{t_{1}}$ satisfies $\lim_{\tau'\searrow0}\frac{2}{\tau'}\int_{\frac{\tau'}{2}}^{\tau'}\mathcal{N}_{x}(\tau'')d\tau''=0$.
Suppose $\gamma:I^{\ast}\to\mathcal{R}_{t_{1}}$ is an integral curve
of $J\nabla q$, and fix $\tau'>0$ sufficiently small so that if
$t_{0}:=t_{1}-\tau'$ and $t_{0}':=t_{1}-\frac{1}{2}\tau'$, then
there exists $x_{0}\in\mathcal{R}_{t_{0}}$ which exists until time
$t_{0}'$. Write $K(\gamma(s);\cdot)=(4\pi\tau)^{-\frac{n}{2}}e^{-f_{s}}$,
where $f_s\in C^{\infty}(\mathcal{R}_{[t_{0},t_{0}']})$, $s\in I^{\ast}$.
By Theorem 14.54$(b)$ of \cite{bamlergen3}, the completeness of
$J\nabla q$ will follow from showing the following identity:
\[
\frac{d}{ds}\int_{t_{0}}^{t_{0}'}\int_{\mathcal{R}_{t}}f_{s}e^{-f_{s}}dg_{t}dt=0.
\]
For $r>0$, let $\eta_{r}\in C^{\infty}(\mathcal{R})$ be the cutoff
functions from Lemma \ref{cutoff}. Fix $\delta>0$
and a cutoff $\overline{\eta}_{\delta}\in C^{\infty}([0,\infty))$
with $\overline{\eta}_{\delta}|[0,\delta]\equiv0$ and $\overline{\eta}_{\delta}(a)=a$
for all $a\in[2\delta,\infty)$.\\

\noindent \textbf{Claim:} There exists $A=A(\delta)<\infty$ such that $\text{supp}(\overline{\eta}_{\delta}\circ e^{-f_s})\subseteq P^{\ast}(x_{\infty},A,-A^2)$ for all $s\in I^{\ast}\cap [-\sigma^{-1},\sigma^{-1}]$.

We recall the following Gaussian estimate for the conjugate heat kernel on $\mathcal{X}$ (Lemma 15.9 of \cite{bamlergen3}):
$$e^{-f(y)}\leq C(T) \exp\left( -\frac{\left(d_{W_1}^{\mathcal{X}_{\mathfrak{t}(y)}}(\nu_{x_{\infty};t_0},\delta_y)\right)^2}{10\mathfrak{t}(y)}\right)$$
for all $y\in \mathcal{R}_{[-T,0)}$. We let $y=\gamma(s)$, and observe that $s\mapsto f(\gamma(s))$ is constant, so there exists $\Lambda \in (|t_0|^{\frac{1}{2}},\infty)$ such that 
$$d_{W_1}^{\mathcal{X}_{t_1}}(\nu_{x_{\infty};t_0},\delta_{\gamma(s)}) \leq \Lambda$$
for all $s\in I^{\ast}$. This implies $\gamma(I^{\ast})\subseteq P^{\ast}(x_{\infty},\Lambda,-\Lambda^2)$. We may therefore apply Lemma 15.9 of \cite{bamlergen3} to conclude
$$f_s(y)\geq -C + \frac{1}{10 \tau'}\left( d_{W_1}^{\mathcal{X}_t}(\nu_{\gamma(s);t},\delta_y)\right)^2$$
for any $t\in [t_0,t_0']$ and $y\in \mathcal{R}_t$. Thus, there exists $A'=A'(\delta)<\infty$ such that 
$$\text{supp}(\overline{\eta}_{\delta}\circ e^{-f_s})\subseteq P^{\ast}(\gamma(s),A',-(A')^2)$$
for all $s\in I^{\ast}\cap [-\sigma^{-1},\sigma^{-1}]$. The Claim then follows from Proposition 3.40 of \cite{bamlergen2}, which describes inclusion properties of $P^{\ast}$-parabolic neighborhoods. $\square$

By the Claim and Lemma 2.11$(iv)$, we see that $$\bigcup_{s\in [-\sigma^{-1},\sigma^{-1}]} \text{supp}((\overline{\eta}_{\delta}\circ e^{-f_s})\eta_r) \cap \mathcal{R}_{[t_0,t_0']}$$
is relatively compact in $\mathcal{R}_{[t_0,t_0']}$ for any fixed $\delta, r>0$. Thus, for any $s_{1},s_{2}\in I^{\ast}$, the infinitesimal symmetry of $K$ gives
\begin{align} \nonumber
\left.\int_{t_{0}}^{t_{0}'}\int_{\mathcal{R}_{t}}f_{s}(\overline{\eta}_{\delta}\circ e^{-f_{s}})\eta_{r}dg_{t}dt\right|_{s=s_{1}}^{s=s_{2}}= &- \int_{s_{1}}^{s_{2}}\int_{t_{0}}^{t_{0}'}\int_{\mathcal{R}_{t}}\left\langle J\nabla q,\nabla\left(f_{s}(\overline{\eta}_{\delta}\circ e^{-f_{s}})\right)\right\rangle \eta_{r}dg_{t}dtds\\
\nonumber
= & \int_{s_{1}}^{s_{2}}\int_{t_{0}}^{t_{0}'}\int_{\mathcal{R}_{t}}\left(\left\langle J\nabla q,\nabla\eta_{r}\right\rangle +\text{div}(J\nabla q)\eta_{r}\right)f_{s}(\overline{\eta}_{\delta}\circ e^{-f_{s}})dg_{t}dtds\\
= & \int_{s_{1}}^{s_{2}}\int_{t_{0}}^{t_{0}'}\int_{\mathcal{R}_{t}}\left\langle J\nabla q,\nabla\eta_{r}\right\rangle f_{s}(\overline{\eta}_{\delta}\circ e^{-f_{s}})dg_{t}dtds. \label{RHS}
\end{align}
We recall that $f_{s}$ is bounded uniformly (in $s$) from below on $\mathcal{R}_{[t_0,t_0']}$, so $f_s (\overline{\eta}_{\delta}\circ e^{-f_s})$ is uniformly bounded on $\mathcal{R}_{[t_0,t_0']}$.  We note that
$$\bigcup_{s\in [-\sigma^{-1},\sigma^{-1}]} \text{supp}\left(\overline{\eta}_{\delta}\circ e^{-f_s}\right)\cap \mathcal{R}_{[t_0,t_0']} \subseteq L$$
for any fixed $\delta>0$, where $L \subseteq \mathcal{X}_{[t_0,t_0']}$ is compact. Integrating the estimate $|\nabla \sqrt{f-W}|\leq \frac{1}{2\sqrt{\tau}}$ along almost-minimizing curves in $\mathcal{R}_t$ we obtain $\sup_L |f| <\infty$, and so $\sup_L |\nabla f|<\infty$. Thus, we can bound the right hand side of (\ref{RHS}) by
\[
C\int_{t_{0}}^{t_{0}'}\int_{L\cap\mathcal{R}_{t}}|\nabla f|\cdot|\nabla\eta_{r}|dg_{t}dt\leq C\int_{t_{0}}^{t_{0}'}\int_{L\cap\mathcal{R}_{t}}|\nabla\eta_{r}|dg_{t}dt \leq Cr^2,
\]
where $C<\infty$ is independent of $r>0$, and the last inequality follows from the estimate for $\{r_{Rm} <r\}$ in Lemma 15.27 of \cite{bamlergen3}. We can therefore take $r\to 0$ to obtain
\[
\int_{t_{0}}^{t_{0}'}\int_{\mathcal{R}_{t}}f_{s_{1}}(\overline{\eta}_{\delta}\circ e^{-f_{s_{1}}})dg_{t}dt=\int_{t_{0}}^{t_{0}'}\int_{\mathcal{R}_{t}}f_{s_{2}}(\overline{\eta}_{\delta}\circ e^{-f_{s_{2}}})dg_{t}dt.
\]
Finally, we take $\delta\searrow0$ and appeal to the dominated convergence
theorem to get the desired identity.

Now let $(\phi_s)_{s\in \mathbb{R}}$ be the flow on $\mathcal{R}_X$ generated by $J\nabla q$ (restricted to a time slice of $\mathcal{R}$).
For any $x_{1},x_{2}\in X$, if $\epsilon>0$ and $\gamma:[0,1]\to X$
is a curve with image in the regular set $\mathcal{R}_{X}$ of
$X$ and $\text{length}(\gamma)<d(x_{1},x_{2})+\epsilon$, then
$\phi_{s}\circ\gamma$ is a curve in $\mathcal{R}_{X}$ from $\phi_{s}(x_{1})$
to $\phi_{s}(x_{2})$ with $\text{length}(\gamma)$, so $d(\phi_{s}(x_{1}),\phi_{s}(x_{2}))<d(x_{1},x_{2})+\epsilon$;
taking $\epsilon\searrow0$, and replacing $x_{1},x_{2}$ with $\phi_{-s}(x_{1}),\phi_{-s}(x_{2})$
implies that $\phi_{s}:(\mathcal{R}_X,d)\to(\mathcal{R}_{X},d)$
is an isometry for all $s\in\mathbb{R}$. We can therefore extend
to a unique isometry $\phi_{s}:(X,d)\to(X,d)$, whose image is closed
and contains $\mathcal{R}_{X}$, hence is bijective. 

\end{proof}

The proof strategy for the following proposition is roughly similar to that of Theorem 2 in \cite{gangliucone}. 

\begin{prop} \label{locallyfree} Let $\mathcal{X}$ be as in Proposition 6.1, and assume $Rc(g_X)=0$, so that $X=C(Y)$ is a metric cone with vertex $\{ o\}$. Then the 1-parameter group of isometries $(\phi_{s})_{s\in \mathbb{R}}$ of $C(Y)$ acts locally freely on the link $Y$.
\end{prop}

\begin{rem}
The rough idea to assume by way of contradiction that a point $z\in C(Y)\setminus \{o\}$ is fixed by the action $(\phi_{s})$, so that $\phi_s$ preserves the distance to $z$. Let $q_i$ be a sequence of almost-radial functions based at $(x_i,0)$, and let $(z_i,-1) \in M_i\times [\epsilon_i^{-1},0]$ converge to $(z,-1)$. At sufficiently small scales near $(z_i,-1)$, appropriate rescalings of $q_i$ look like almost-splitting functions, so Proposition \ref{meat} gives almost-splitting functions $y_i$ with $\nabla y_i \approx J\nabla q_i$. By a diagonal argument, after parabolic rescaling of flows, we get convergence of $y_i$ to a function $y_{\infty}$ on the tangent cone $C(Z)$ at $z$ which induces a metric splitting. On the other hand, $\nabla y_i \approx J\nabla q_i$ implies that the flow of $\nabla y_{\infty}$ preserves the distance to the vertex of $C(Z)$, a contradiction.
\end{rem}

\begin{proof}
Fix a correspondence $\mathfrak{C}$ realized the $\mathbb{F}$-convergence to $\mathcal{X}$. It suffices to show that there is no point $z\in\partial B(o,1)$
satisfying $\phi_{s}(z)=z$ for all $s\in\mathbb{R}$. Suppose by
way of contradiction such a point exists. For any $x\in\mathcal{R}_{C(Y)}$,
we then have $d(\phi_{s}(x),z)=d(x,z)$ for all $s\in\mathbb{R}$.
Choose a sequence $z_{i}\in M_{i}$ such that 
\[
(z_{i},-1)\xrightarrow[i\to\infty]{\mathfrak{C}}(z,-1)\in C(Y)\times(-\infty,0)=\mathcal{X}_{<0}.
\]
By Proposition \ref{stronger}, there is a sequence $\delta_{i}\searrow0$ such that if $W_i :=\mathcal{N}_{x_i,0}(1)$, then $q_i :=4\tau(h_i-W_i)$ are strong $(\delta_i,1)$-conical functions based at $(x_i,0)$ which satisfy
\begin{equation} \label{superq}
\int_{-\delta_{i}^{-1}}^{-\delta_{i}}\int_{M_{i}}\left(\left|\nabla^{2}q_{i}-2g_{i}\right|^{2}+\left||\nabla q_{i}|^{2}-4q_{i}\right|\right)e^{\alpha f_{i}}d\nu_{x_{i},0;t}^{i}dt\leq\delta_{i}
\end{equation}
for some $\alpha>0$, where we have written $\nu_{x_{i},0;t}^{i}=(4\pi\tau)^{-\frac{n}{2}}e^{-f_{i}}dg_{i,t}$. By the proof of Theorem 15.80 in \cite{bamlergen3}, we can therefore pass to a subsequence so that $q_i \circ \psi_i \to q_{\infty}:=d^2(\cdot,o)$ in $C_{loc}^{\infty}(\mathcal{R}_{<0})$ as $i \to \infty$, where $\psi_i$ are as in Theorem \ref{bamconvergence}. This implies
\[
\liminf_{i\to\infty}\frac{1}{\tau^{\ast}}\int_{-1-2\tau^{\ast}}^{-1-\tau^{\ast}}\int_{M_i}(q_{i})_{+}d\nu_{z_{i},-1;t}^idt\geq\frac{1}{\tau^{\ast}}\int_{-1-2\tau^{\ast}}^{-1-\tau^{\ast}}\int_{\mathcal{R}_{C(Y)}}q_{\infty} d\nu_{z,-1;t}dt.
\]

\noindent \textbf{Claim:} $\lim_{t\nearrow -1} \int_{\mathcal{R}_t} q_{\infty} d\nu_{z,-1;t}=q_{\infty}(z)=1$.\\

Now choose sequences $t_j \nearrow -1$, $y_j \in X$ such that $(y_j,t_j)$ are $H_n$-centers of $(z,-1)$. Because $\min(q_{\infty},4)$ is 2-Lipschitz, then have
$$\left| \int_{\mathcal{R}_t} \min(q_{\infty},4) d\nu_{z,-1;t_j}-\min(q_{\infty}(y_j),4) \right| \leq 2 \sqrt{H_n (1+t_j)} \to 0$$
as $j\to \infty$. However, Claim 22.9(d) of \cite{bamlergen3} implies that the natural topology agrees on $\mathcal{X}$ agrees with the product topology on $C(Y)\times (-\infty,0)$; because $(y_j,t_j) \to (z,-1)$ in the natural topology, we have $y_j \to z$ in $C(Y)$, hence
$$\lim_{j\to \infty}q_{\infty}(y_j)=q_{\infty}(z)=1.\hspace{20 mm} \square$$

We can therefore find $\gamma>0$ such that, for any $\tau^{\ast}\in (0,1)$, we have
$$\frac{1}{\tau^{\ast}} \int_{-1-2\tau^{\ast}}^{-1-\tau^{\ast}}\int_{M_i} (q_i)_+ d\nu_{z_i,-1;t}^idt \geq \gamma^2$$
for $i=i(\tau^{\ast},\gamma)\in \mathbb{N}$ sufficiently large. Because $(z_i,-1)\xrightarrow[i\to \infty]{\mathfrak{C}}(z,-1)$, there exists $A<\infty$ such that $(z_i,-1)\in P^{\ast}(x_{\infty};A,-A^2)$ for all $i\in \mathbb{N}$. We can therefore use Bamler's conjugate heat kernel comparison theorem (Proposition 8.1 in \cite{bamlergen3}) and (\ref{superq}) to obtain
$$\int_{-1-(\delta_i')^{-1}}^{-1-\delta_i'} \int_{M_i} \left( |\nabla^2 q_i - 2g_i|^2 +\left| |\nabla q_i|^2-4q_i \right| \right) d\nu_{z_i,-1;t}^idt \leq \delta_i'$$
for some sequence $\delta_i'\searrow 0$. We may then proceed as in the proof of Proposition 13.19 of \cite{bamlergen3} to conclude that, for any $\epsilon>0$, 
$$\frac{1}{2\sqrt{a_i}}q_i:M_i \times [-1-(\gamma \beta)^2 \epsilon^{-1},-1-(\gamma\beta)^2\epsilon]\to \mathbb{R}$$ are $(1,\epsilon,\gamma\beta)$-splitting maps for $\beta \leq \overline{\beta}(\epsilon)$, where
$$a_i:= \int_{M_i} q_i d\nu_{z_i,-1;t_1 -(\beta \gamma)^2}^i$$
satisfies 
$$\frac{1}{2}\gamma^2 \leq a_i \leq C\int_{M_i} q_i e^{\alpha f_i}d\nu_t^i \leq C\left(\int_{M_i} q_i^2 d\nu_t^i \right)^{\frac{1}{2}} \left( \int_{M_i} e^{2\alpha f_i} d\nu_t^i \right)^{\frac{1}{2}} \leq C(Y,\epsilon,\beta,\gamma).$$
In fact, the lower bound follows from the estimate for $|\square q_i|$, (c.f. the proof of Proposition 12.1 of \cite{bamlergen3}), while the upper bound follows from the $L^2$-Poincare inequality and and property $(iv)$ of strong almost-radial functions.

We now apply Proposition 12.1 of \cite{bamlergen3}
to obtain strong $(1,\epsilon'',\gamma \beta)$-splitting maps $y_i ''$ with
$$\beta^{-2} \int_{-1-(\beta \gamma)^2 \epsilon''^{-1}}^{-1-(\beta \gamma)^2 \epsilon''} \int_{M_i} \left|\nabla \left( \frac{q_i}{2\sqrt{a_i}} -y_i'' \right) \right|^2 d\nu_{z_i,-1;t}^idt \leq \Psi(\beta | Y,\epsilon'')$$
for sufficiently large $i\in \mathbb{N}$. Next, apply Proposition \ref{meat} to obtain a weak $(1,\epsilon',\gamma \beta)$-splitting map $y_i '$ satisfying
$$\beta^{-2} \int_{-1-(\beta \gamma)^2 \epsilon'^{-1}}^{-1-(\beta \gamma)^2 \epsilon'} \int_{M_i} \left| \frac{J\nabla q_i}{2\sqrt{a_i}} - \nabla y_i '\right|^2 d\nu_{z_i,-1;t}^idt \leq \Psi(\beta | Y,\epsilon ')$$
for sufficiently large $i\in \mathbb{N}$, assuming $\epsilon'' \leq \overline{\epsilon}'(\epsilon '')$. Another application of Proposition 12.1 of \cite{bamlergen3} yields strong $(1,\epsilon,\gamma \beta)$-splitting maps $y_i^{\beta} : M_i \times [-1-\beta^2 \epsilon^{-1}, -1-\beta^2 \epsilon]\to \mathbb{R}$ satisfying
$$\beta^{-2} \int_{-1-(\beta \gamma)^2 \epsilon^{-1}}^{-1-(\beta \gamma)^2 \epsilon} \int_{M_i} \left| \frac{J\nabla q_i}{2\sqrt{a_i}} -\nabla y_i^{\beta}\right|^2 d\nu_{z_i,-1;t}^idt\leq \Psi(\beta|Y,\epsilon)$$
for large $i=i(\beta)\in \mathbb{N}$, assuming $\epsilon '\leq \overline{\epsilon} '(\epsilon)$. We also pass to a subsequence so that $a_i \to a \in (0,\infty)$. Then 
$$(\psi_i^{-1})_{\ast}\left( J_i\frac{\nabla q_i}{2\sqrt{a_i}} \right) \to V$$
in $C_{loc}^{\infty}(\mathcal{R}_{C(Y)})$, where $V:=\frac{1}{2\sqrt{a}}J\nabla q_{\infty}$.
\\

Using Theorem \ref{tangentflows}, choose a sequence $\beta_j \searrow 0$ such that we have $\mathbb{F}$-convergence of the corresponding parabolic rescalings to a tangent flow of $\mathcal{X}$ based at $(z,-1)$:
$$\left( \mathcal{X}^{-1,\beta_j^{-1}},(\nu_{(z,0);t}^{-1,\beta_j^{-1}})_{t\in [-2,0]} \right) \xrightarrow[i\to \infty]{\mathbb{F}} \left( \mathcal{Y},(\nu_{y_{\infty};t})_{t\in [-2,0]}\right),$$
where $\mathcal{Y}$ is a static metric flow modeled on a a Ricci flat cone $C(Z)$. By Theorem 2.16 of \cite{bamlergen3}, there is a precompact exhaustion $(W_{j})$ of $\mathcal{R}_{C(Z)}$
along with diffeomorphisms $\eta_{j}:W_{j}\to \mathcal{R}_{C(Y)}$
such that $\eta_{j}^{\ast}(\beta_{j}^{-2}g_{C(Y)})\to g_{C(Z)}$ in $C_{loc}^{\infty}(\mathcal{R}_{C(Z)})$, and so that for any $\epsilon>0$ and $D<\infty$, 
$$\eta_{j}:\left( W_{j}\cap B(o_{Z},D), d_{C(Z)} \right) \to \left(C(Y),\beta_j^{-1}d_{C(Y)}\right)$$
are $\epsilon$-Gromov-Hausdorff maps to $B(z,\beta_j D)$ for sufficiently
large $j=j(\epsilon,D)\in\mathbb{N}$. Define $\widetilde{g}_j:=\beta_j^{-2}g_{C(Y)}$ and $V_j := \beta_j V$, so that $|V_j|_{\widetilde{g}_j}\leq 10$ on $\mathcal{R}_{C(Y)}\cap B(o,10)$; by $\Delta V_j=0$ and elliptic regularity, we can pass to a subsequence so that $(\eta_j^{-1})_{\ast}V_j \to V_{\infty}$ in $C_{loc}^{\infty}(\mathcal{R}_{C(Z)})$. Moreover, if we define $r_j := \beta_j^{-1}d(z,\cdot)$, then $r_j \circ \eta_j \to r_{\infty} :=d(o_Z,\cdot)$ locally uniformly on $\mathcal{R}_{C(Z)}$; this follows from the Gromov-Hausdorff convergence $(X,\beta_j^{-1}d,z)\to (C(Z),d_{C(Z)},o_Z)$ (Theorem 2.16 of \cite{bamlergen3}). Because $r_j$ are Lipschitz, we know $V_j r_j$ is well-defined almost everywhere on $\mathcal{R}_{C(Y)}$. Because $r_j \circ \phi_s =r_j$, we may then conclude $V_j r_j=0$ almost everywhere. Given $\chi \in C_c^{\infty}(\mathcal{R}_{C(Z)})$, we have
\begin{align*} \int_{\mathcal{R}_{C(Z)}} \phi V_{\infty}r_{\infty} dg_{C(Z)}&= \int_{\mathcal{R}_{C(Z)}} r_{\infty}\text{div}(\phi V_{\infty}) dg_{C(Z)} =\lim_{j\to \infty} \int_{\mathcal{R}_{C(Y)}} r_j \text{div}\left( (\phi \circ \eta_j^{-1})V_j \right) dg_{C(Y)} \\
&= \lim_{j \to \infty} \int_{\mathcal{R}_{C(Y)}} (\phi \circ \eta_j^{-1})V_j r_j dg_{C(Y)}=0
\end{align*}
since $r_j \circ \eta_j \to r_{\infty}$ uniformly and $(\eta_j^{-1})_{\ast}V_j \to V_{\infty}$ in $C_{loc}^{\infty}$.  Thus $V_{\infty}r_{\infty}=0$ almost everywhere in $\mathcal{R}_{C(Z)}$, so the flow of $V_{\infty}$ preserves $r_{\infty}$.

By Bamler's change of basepoint theorem (Theorem 6.40 in \cite{bamlergen2}), we have
$$(M_i,(g_{i,t})_{t\in[-\delta_i^{-1},-1]},(\nu_{z_i,-1;t})_{t\in[-\delta_i^{-1},-1]})\xrightarrow[i\to\infty]{\mathbb{F},\mathfrak{C}}(\mathcal{X},(\nu_{z,-1;t})_{t\in(-\infty,-1]})$$
on compact time intervals. For each $j\in \mathbb{N}$, we can therefore choose $i(j)\in \mathbb{N}$ such that $\eta_j(W_j)\subseteq U_{i(j)}$,
$$||\psi_{i(j)}^{\ast} g_{i(j)}-g_{C(Y)}||_{C^j(U_{i(j)} \times [-2,-1-\beta_j^{4}],g_{C(Y)})}+||\psi_{i(j)}^{\ast} q_{i(j)}-q_{\infty}||_{C^j(U_{i(j)}\times [-2,-1-\beta_j^{4}],g_{C(Y)})}\leq \beta_j^{4},$$
$$d_{\mathbb{F}}\left( (M_{i(j)},(g_{i(j),t})_{t\in[-\beta_j^{-4},0]},(\nu_{z_{i(j)},-1;t})_{t\in[-\beta_j^{-4},-1]}),(\mathcal{X}_{[\beta_j^{-4},-1]},(\nu_{z,-1;t})_{t\in[-\beta_j^{-4},-1]}) \right) < \beta_j^4,$$
$$\int_{-1-\epsilon_j^{-1} \beta_j^2}^{-1-\epsilon_j \beta_j^2} \int_{M_{i(j)}} \left| \frac{J\nabla q_{i(j)}}{2\sqrt{a_{i(j)}}} -\nabla y_{i(j)}^{\beta_j}\right|^2 d\nu_{z_{i(j)},-1;t}dt\leq \epsilon_j \beta_j^2$$
for some sequence $\epsilon_j \searrow 0$, where we now view $\eta_j$ as maps $W_j \times [-2,-1]\to \mathcal{R}_{C(Y)}\times [-2,-1]$ which are constant in time. 
. Define parabolic rescalings $\widehat{g}_{j,t} := \beta_j^{-2}g_{i(j),-1+\beta_j^{2}t}$, $\widehat{\nu}_t^j:=\nu_{z_{i(j)},-1;-1+\beta_j^{2}t}$, $(1,\epsilon_j,1)$-splitting maps $$\widehat{y}_j(\cdot,t):=\beta_j^{-1} y_{i(j)}^{\beta_j}(\cdot,-1+\beta_j^{2}t)$$
based at $(z_{i(j)},0)$ in the rescaled flow, $\widehat{a}_j:=\beta_j^{-2}a_{i(j)}$, and
$$\widehat{q}_j(\cdot,t):=\beta_j^{-2}q_{i(j)}(\cdot,-1+\beta_j^2 t),$$
Then
$$(M_{i(j)},(\widehat{g}_{j,t})_{t\in [-2,0]},(\widehat{\nu}_t^j)_{t\in[-2,0]})\xrightarrow[i\to\infty]{\mathbb{F}}(\mathcal{Y},(\nu_{o_Z,0;t})_{t\in [-2,0]}),$$
and $\psi_{i(j)}\circ \eta_j$ realizes smooth convergence on $\mathcal{R}_{C(Z)}$; for example, $(\psi_{i(j)}\circ \eta_j)^{\ast}\widehat{g}_j \to g_{C(Z)}$ in $C_{loc}^{\infty}(\mathcal{R}_{C(Z)})$.
so the proof of Theorem 15.50 in \cite{bamlergen3} shows that $\mathcal{Y}=\mathcal{Y}'\times \mathbb{R}$ splits as a metric flow, and $(\psi_{i(j)}\circ \eta_j)^{\ast}\widehat{y}_j \to y_{\infty}$, where $y_{\infty}: \mathcal{Y}\to \mathbb{R}$ denotes the projection onto the $\mathbb{R}$-factor. On the other hand, our assumptions guarantee that  
$$\left\Vert \frac{1}{2\sqrt{\widehat{a}_j}}(\psi_{i(j)}^{-1})_{\ast}\nabla^{\widehat{g}_j}\widehat{q}_j(\cdot,t)-V_j\right\Vert_{C^{j-1}(U_{i(j)},\widetilde{g}_j)}\leq C\beta_j^{2},$$
which implies
$$\frac{1}{2\sqrt{\widehat{a}_j}}((\psi_{i(j)} \circ \eta_j)^{-1})_{\ast} \nabla^{\widehat{g}_j}\widehat{q}_j(\cdot,t)\to V_{\infty}$$
in $C_{loc}^{\infty}(\mathcal{R}_{C(Z)})$. Here, we again view $\eta_j$ as a map $W_j \to U_{i(j)}$ which is constant in time.  Let $\widehat{K}_j$ denote the heat kernel of the rescaled flows. For any compact subset $K\subseteq \mathcal{R}_{C(Z)}$, we then have
\begin{align*} \int_{-2}^{-1}\int_K |V_{\infty}-\nabla y_{\infty}|^2 d\nu_{o_Z,-1;t}dt \hspace{-40 mm}& \\
&= \lim_{j\to \infty} \int_{-2}^{-1}\int_K \left| \left( (\psi_{i(j)}\circ \eta_j)^{-1} \right)_{\ast} \left( \frac{\nabla^{\widehat{g}_j}\widehat{q}_j}{2\sqrt{\widehat{a}_j}} - \nabla^{\widehat{g}_j}\widehat{y}_j \right) \right|^2 (\psi_{i(j)}\circ \eta_j)^{\ast}\widehat{K}_j(z_{i(j)},0;\cdot,t)d\left( (\psi_{i(j)}\circ \eta_j)^{\ast} \widehat{g}_j \right) dt \\ 
&= \lim_{j \to \infty} \int_{-2}^{-1} \int_{(\psi_{i(j)}\circ \eta_j)(K)} \left| \frac{\nabla^{\widehat{g}_j}\widehat{q}_j}{2\sqrt{\widehat{a}_j}} -\nabla^{\widehat{g}_j} \widehat{y}_j \right|^2 d\widehat{\nu}_t^j dt \\ 
&\leq \liminf_{j\to \infty} \beta_j^{-2} \int_{-1-2\beta_j^2}^{-1-\beta_j^2} \int_{M_{i(j)}} \left| \frac{\nabla^{g_{i(j)}}q_j}{2\sqrt{a_{i(j)}}} - \nabla^{g_{i(j)}}y_{i(j)} \right|^2 d\nu_{z_{i(j)},-1;t}dt=0,
\end{align*}
where we used that 
$$\left( (\psi_{i(j)}\circ \eta_j)^{-1} \right)_{\ast} \nabla^{\widehat{g}_j}\widehat{y}_j = \nabla^{(\psi_{i(j)}\circ \eta_j)^{\ast} \widehat{g}_j} (\psi_{i(j)} \circ \eta_j)^{\ast}\widehat{y}_j \to \nabla^{g_{C(Z)}} y_{\infty}$$
in $C_{loc}^{\infty}(\mathcal{R}_{C(Z)})$. Thus $V_{\infty}=\nabla y_{\infty}$. However, $\nabla y_{\infty}$ is a complete vector field on $\mathcal{R}_{C(Z)}$ which leaves any compact set in finite time, whereas the flow of $V_{\infty}$ preserves any geodesic ball centered at $o_Z$, a contradiction.
\end{proof}

\begin{proof}[Proof of Theorem \ref{mainthm3}] This is a consequence of Propositions \ref{action} and \ref{locallyfree}.
\end{proof}

\printbibliography
\vspace{ 6 mm}
\noindent M. Hallgren, Department of Mathematics, Cornell University, Ithaca NY 14850\\
\textit{Email:} meh249@cornell.edu\\

\noindent W. Jian, Institute of Mathematics, Academy of Mathematics and Systems Science, Chinese Academy of Sciences, Beijing, 100190, China.\\
\textit{Email:} wangjian@amss.ac.cn\\

\end{document}